\documentclass[a4paper]{scrartcl}
\usepackage[left=2cm,right=2cm]{geometry}
\usepackage{amsthm}
\usepackage[dvipsnames]{xcolor}
\usepackage{graphicx}
\usepackage{epstopdf}
\usepackage{tensor}
\usepackage{colonequals}
\usepackage{diagbox}
\usepackage{subcaption}
\usepackage{verbatim}
\usepackage{mathtools}
\usepackage[authoryear,round]{natbib}
\usepackage{appendix}
\usepackage{aliascnt}
\usepackage{hyperref}
\usepackage[capitalise,nameinlink,noabbrev]{cleveref}
\hypersetup{
    colorlinks,
    allcolors={green!50!black},
    urlcolor={red!50!black}
}

\newcommand{\figurealt}[1]{}

\theoremstyle{plain}
\newtheorem{theorem}{Theorem}[section]
\newaliascnt{lemma}{theorem}
\newtheorem{lemma}[lemma]{Lemma}
\aliascntresetthe{lemma}
\newaliascnt{corollary}{theorem}
\newtheorem{corollary}[corollary]{Corollary}
\aliascntresetthe{corollary}
\newaliascnt{remark}{theorem}
\newtheorem{remark}[remark]{Remark}
\aliascntresetthe{remark}
\newaliascnt{example}{theorem}
\aliascntresetthe{example}

\theoremstyle{plain}

\newtheorem{assumption}{Assumption}
\newaliascnt{definition}{theorem}
\newtheorem{definition}[definition]{Definition}
\aliascntresetthe{definition}

\ifpdf
  \DeclareGraphicsExtensions{.eps,.pdf,.png,.jpg}
\else
  \DeclareGraphicsExtensions{.eps}
\fi

\newcounter{subeq}

\usepackage{amsmath,amsfonts,amssymb,amsopn}
\usepackage{bm}
\usepackage{stmaryrd}
\usepackage{xargs}
\usepackage{xifthen}
\usepackage{todonotes}

\def\R{\mathbb{R}}

\DeclareMathOperator{\diam}{diam}
\DeclareMathOperator{\Id}{Id}
\DeclareMathOperator{\tr}{tr}

\DeclareMathSymbol{\shortminus}{\mathbin}{AMSa}{"39}

\newcommand{\OO}{\mathcal{O}}

\newcommand{\order}{{k}}  
\newcommand{\kg}{{m}} 
\newcommand{\ku}{{r}} 
\newcommand{\kd}{{b}} 

\newcommand{\M}{\Omega}

\newcommand{\G}{\Gamma}
\newcommand{\Gh}{\Gamma^\kg_{h}}
\newcommand{\GBar}{\overline{\G}}

\newcommand{\GM}{\G_{\M}}
\newcommand{\GhM}{\G^\kg_{h,\M}}

\newcommand{\Tri}{\mathcal{T}}
\newcommand{\TriBar}{\overline{\Tri}}

\newcommand{\TriRef}{\TriBar_{h_0}}

\newcommand{\T}{T}

\newcommand{\TRef}{T_\mathrm{ref}}
\newcommand{\TriM}{\TriBar}

\newcommand{\symT}{\mathbb{T}}

\newcommand{\dist}{\rho_\G}

\newcommand{\nB}{\bm{n}}
\newcommand{\nn}{\nB}
\newcommand{\nnh}{\nB^\kg_{h}}
\newcommand{\nM}{\bar{\nB}}
\newcommand{\nhM}{\bar{\nB}_{h}}

\newcommand{\eB}{\bm{e}}
\newcommand{\eM}{\bar{\eB}}
\newcommand{\tauB}{\bm{\tau}}
\newcommand{\tauM}{\bar{\tauB}}

\newcommand{\dG}{\mathrm{d}\sigma}
\newcommand{\dGh}{\mathrm{d}\sigma_h}
\newcommand{\dM}{\mathrm{d}\bar{\sigma}}
\newcommand{\dt}{\mathrm{d}t}
\newcommand{\mG}{\mu}
\newcommand{\mGh}{\mu^\kg_{h}}

\newcommand{\PhiM}{\Phi_{\M}}
\newcommand{\Phih}{\Phi^\kg_{h}}
\newcommand{\PhihM}{\Phi^\kg_{h,\M}}

\newcommand{\xG}{x}
\newcommand{\xGh}{\hat{x}}

\newcommand{\xM}{\bar{x}}
\newcommand{\xBar}{\bar{x}}
\newcommand{\xRef}{\xi}

\newcommand{\phiBar}{\bar{\phi}}

\newcommand{\PP}{\bm{P}}

\newcommand{\PPh}{\PP_{\!h}}

\newcommand{\Weingarten}{\bm{H}}
\newcommand{\Weingartenh}{\Weingarten_h}
\newcommand{\MeanCurvature}{H}
\newcommand{\MeanCurvatureh}{\MeanCurvature_h}
\newcommand{\GaussCurvature}{K}
\newcommand{\GaussCurvatureh}{\GaussCurvature_h}

\newcommand{\D}{D}

\newcommand{\DM}{\bar{\D}}

\newcommand{\grad}{\nabla}
\newcommand{\gradG}{\grad_{\G}}
\newcommand{\gradGh}{\grad_{\Gh}}
\newcommand{\gradM}{\overline{\grad}}

\newcommand{\Inner}[3]{\big({#1} \,,\, {#2}\big)_{#3}}

\newcommand{\Norm}[2]{\lVert{#1}\rVert_{#2}}

\newcommand{\EnergyNorm}[2]{{\lvert\kern-0.25ex\lvert\kern-0.25ex\lvert{#1}\rvert\kern-0.25ex\rvert\kern-0.25ex\rvert}_{#2}}

\newcommand{\restr}[2]{\left.{#1}\right|_{#2}}

\newcommandx{\SobolevSpace}[3][1=s,2=p]%
  {\bm{W}^{#1,#2}({#3})}
\newcommandx{\SobolevSpaceTan}[3][1=s,2=p]%
  {\bm{W}_\mathrm{tan}^{#1,#2}({#3})}
\newcommandx{\SobolevSpaceAmb}[3][1=s,2=p]%
  {\bm{W}^{#1,#2}({#3})}
\newcommandx{\SobolevNormTan}[4][1=s,2=p]%
  {\Norm{#3}{\SobolevSpaceTan[#1][#2]{#4}}}
\newcommandx{\SobolevNormAmb}[4][1=s,2=p]%
  {\Norm{#3}{\SobolevSpaceAmb[#1][#2]{#4}}}
\newcommandx{\SobolevNorm}[4][1=s,2=p]%
  {\Norm{#3}{\SobolevSpace[#1][#2]{#4}}}

\newcommandx{\HSpace}[2][1={1}]%
  {\bm{H}^{#1}({#2})}
\newcommandx{\HSpaceTan}[2][1={1}]%
  {\bm{H}_\mathrm{tan}^{#1}({#2})}
\newcommandx{\HSpaceAmb}[3][1={1},2={3}]%
  {[H^{#1}({#3})]^{#2}}
\newcommandx{\HNormTan}[3][1={1}]%
  {\Norm{#2}{\HSpaceTan[#1]{#3}}}
\newcommandx{\HNormAmb}[3][1={1}]%
  {\Norm{#2}{\HSpaceAmb[#1]{#3}}}
\newcommandx{\HNorm}[3][1={1}]%
  {\Norm{#2}{\HSpace[#1]{#3}}}

\newcommandx{\LSpace}[3][1={2},2={}]%
  {\bm{L}^{#1}_{#2}(#3)}
\newcommandx{\LSpaceTan}[2][1={2}]%
  {\bm{L}_\mathrm{tan}^{#1}(#2)}
\newcommandx{\LSpaceAmb}[3][1={2},2={3}]%
  {[L^{#1}(#3)]^{#2}}
\newcommandx{\LNormTan}[3][1={2}]%
  {\Norm{#2}{\LSpaceTan[#1]{#3}}}
\newcommandx{\LNormAmb}[3][1={2}]%
  {\Norm{#2}{\LSpaceAmb[#1]{#3}}}
\newcommandx{\LNorm}[3][1={2}]%
  {\Norm{#2}{\LSpace[#1]{#3}}}
\newcommandx{\LTwoNorm}[3][3={}]%
  {\Norm{#1}{#2}}

\newcommandx{\LSpaceAvg}[3][1={2},2={}]%
{L^{#1}_{0{#2}}(#3)}

\newcommandx{\Vh}[1][1=h]{\bm{V}_{\!\!{#1}}}

\newcommand{\NormZero}[2][]{\Norm{#2}{0,h\ifthenelse{\isempty{#1}}{}{,#1}}}
\newcommand{\NormOne}[2][]{\Norm{#2}{1,h\ifthenelse{\isempty{#1}}{}{,#1}}}

\newcommand{\GNorm}[2][]{\Norm{#2}{\star,h\ifthenelse{\isempty{#1}}{}{,#1}}}
\newcommand{\ahNorm}[2][]{\HNormTan[1]{\PPh{#2}}{\ifthenelse{\isempty{#1}}{\Gh}{#1}}}
\newcommand{\AhNorm}[2][]{\EnergyNorm{#2}{A_h\ifthenelse{\isempty{#1}}{}{,#1}}}

\newcommand{\I}{\mathrm{I}}
\newcommand{\IBar}{\overline{\I}}

\newcommand{\IRef}{\I_\mathrm{ref}}

\newcommand{\II}{\mathbb{I}}

\newcommandx{\PDQ}[4][1={\Weingarten},2={\PP},3={n}]{\sum_{r=1}^{k+l}{#1}\bigotimes_{r,l+1}{#2}\big[{#4}\underset{r}{\cdot}{#3}\big]}

\newcommand{\myTitle}{Odd behaviour of even geometries}
\newcommand{\mySubtitle}{An explanation for superconvergent geometric consistency errors}

\newcommand{\myAbstract}{%
Piecewise polynomial surface approximations used in surface finite element methods often seem to behave better than their standard approximation properties suggest if their polynomial order is even. We explain this superconvergence through cancellation of leading interpolation errors on suitably structured meshes that naturally arise in some refinement processes. This cancellation improves weighted integral estimates for functions, derivatives, and geometric quantities. Applications include estimates for surface normals, the Weingarten map, and Gaussian curvature. Numerical experiments reproduce the predicted parity-dependent behaviour and support the proposed explanation of superconvergent geometric consistency errors, while the corresponding pointwise errors retain their standard orders.%
}
\newcommand{\myKeywords}{%
surface finite element methods; higher order surface approximation; superconvergence; geometric consistency errors; symmetric triangulations; curvature.%
}
\numberwithin{equation}{section}
\begin{document}

\title{\myTitle}
\subtitle{\mySubtitle}
\date{}

\author{%
Hanne Hardering%
  \thanks{Technische Universit{\"a}t Dresden, %
  Institute of Numerical Mathematics, %
  01062 Dresden, Germany %
  (\url{hanne.hardering@tu-dresden.de}).}
~~and~~
Simon Praetorius%
  \thanks{Technische Universit{\"a}t Dresden, %
  Institute of Scientific Computing, %
  01062 Dresden, Germany %
  (\url{simon.praetorius@tu-dresden.de}). Corresponding author.}
~~and~~
Gentian Zavalani%
  \thanks{Technische Universit{\"a}t Dresden, %
  Institute of Numerical Mathematics, %
  01062 Dresden, Germany %
  (\url{gentian.zavalani@tu-dresden.de}).}%
}

\maketitle

\paragraph*{Abstract.}
\myAbstract


\section{Introduction}

Higher order surface finite element methods approximate a smooth surface by a
piecewise polynomial parametrization and then solve the discrete problem on this
approximate geometry.  The standard geometric estimates for such
parametrizations are well understood and usually predict $h^{\kg+1}$ for the
position error, $h^\kg$ for the error in the normal vector, and then one order less for each
derivative, with $\kg$ the polynomial order of the parametrization \citep[see, for example,][]{Demlow2009Higher,Kovacs2017High}.
In computations, however, this is not always the whole story.  For several
surface finite element discretizations one observes convergence orders that are
better than the standard analysis suggests, especially when even-order geometry
approximations or even polynomial degrees are used
\citep[see, for example,][]{BonitoEtAl2018Eigenvalues,HP2024Parametric}.  Related superconvergence
phenomena on special meshes are also known in other finite element settings
\citep{BankXu2003AsymptoticallyExact}.

The aim of this paper is to explain a class of these improvements.  The
essential point is not only the polynomial degree, but also the local structure
of the refined triangulation.  On symmetric pairs of elements, i.e., grid elements
that are mirror images of each other, the leading contributions in certain interpolation
and geometry errors cancel.  This produces the characteristic odd-even pattern: even
approximation orders gain an extra order in weighted integral quantities, while
odd orders do not.  Results of this type were proved in the context of surface
collocation and surface integration errors in
\citet{Chien1993Piecewise,AtkinsonChien1995Piecewise,Chien1995Numerical,zavalani2024note}.
We extend this idea to interpolation errors for functions and their derivatives, and
then use these estimates to improve bounds for geometric quantities such as the normal
vector, the Weingarten map, and curvature.

Our analysis is formulated for grid sequences obtained from an initial
macro-triangulation by structured refinement.  The main examples are red
refinement and related refinement patterns that preserve a symmetric
pair-structure on each macro element.  The initial grid itself is only required
to satisfy the usual finite element assumptions, such as shape regularity.  The additional
symmetry enters on the refined macro-elements and is used only for the improved
estimates.  This setup is very common in convergence studies of surface finite element
discretization.  The numerical experiments also indicate that some of these
structural requirements may be relaxed, for instance for certain
newest-vertex-bisection refinements \citep{Rivara1984Mesh}, but the proofs in
this paper use the explicit symmetric structure.

A second point of emphasis is the choice of lifting.  We primarily work with
the lifting induced by the surface parametrization, in the spirit of
\citet{Nedelec1976Curved}, and distinguish it from the closest-point lifting
used, for example, by \citet{Hei2004Isoparametric,Demlow2009Higher}.  This distinction matters when
derivatives and geometry-dependent quantities are compared on the smooth and
discrete surfaces.  We show that the improved interpolation estimates already
hold for the parametrization lifting, and we also record how the
closest-point projection enters some of the geometric applications.

The main technical contribution is an interpolation theory for weighted
integrals on symmetric macro-elements and its transfer to curved surfaces.  For
function values and for first derivatives with sufficiently regular weights,
the cancellation yields the same improved order.  The higher-order derivative
case is more technical and less common in finite-element computations, so the
corresponding notation, estimates, and numerical verification are collected in
the appendix.  There the improvement depends on the interplay between the
geometry order $\kg$, the function interpolation order $\ku$, and the derivative
order $r$.  This derivative-order dependence is one of the features that is not
visible in the classical surface approximation estimates.

The improved estimates for normal vectors and curvature are relevant whenever
geometric consistency terms enter a surface finite element error analysis.  This
occurs, for example, in vector- and tensor-valued surface PDEs, surface Stokes
discretizations, and problems involving curvature terms
\citep{HLL2020Analysis,HP2023Tangential,HP2024Parametric,HLZ2015Stabilized,Cenanovic2017MeanCurvature}.  The
improvement is not automatically sufficient to improve the convergence order of
every full discretization: in geometric evolution problems, for instance,
separate evolution equations for normal vectors or curvature can still be
essential for optimal schemes \citep{Kovacs2019Convergent}.  Nevertheless, the
estimates identify which geometric error terms are genuinely smaller on
structured even geometries.

The paper is organized as follows.  In \cref{sec:discrete surface} we introduce
the surface parametrizations, discrete surfaces, liftings, and mesh assumptions,
including the symmetric macro-triangulations used later.  In
\cref{sec:interpolation_errors} we prove the improved flat macro-element
estimates and lift them to curved surfaces.  In \cref{sec:application} we apply
the interpolation theory to geometric quantities.  Finally,
\cref{sec:numerical-experiments} verifies the predicted orders numerically and
documents additional observations about the role of the refinement structure.
The higher-order lifted derivative interpolation estimates and their numerical
tests are given in \cref{sec:appendix-higher-order-derivatives} and some additional
numerical tests in \cref{sec:additional-numerical-experiments} show observations
beyond the proved results.

\section{Surface and discrete surface approximation}\label{sec:discrete surface}

In this section we collect the geometric notation used throughout the paper.
Let $\kg\geq 1$ denote the polynomial order used for the geometry
approximation.
We assume that $\G$ is a connected, oriented, closed, two-dimensional
$C^{\kg+2}$ hypersurface in $\R^3$.

\subsection{Surface parametrization and discrete geometry}\label{sec:surface-parametrization}
Let $\GBar$ be a fixed two-dimensional reference surface given by a conforming triangulation $\TriBar=\TriBar_{h_0}$, called macro grid, consisting of finitely many nondegenerate planar
triangles, called macro elements.

On each macro element $\M$, we use its flat affine structure.
After choosing an orthonormal basis
$\{\eM_j\}_{j=1}^2$ of its tangent plane, we define, for a vector-valued map
$\bm{f}:\M\to\R^3$,
\[
  \DM\bm{f}(\xM) \colonequals (\partial_j f_i)_{ij}(\xM) \in\R^{3\times2}.
\]
Flat Sobolev norms on $\M$ and on elements $\T\subset\M$ are defined with
respect to these derivatives.
Since the macro grid is fixed and finite, all choices of orthonormal bases are
equivalent.
Whenever the inverse of a rectangular Jacobian is used, it is understood in
the Moore--Penrose sense and denoted by $(\DM\bm{f})^+$.

\begin{assumption}[Admissible macro parametrization]\label{ass:admissible-macro-parametrization}
For every macro element $\M\in\TriBar$, there exists a map
$\PhiM\in C^{\kg+2}(\M;\R^3)$ whose restriction to the relative interior of
$\M$ is an embedding. We define the corresponding surface patch by
$\GM\colonequals\PhiM(\M)\subset\G$.
The patch images form a conforming partition of $\G$.
Whenever two macro elements share an edge, the corresponding
parametrizations coincide on that edge, so that the local maps glue to a
continuous map $\Phi:\GBar\to\G$, $\Phi|_{\M}=\PhiM$.
We assume that $\Phi$ is a homeomorphism from $\GBar$ onto $\G$.
The parametrizations are uniformly regular: for every $\M\in\TriBar$,
we assume that $\DM\PhiM$ has full column rank and bounded smallest singular values.
\end{assumption}

We assume this parametrization throughout the paper. We refer to
\cref{ass:admissible-macro-parametrization} explicitly only when a particular
property, such as uniform regularity or compatibility across macro-element
interfaces, is used in the argument.

\begin{assumption}[Admissible triangulation]\label{ass:admissible-triangulation}
For every $h\le h_0$, let $\TriBar_h$ denote a conforming refinement of the
macro grid, and let $\GBar_h$ be the reference surface equipped with this
refined triangulation. As sets the domains coincide, i.e., $\GBar_h=\GBar$.
For every macro element $\M\in\TriBar$, we write
$\TriBar_{h,\M}\subset\TriBar_h$ for the induced triangulation of $\M$.
We assume that the refined triangulations are uniformly shape regular and
quasi-uniform with grid size $h=\max_{\T\in\TriBar_h}\diam(\T)$.
\end{assumption}

Throughout the paper, generic constants may depend on
the uniform regularity bounds of the parametrizations, and the fixed macro grid, but are independent of the refinement level~$h$.
Further assumptions on the refined triangulations will be stated separately
when they are needed.

On each element $\T\in\TriBar_{h,\M}$ let $F_{\T}\colon\TRef\to\T$ be the affine map from the reference triangle.
The degree-$\kg$ geometry approximation is the elementwise Lagrange interpolation of the smooth patch map:
\begin{equation}\label{eq:element-wise-parametrization}
  (\Phi^\kg_{\T}\circ F_{\T})(\xRef)
  \colonequals \IRef^{\kg}(\PhiM\circ F_{\T})(\xRef)
  \colonequals \sum_{i=1}^{n_{\kg}}\PhiM\big(F_{\T}(\xRef^\kg_{i})\big)\phi^\kg_{i}(\xRef),
  \qquad \xRef\in \TRef .
\end{equation}
Here $\{\xRef^\kg_i\}_{i=1}^{n_\kg}$ are the degree-$\kg$ Lagrange nodes on $\TRef$ with basis functions $\phi_i^\kg$.
We define $\PhihM|_{\T}\colonequals\Phi^\kg_{\T}$ and $\GhM \colonequals \PhihM(\M)$.
Since the exact macro maps agree on common macro edges and the same Lagrange interpolation is used from both sides, the maps $\PhihM$ glue to a continuous global piecewise polynomial map $\Phih\in C^0(\GBar,\R^3)$ with $\Phih|_{\M}=\PhihM$ and
\[
  \Gh\colonequals\Phih(\GBar)=\bigcup_{\M\in\TriBar}\GhM .
\]

Coordinates in the domain $\M$ are denoted by $\xM$, coordinates in $\G$ by $\xG$, and coordinates in $\Gh$ by $\xGh$.
If the active macro element is clear from context, we write $\Phi$ and $\Phih$ instead of $\PhiM$ and $\PhihM$.

Let $U_\delta\subset\R^3$ be a $\delta$-neighbourhood of $\G$ such that the
signed distance function $\dist$ is of class $C^{\kg+2}$ in $U_\delta$ and the
closest-point projection $\pi:U_\delta\to\G$ is unique. See
\citet[Lemma 2.8]{DE2013Finite}. It is given by
\begin{equation}\label{eq:closest-point-projection}
  \pi(x) \colonequals x - \dist(x)\nn(\pi(x)),
\end{equation}
with $\dist(x)<0$ for $x$ in the interior of $\G$ and $\nn\colonequals\D\dist$ the outer surface normal
field. The latter is extended constant in normal direction, $\nn(x)\doteq\nn(\pi(x))$.

\begin{lemma}[Regularity of the interpolated geometry]\label{lem:regularity-interpolated-geometry}
  Under \cref{ass:admissible-macro-parametrization}, for $h$ sufficiently small each $\PhihM$ is elementwise regular and is a homeomorphism from $\M$ onto $\GhM$.
  Moreover, $\DM\PhihM$ has full column rank on each element, with rank constants independent of $h$.
  The global map $\Phih\colon\GBar\to\Gh$ is a homeomorphism.
  In addition, $\pi|_{\Gh}\colon\Gh\to\G$ is a homeomorphism.
\end{lemma}
\begin{proof}
  The interpolation estimates for curved finite elements
  \citep[see, e.g.][Lemma 1]{Nedelec1976Curved} imply
  $\|\PhiM-\PhihM\|_{W^{1,\infty}(\M)}\to0$ on every macro element. Hence, the
  uniform lower bound for the singular values of $\DM\PhiM$ is inherited by
  $\DM\PhihM$ for all sufficiently small $h$, with constants independent of
  $h$. The standard perturbation argument for regular isoparametric mappings
  then shows that the element maps are regular and that the glued map
  $\Phih\colon\GBar\to\Gh$ is a homeomorphism, because $\Phi$ is a
  homeomorphism and $\Phih$ is uniformly close to $\Phi$. The same argument
  applied on a single macro element gives that
  $\PhihM\colon\M\to\GhM$ is a homeomorphism. Finally,
  $\|\Phi-\Phih\|_{L^\infty(\GBar)}\to0$ implies $\Gh\subset U_\delta$, and
  the $C^1$ closest-point projection satisfies
  $\pi\circ\Phih\to\Phi$ in the same piecewise $W^{1,\infty}$ sense. Thus,
  $\pi\circ\Phih$ is again a homeomorphism from $\GBar$ onto $\G$, and
  $\pi|_{\Gh}=(\pi\circ\Phih)\circ{\Phih}^{-1}$ is a homeomorphism.
\end{proof}

By \cref{lem:regularity-interpolated-geometry}, for $h$ sufficiently small the discrete elements are regular and $\Gh\subset U_\delta$.
We denote by $\nnh$ the elementwise normal vector on $\Gh$, choosing the orientation consistent with $\nn$.

Surface derivatives are defined by first differentiating on the flat domain
$\M$ and then lifting these derivatives to $\G$ and $\Gh$ through the
parametrizations $\Phi$ and $\Phih$, respectively.
Let $f\colon\G\to\R$ be a scalar function and denote derivatives on $\M$ by
$\gradM=(\partial_j)_{j=1}^2$.
For the pullback $\bar{f}\colonequals f\circ\Phi$, the tangential gradient
satisfies
\[
  (\gradG f)\circ\Phi = \gradM\bar{f}\cdot (\DM\Phi)^+ .
\]
The analogous formula with $\Phih$ defines $\gradGh$ elementwise on $\Gh$.
We use $\nabla$ for covariant derivatives on surfaces and $\DM$ and $\D$ for Jacobians
of mappings on $\M$ and $\R^3$, respectively.
If $f\in C^1(U_\delta,\R)$ is defined in a neighbourhood of the surface, then
$\gradG \restr{f}{\G}=\restr{(\PP\D f)}{\G}$, with
$\PP\colonequals \bm{I}-\nn\otimes\nn$.
For vector fields $\bm{v}\in C^1(U_\delta,\R^3)$ we write $\gradG\bm{v}=\PP\D\bm{v}\PP$ for the surface Jacobian and
$\gradG\cdot\bm{v}=\tr(\PP \D\bm{v})$ for the surface divergence.
The corresponding discrete definitions use
$\PPh\colonequals \bm{I}-\nnh\otimes\nnh$ elementwise.
All derivatives, normals, and conormals on $\Gh$ are understood elementwise
unless stated otherwise.
Jumps across interelement edges are allowed and are used later.

We also introduce curvature terms on the continuous and discrete surface.
The Weingarten map on $\G$ is denoted by $\Weingarten\colonequals -\gradG\nn$.
Its trace defines the mean curvature $\MeanCurvature\colonequals\tr\Weingarten=-\gradG\cdot\nn$.
For two-dimensional surfaces we also use the Gaussian curvature $\GaussCurvature \colonequals\frac{1}{2}\big((\tr\Weingarten)^2 - \tr(\Weingarten^2)\big)$. On $\Gh$ the quantities are defined elementwise in the same fashion.

In the following lemma we collect approximation results for these geometric quantities.
\begin{lemma}\label{lem:surface-approximation-estimates}
  The surface approximation $\GhM$ of $\GM$ fulfills standard geometric estimates.
  \begin{align}
    \|\dist\|_{L^\infty(\GhM)} \leq \|\Phi - \Phih\|_{L^\infty(\M)} &\leq C h^{\kg+1}, \label{eq:lem-surface-approx-estimate1} \\
    \max_{\T\in\TriBar_{h,\M}}\|\Phi - \Phih\|_{W^{l,\infty}(\T)} &\leq C h^{\kg+1-l}, \qquad 0\leq l\leq \kg+1, \label{eq:lem-surface-approx-estimate2} \\
    \|\nn\circ\pi - \nnh\|_{L^\infty(\GhM)} &\leq C h^\kg.\label{eq:lem-surface-approx-estimate3} \\
    \max_{\T\in\TriBar_{h,\M}}\|\Weingarten\circ\pi - \Weingartenh\|_{L^\infty(\GhM)} &\leq C h^{\kg-1}.\label{eq:lem-surface-approx-estimate4}
  \end{align}
\end{lemma}
\begin{proof}
  The estimates \eqref{eq:lem-surface-approx-estimate1} and \eqref{eq:lem-surface-approx-estimate2} follow from the standard interpolation estimates for the patch maps and the definition of the surface distance function $\dist$ \citep[cf.][Lemma 1]{Nedelec1976Curved}.
  The normal vector estimate \eqref{eq:lem-surface-approx-estimate3} and Weingarten map estimate \eqref{eq:lem-surface-approx-estimate4} follow by arguments similar to those in \citet[Proposition 2.3]{Demlow2009Higher}.
\end{proof}

In order to compare derivatives on $\G$ and $\Gh$ we need the following lemma.
\begin{lemma}\label{lem:pseudo-inverse-geometry-mapping}
  Let $\PhiM$ and $\PhihM$ be defined as above on a fixed macro element $\M$,
  and let $h$ be sufficiently small. Then the following elementwise estimate
  holds:
  \begin{align}
    \max_{\T\in\TriBar_{h,\M}}\|(\DM\Phi)^+ - (\DM\Phih)^+\|_{L^{\infty}(\T)} &\leq C h^{\kg}, \label{eq:lem-surface-approx-Dinv}
  \end{align}
  with constant $C$ depending on $\DM\Phi$.
\end{lemma}
\begin{proof}
  By the uniform regularity in \cref{ass:admissible-macro-parametrization}, $\sigma_{\min}(\DM\Phi)$ is bounded away from $0$ on $\M$.
  By \eqref{eq:lem-surface-approx-estimate2}, $\DM\Phih$ converges uniformly to $\DM\Phi$.
  Hence, for $h$ sufficiently small, both $\DM\Phi$ and $\DM\Phih$ have
  uniformly bounded norms and uniformly bounded inverse Gram matrices.
  We set
  \[
    A\colonequals \DM\Phi,\qquad
    B\colonequals \DM\Phih,\qquad
    G_A\colonequals A^\top A,\qquad
    G_B\colonequals B^\top B .
  \]
  Since $A^+=G_A^{-1}A^\top$ and $B^+=G_B^{-1}B^\top$, the identity
  \[
    A^+-B^+
    =G_A^{-1}(A-B)^\top
    +G_A^{-1}(G_B-G_A)G_B^{-1}B^\top
  \]
  holds pointwise.
  Moreover, $G_B-G_A=B^\top(B-A)+(B-A)^\top A$.
  The uniformly bounded factors in these two identities give
  \[
    \|(\DM\Phi)^+ - (\DM\Phih)^+\| \leq C\|\DM\Phi - \DM\Phih\|.
  \]
  Taking the $L^\infty$ norm on each element $\T\subset\M$ and using
  \cref{lem:surface-approximation-estimates} \eqref{eq:lem-surface-approx-estimate2}
  gives \eqref{eq:lem-surface-approx-Dinv}.
\end{proof}

By \cref{lem:regularity-interpolated-geometry}, the surfaces $\GM$ and $\GhM$ parametrized over $\M$ induce the area elements
$\dG = \mG\dM$ and $\dGh = \mGh\dM$, where $\mG = \mu[\Phi]$ and $\mGh=\mu[\Phih]$ with
\begin{align*}
  \mu[\bm{f}](\xM)  &\colonequals \sqrt{\det\left((\DM \bm{f}(\xM))^\top \DM \bm{f}(\xM)\right)} .
\end{align*}

\begin{lemma}\label{lem:surface-elements}
  The area elements fulfill the estimates
  \begin{align}
    \max_{\T\in\TriBar_{h,\M}}\sup_{\xM\in\T}|\mG(\xM) - \mGh(\xM)| &\leq C h^\kg, \label{eq:lem-surface-elements1} \\
    \max_{\T\in\TriBar_{h,\M}}\sup_{\xM\in\T}|\mu[\pi\circ\Phih](\xM) - \mGh(\xM)| &\leq C h^{\kg+1}. \label{eq:lem-surface-elements2}
  \end{align}
\end{lemma}
\begin{proof}
  The first estimate follows from the $W^{1,\infty}$ geometry estimate \eqref{eq:lem-surface-approx-estimate2}, \cref{lem:regularity-interpolated-geometry}, and the Lipschitz continuity of $A\mapsto\sqrt{\det(A^\top A)}$ on uniformly full-rank $3\times2$ matrices.
  By \cref{lem:regularity-interpolated-geometry} and the interpolation estimate, $\Phih$ defines a uniformly regular polynomial surface in $U_\delta$.
  Hence the projected area estimate of \citet[Lemma 3]{Nedelec1976Curved} applies and gives the second estimate.
\end{proof}

\subsection{Functions, interpolation, and liftings}\label{sec:functions-interpolation-liftings}

Sobolev norms on $\G$ and on macro patches $\GM$ are the usual surface
Sobolev norms. On individual curved exact or discrete elements, these norms
are equivalently defined by pullback, with constants uniform in $h$. For
discrete functions on $\Gh$, derivatives are evaluated elementwise on the
surface elements $\Phih(\T)$.

Interpolation of functions is also defined by pullback to the flat macro element or the reference element.

\begin{definition}\label{def:surface-function-interpolation}
  The flat Lagrange interpolation $\IBar_h^\ku\colon C^0(\GBar)\to C^0(\GBar)$ and the induced surface interpolation $\II_{h,\kg}^\ku\colon C^0(\G)\to C^0(\Gh)$ are defined elementwise by
  \begin{align*}
    \II_{h,\kg}^{\ku}(f)\circ\Phih
    \colonequals \IBar_h^\ku(f\circ\Phi)
    \colonequals \IRef^\ku(f\circ \Phi\circ F_{\T})\circ F_{\T}^{-1},
    \qquad \T\in\TriBar_{h,\M},\quad \M\in\TriBar .
  \end{align*}
  Here $\IRef^\ku$ denotes polynomial interpolation of order $\ku$ on the reference element $\TRef$.
\end{definition}

By \cref{lem:regularity-interpolated-geometry}, for $h$ sufficiently small the parametrization lifting is well defined by
\begin{align}
  L &\colonequals \Phi\circ (\Phih)^{-1}\colon\GhM\to \GM \label{eq:lifting-operator}
\end{align}
on each macro patch.
It is tied to the construction of $\Gh$ and to the interpolation operator because exact and discrete functions are pulled back to the same flat macro element.
The closest-point projection $\pi\colon\Gh\to\G$ is the standard geometric lifting in surface finite element analysis.
The same lemma states that $\pi|_{\Gh}$ is a homeomorphism onto $\G$.
Thus $\pi^{-1}$ below means the inverse of this restricted projection.
The following estimate compares the two liftings.

\begin{lemma}\label{lem:interpolation-error-estimates}
  Let $\T\in\TriBar_{h,\M}$ and set $\T_h\colonequals\Phih(\T)\subset\GhM$.
  There exists an element neighbourhood $\omega_{\T,h}\subset\GhM$ of $\T_h$ with $\operatorname{diam}(\omega_{\T,h})\leq C h$ such that $L(\omega_{\T,h})\supset \pi(\T_h)$ and $\pi(\omega_{\T,h})\supset L(\T_h)$.
  Let $f\in W^{l+1,\infty}(D)$ with $D$ either $L(\omega_{\T,h})$ or $\pi(\omega_{\T,h})$ for some $l\leq \kg$. Then
  \begin{align}
    \|f\circ L - f\circ\pi\|_{W^{l,\infty}(\T_h)} &\leq
       C h^{\kg+1-l} \|f\|_{W^{l+1,\infty}(D)}.\label{eq:standard-interpolation-1}
  \end{align}
\end{lemma}
\begin{proof}
  Pulling back to $\T$ gives
  \[
    (L-\pi)\circ\Phih = \Phi-\pi\circ\Phih .
  \]
  Since $\pi\circ\Phi=\Phi$, the smoothness of $\pi$ on $U_\delta$,
  \cref{lem:regularity-interpolated-geometry}, and
  \eqref{eq:lem-surface-approx-estimate2} imply
  $\|L-\pi\|_{W^{l,\infty}(\T_h)}\leq C h^{\kg+1-l}$ for $l\leq\kg$.
  The neighbourhood $\omega_{\T,h}$ is chosen as a uniformly bounded element
  patch around $\T_h$ that contains the images required by $L$ and $\pi$.
  The chain rule then gives
  \begin{align*}
    \|f\circ L - f\circ\pi\|_{W^{l,\infty}(\T_h)}
      &\leq C \|f\|_{W^{l+1,\infty}(D)}\|L-\pi\|_{W^{l,\infty}(\T_h)} \\
      &\leq C h^{\kg+1-l} \|f\|_{W^{l+1,\infty}(D)}.
      \tag*{\qedhere} 
  \end{align*}
\end{proof}

\subsection{Symmetric macro triangulation}\label{sec:symmetric-macro-triangulation}
The refinement procedure for obtaining $\GBar_h$ from the macro parameter
surface $\GBar=\GBar_{h_0}$ can produce additional symmetric structure.
By ``symmetric structure'' we mean that, after some initial refinement, pairs
of elements appear that are mirror images of each other and meet in a single
point.  More precisely:

\begin{definition}\label{def:symmetric-triangles}
  We call a pair $\T_1,\T_2\in\TriBar_{h,\M}$ symmetric if both elements meet
  in a single point $\xBar_0$, i.e., $\T_1\cap\T_2=\{\xBar_0\}$, and there are affine reference mappings
  $F_{\T_1}\colon\TRef\to\T_1$ and $F_{\T_2}\colon\TRef\to\T_2$ such that
  \[
    F_{\T_1}(\xRef) - \xBar_0 = \xBar_0 - F_{\T_2}(\xRef),
  \]
  with $F_{\T_1}(0)=F_{\T_2}(0)=\xBar_0$.
  See \cref{fig:red_refinement} for an illustration.
\end{definition}

We first record one refinement strategy that produces this structure.
In red refinement \citep{BankEtAl1983Refinement}, each triangle is subdivided
into four congruent subtriangles by connecting the edge midpoints.  This
refinement strategy is very common in finite element grid implementations
\citep[e.g.,][]{SKSF2017DuneFoamGrid}.

\begin{lemma}[Symmetric pair decomposition for red refinement]\label{lem:existence-of-symmetric-triangles}
  Assume that $\TriBar_{h,\M}$ is obtained from the macro triangle $\M$ by sufficiently many uniform red refinements.  Then there is a disjoint decomposition
  \[
    \TriBar_{h,\M}=\mathcal A_{h,\M}\dot\cup\mathcal B_{h,\M},
  \]
  where $\mathcal A_{h,\M}$ can be partitioned into disjoint symmetric pairs in the sense of \cref{def:symmetric-triangles} and $\mathcal B_{h,\M}=\TriBar_{h,\M}\setminus\mathcal A_{h,\M}$ is the remainder.  Moreover, since $h=\max_{\T\in\TriBar_h}\diam(\T)$,
  \[
    \bigcup_{\T\in\mathcal B_{h,\M}}\T
    \subset
    \{\xM\in\M:\operatorname{dist}(\xM,\partial\M)\le h\},
  \]
  and consequently
  \[
    |\mathcal B_{h,\M}|\le C_{\partial\M}h^{-1}.
  \]
  The constant $C_{\partial\M}$ depends only on the fixed macro element $\M$ and the uniform shape-regularity and quasi-uniformity constants.
\end{lemma}
\begin{proof}
  After two uniform red refinements, the refined macro element is a triangular
  lattice.  Its vertices admit a periodic colouring with three colours such that
  each fine triangle has exactly one vertex of each colour.  We choose one
  colour class.  If $\xBar_0$ is a chosen interior vertex, then the six fine
  triangles incident to $\xBar_0$ form a hexagonal patch.  The patches
  associated with the chosen interior vertices are pairwise disjoint, since
  every fine triangle has only one vertex of the chosen colour. This decomposition
  into hexagonal patches is not unique, but we only need the existence of such
  a pattern.

  Opposite triangles in each hexagonal patch are reflections of one another
  with respect to the common vertex $\xBar_0$.  Hence, their affine reference
  maps can be chosen so that the reflection relation in
  \cref{def:symmetric-triangles} holds.  Thus, every interior patch splits into
  three symmetric pairs.

  Let $\mathcal A_{h,\M}$ be the union of all triangles belonging to these
  symmetric pairs, and let
  $\mathcal B_{h,\M}\colonequals\TriBar_{h,\M}\setminus\mathcal A_{h,\M}$.
  Every fine triangle whose chosen-colour vertex is an interior vertex belongs
  to exactly one of the patches above.  The remaining triangles have their
  chosen-colour vertex on $\partial\M$.  Since finite element cells are treated
  as closed cells and $\diam(\T)\le h$, every point of such a triangle has
  distance at most $h$ from $\partial\M$.  This proves the asserted boundary
  strip inclusion.

  It remains to count the remainder.  The area of the strip
  $\{\xM\in\M:\operatorname{dist}(\xM,\partial\M)\le h\}$ is bounded by
  $C|\partial\M|h$ for $h$ sufficiently small, with $C$ depending only on the
  shape of the fixed macro element.  Uniform shape regularity and
  quasi-uniformity give $|\T|\ge c h^2$ for each fine triangle, with some
  $c>0$ independent of $h$.  Hence the number of fine triangles contained in
  the strip is bounded by $C|\partial\M|/c\,h^{-1}$.

  See \cref{fig:red_refinement} for an illustration.
\end{proof}

\begin{figure}[ht]
    \begin{subfigure}{0.25\linewidth}
        \includegraphics[width=.95\textwidth]{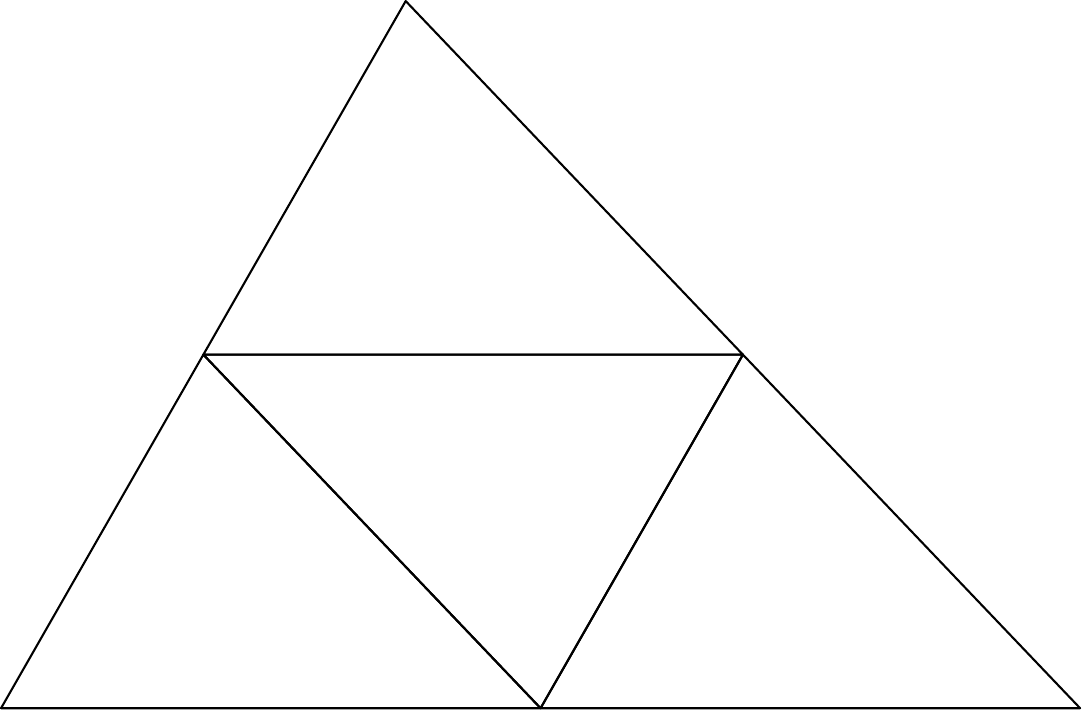}%
    \end{subfigure}\hfill%
    \begin{subfigure}{0.25\linewidth}
      \def\svgwidth{.95\textwidth}
        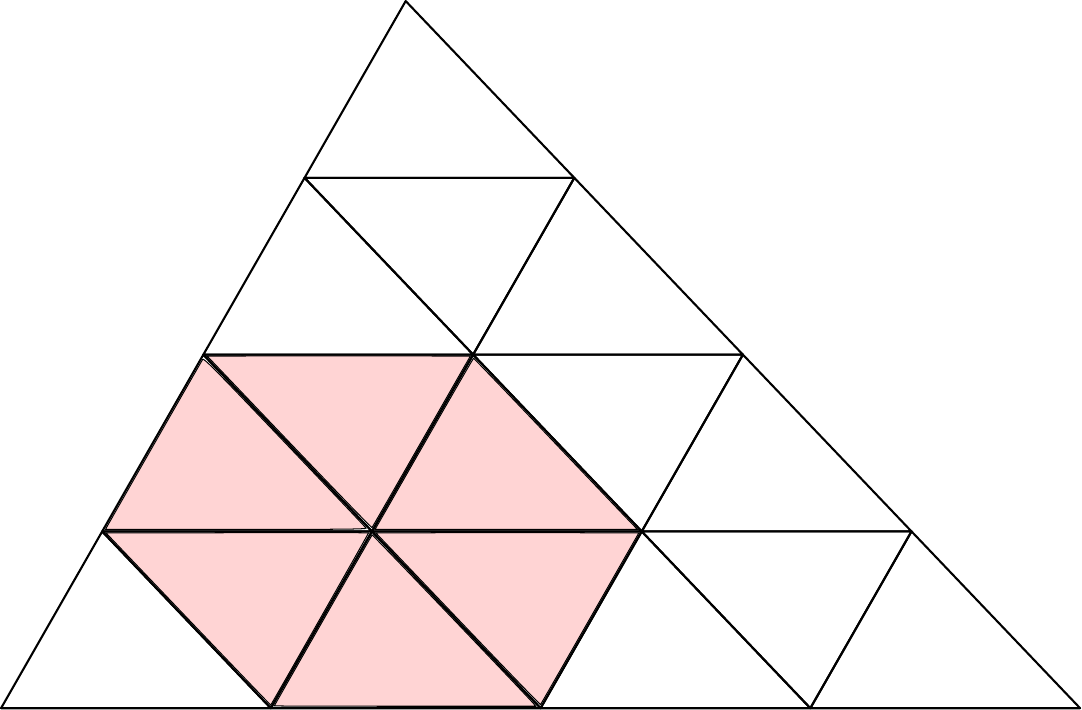%
    \end{subfigure}\hfill%
    \begin{subfigure}{0.25\linewidth}
        \includegraphics[width=.95\textwidth]{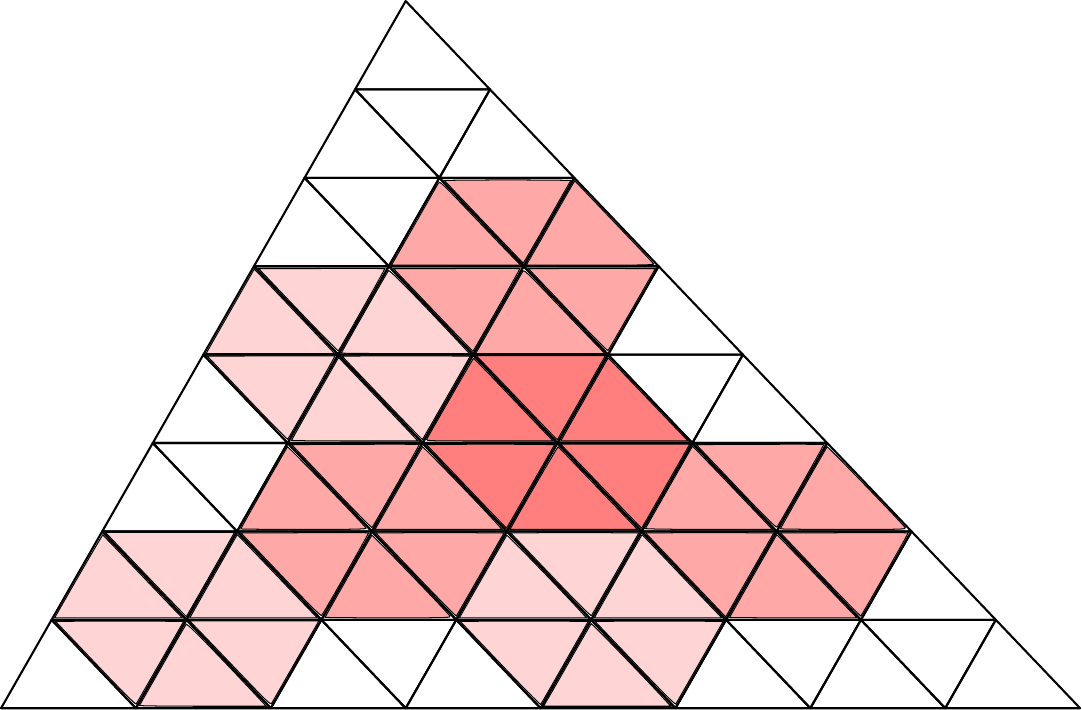}%
    \end{subfigure}\hfill%
    \begin{subfigure}{0.25\linewidth}
        \includegraphics[width=.95\textwidth]{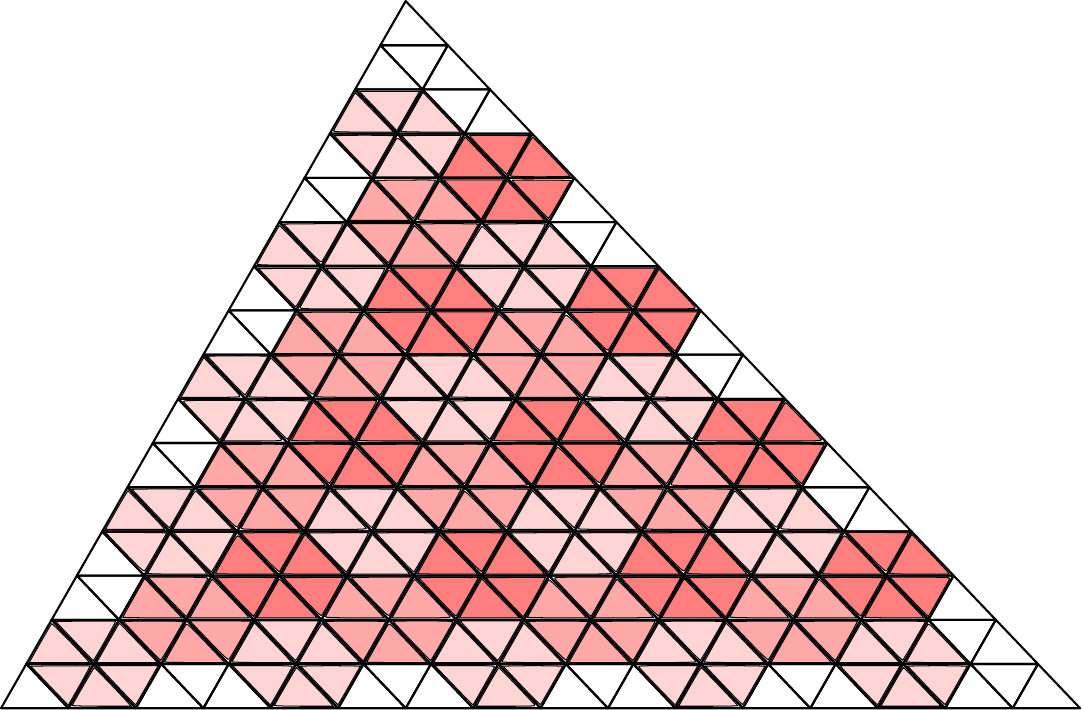}%
    \end{subfigure}
    \caption{Red refinement of a triangular macro element. Hexagonal groups of six triangles are shaded. Boundary triangles are still white. In the second figure included is a sketch of a symmetric pair of triangles $(T_1,T_2)$ connected in the vertex $\xBar_0$.}\label{fig:red_refinement}
    \figurealt{From left to right, one red-refinement step divides the macro triangle into four triangles. After two steps, one shaded six-triangle patch contains the labelled symmetric pair T one and T two meeting at x zero. After three and four steps, shaded six-triangle patches cover the interior, while unshaded triangles remain near the boundary.}
\end{figure}

\begin{remark}\label{rem:symmetric-boundary-triangulation}
  Since red refinement is based on dividing all edges of the triangles in half, the boundary $\partial\M$ of a macro element refined in that way is composed of equal length line segments in each boundary patch.
  Whenever this property holds on all macro-element edges, we say that the refined triangulation has a uniform induced boundary triangulation.
  Consecutive boundary intervals can then be paired on each macro edge, with at most one unpaired interval per macro edge.
  This is an additional boundary condition and is not part of the definition of a symmetric triangulation.
\end{remark}

\begin{definition}\label{def:symmetric-triangulation}
  We call a triangulation $\TriBar_{h,\M}$ \emph{symmetric} if it admits a
  decomposition into symmetric pairs and unpaired boundary elements as in
  \cref{lem:existence-of-symmetric-triangles}.  A global triangulation
  $\TriBar_h$ is also called \emph{symmetric} if all its sub-triangulations
  $\TriBar_{h,\M}$ for $\M\in\TriBar$ are symmetric.
\end{definition}

In the red refinement strategy with global refinement of all elements simultaneously, the global pattern can directly be extracted from the macro element pattern, since neighbouring macro elements have the refinement points in common.

\begin{figure}[ht]
    \begin{subfigure}{0.166\linewidth}
        \includegraphics[width=\textwidth]{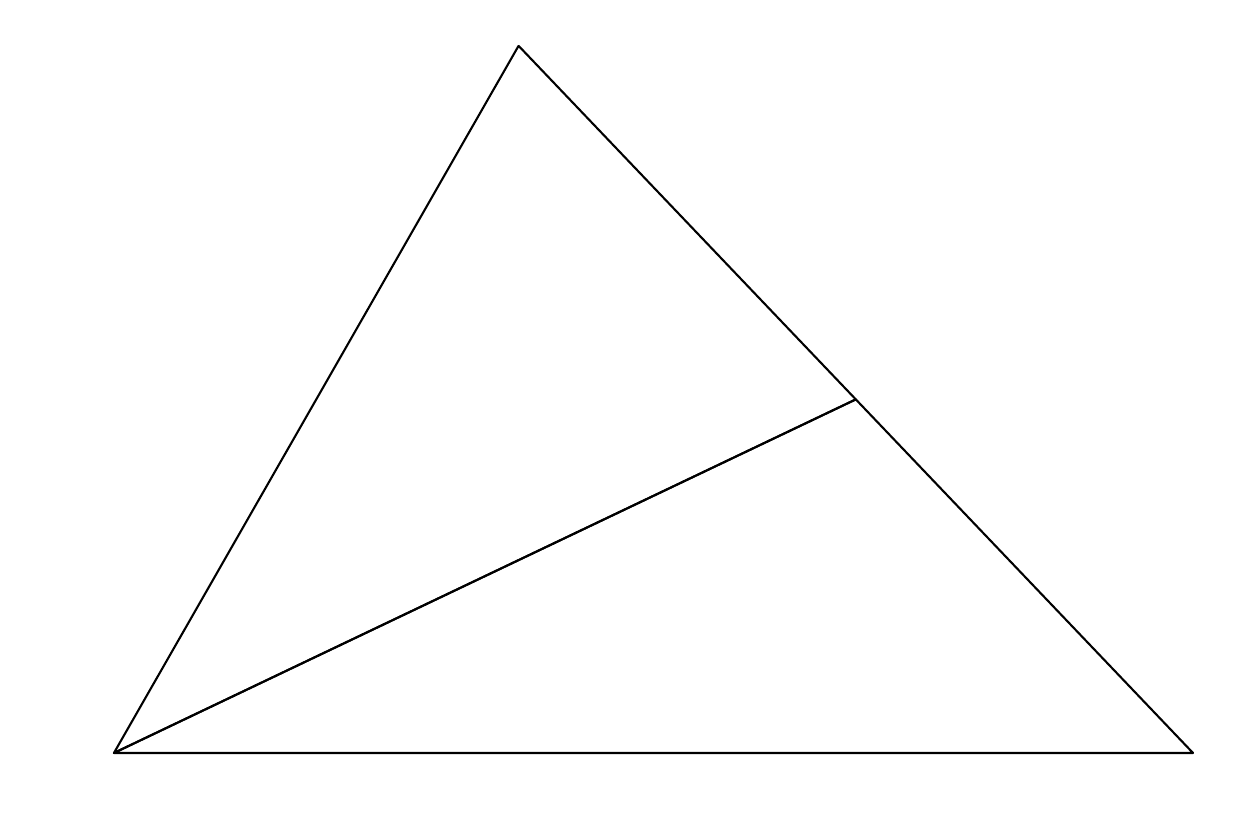}%
    \end{subfigure}\hfill%
    \begin{subfigure}{0.166\linewidth}
        \includegraphics[width=\textwidth]{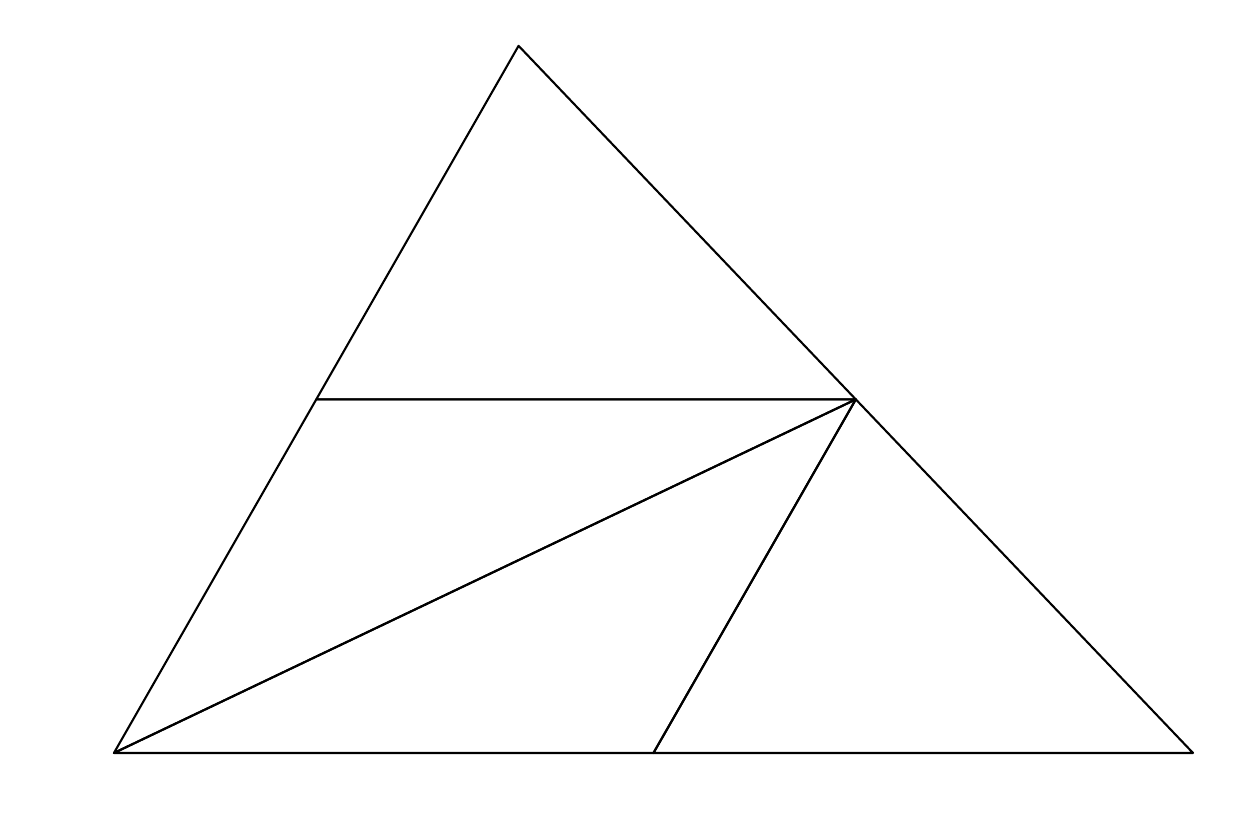}%
    \end{subfigure}\hfill%
    \begin{subfigure}{0.166\linewidth}
        \includegraphics[width=\textwidth]{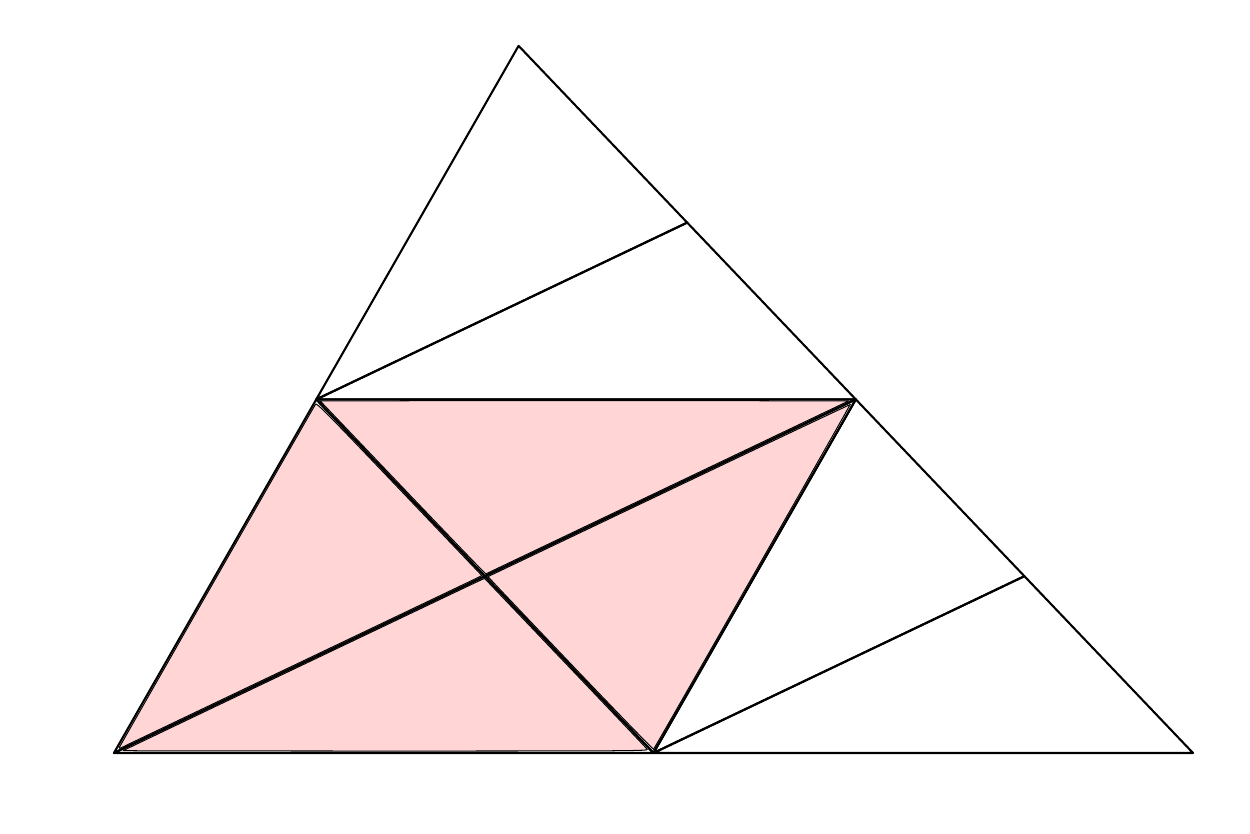}%
    \end{subfigure}\hfill%
    \begin{subfigure}{0.166\linewidth}
        \includegraphics[width=\textwidth]{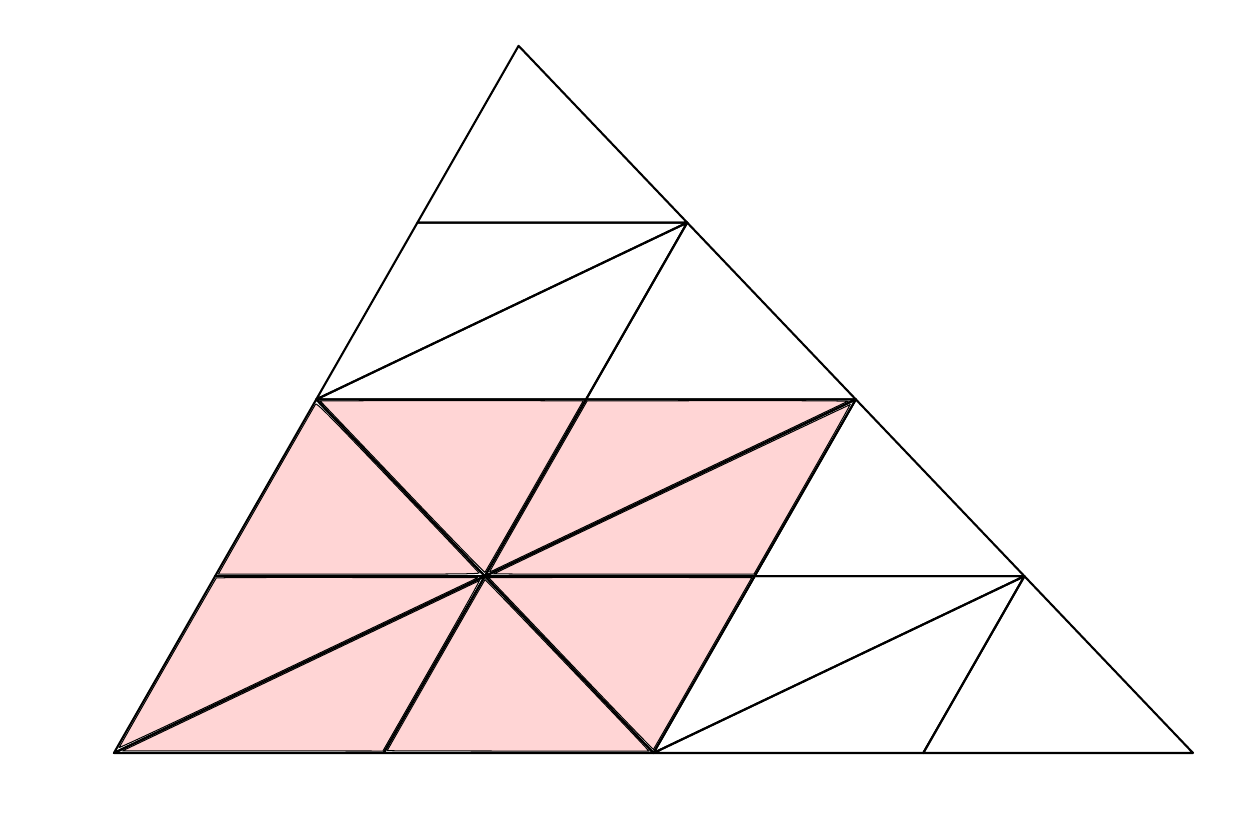}%
    \end{subfigure}\hfill%
    \begin{subfigure}{0.166\linewidth}
        \includegraphics[width=\textwidth]{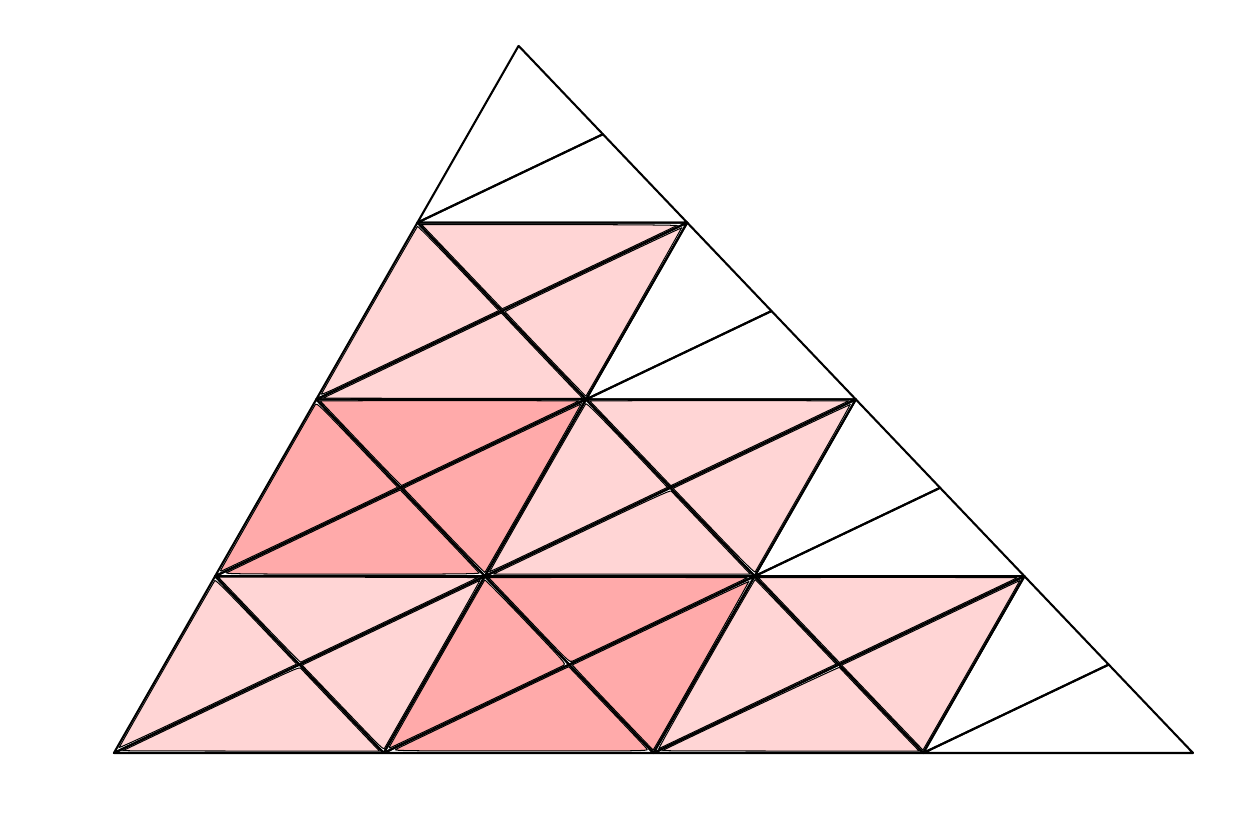}%
    \end{subfigure}\hfill%
    \begin{subfigure}{0.166\linewidth}
        \includegraphics[width=\textwidth]{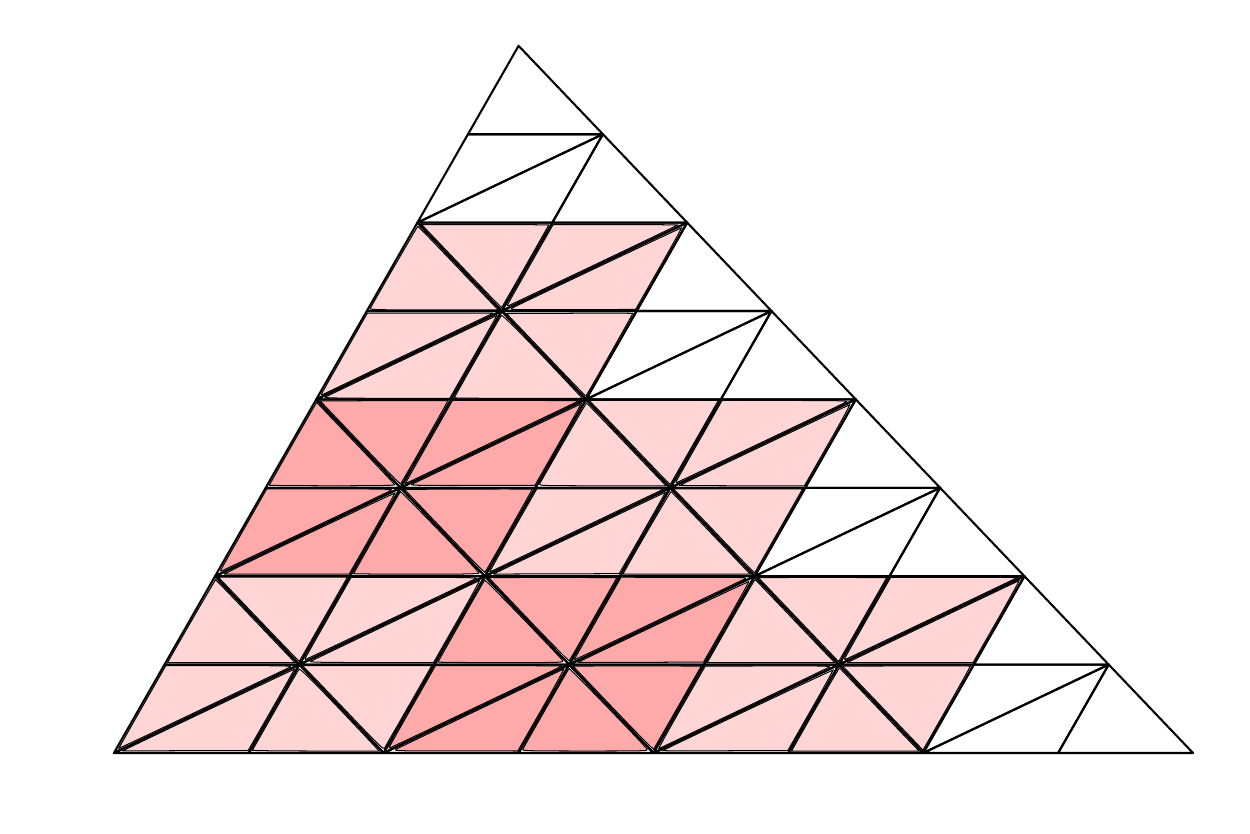}%
    \end{subfigure}
    \caption{Newest vertex bisection of a triangular macro element. Quadrilateral groups of four or eight triangles are shaded.}\label{fig:newest_vertex_bisection}
    \figurealt{From left to right, one through six uniform newest-vertex bisection steps produce two, four, eight, sixteen, thirty-two and sixty-four triangles. Beginning with the third panel, shaded groups of four or eight triangles identify local symmetric pairs, while unshaded triangles remain near the macro-element boundary.}
\end{figure}

\begin{remark}[Newest-vertex bisection]\label{rem:nvb-symmetric-pairs}
  Newest-vertex bisection \citep{Sewell1972Automatic,Bansch1991Local} is another
  common refinement strategy
  \citep[see also][]{SchmidtSiebert2005Alberta,ADKN2016DuneALUGrid}.  It can produce local
  symmetric pair structures on a single triangular macro element when the
  initial marked edge and the inherited newest vertices are fixed in a
  compatible way.  For example, in \cref{fig:newest_vertex_bisection}
  after three uniform refinement steps one
  observes groups of four triangles sharing the midpoint of the initial
  refinement edge.  These four triangles form a quadrilateral that can be
  partitioned into two symmetric pairs.  Subsequent uniform bisections generate
  related local quadrilateral groups among the finitely many similarity classes
  of newest-vertex bisection refinements \citep[cf.][]{Bansch1991Local}.

  We do not use this observation as a general construction of symmetric
  triangulations.  In a global conforming triangulation, the initial markings,
  node numbering, and closure refinements near macro-element interfaces can
  disrupt the simple macro-element-wise pattern.  In particular, unpaired
  elements may occur away from $\partial\M$.  Although such elements are often
  expected to form a small set in structured implementations, no general bound
  is needed for the results below and none is claimed here.
\end{remark}

\begin{remark}
  An alternative refinement strategy is the longest-edge bisection \citep{Rivara1984Mesh}. Unlike uniform red refinement, it generally does not yield a symmetric pattern of element pairs covering the interior of each macro element, since the refinement direction depends dynamically on the geometry and aspect ratio of the macro triangle. As a consequence, no global symmetry comparable to that of red refinement arises directly from this strategy.
\end{remark}

\section{Interpolation errors}\label{sec:interpolation_errors}

Following \citet{Chien1993Piecewise,Chien1995Numerical}, the symmetric structure of the triangulation can be exploited to obtain sharper integral approximation estimates, compared to \cref{lem:surface-approximation-estimates} and standard interpolation estimates. We will present this argument in more detail and extend the findings to a more general setting of geometric quantities in the following subsections.

Throughout this section we use the parametrization from
\cref{ass:admissible-macro-parametrization} and the shape-regular,
quasi-uniform refinements from \cref{ass:admissible-triangulation}.
The standard interpolation estimates do not require any further mesh structure.
The improved integral estimates assume symmetric triangulations in the sense of
\cref{def:symmetric-triangulation}. For the estimates in which boundary terms
from first derivatives have to cancel, we also assume the uniform induced
boundary triangulation described in \cref{rem:symmetric-boundary-triangulation}.
These hypotheses are stated in the local estimates where they are used.

\subsection{Interpolation errors over a single macro element}

In this subsection we study interpolation errors on a single flat macro element (macro patch) $\M$ with symmetric triangulation.
Our main goal is to show that, for even interpolation order, symmetry-induced cancellation in Taylor expansions improves the standard interpolation estimates by one order, and in some cases by two orders. We also extend these improved bounds to derivatives of the interpolated functions.
Throughout this subsection, $C$ denotes a generic constant independent of the mesh size $h$.

We begin by recalling standard interpolation estimates.

\begin{lemma}\label{lem:standard-flat-interpolation-estimate}
  Let $\M\in\TriBar_{h_0}$ be a macro element, let $\bar{f}\in C^{\order+2}(\M;\R)$ with $\order\geq 1$, and let $\TriM_{h,\M}$ be a triangulation of $\M$ with $h<h_0$ sufficiently small. Consider the piecewise Lagrange interpolation $\bar{f}_h = \IBar_h^\order(\bar{f})\in C^0(\M;\R)$. Then, for every multi-index $\vec{\beta}\in\mathbb{N}^{2}_{0}$ with $0\leq|\vec{\beta}|\leq \order+1$, the standard interpolation estimates hold:
  \begin{align}
		\|\DM^{\vec{\beta}}\bar{f}_h - \DM^{\vec{\beta}}\bar{f}\|_{L^p(\M)} &\leq C h^{\order+1-|\vec{\beta}|} |\bar{f}|_{W^{\order+1,p}(\M)},\qquad 1\leq p<\infty, \\
		\max_{\T\in\TriBar_{h,\M}} \|\DM^{\vec{\beta}}\bar{f}_h - \DM^{\vec{\beta}}\bar{f}\|_{L^\infty(\T)} &\leq C h^{\order+1-|\vec{\beta}|} |\bar{f}|_{C^{\order+1}(\M)}.
  \end{align}
\end{lemma}
\begin{proof}
  See, e.g., \citet[Corollary 19.8]{ErnGuermond2021FiniteElementeI}.
\end{proof}

If one is interested in the integrals of differences instead of norms, and the triangulation is symmetric, the estimates can be improved, whenever the interpolation order minus the derivative order is an even number:

\begin{theorem}\label{thm:flat-interpolation-estimate}
  Let $\M\in\TriBar_{h_0}$ be a macro element, let $\bar{f}\in C^{\order+2}(\M;\R)$ with $\order\geq 1$, and let $\TriM_{h,\M}$ be a symmetric triangulation of $\M$ with $h<h_0$ in the sense of \cref{def:symmetric-triangulation}. Consider the piecewise Lagrange interpolation $\bar{f}_h= \IBar_h^\order(\bar{f})\in C^0(\M;\R)$. If $\vec{\beta}\in\mathbb{N}^{2}_{0}$ satisfies $0\leq|\vec{\beta}|\leq \order+2$ and $\order+2-|\vec{\beta}|$ is \emph{even}, then the following improved estimate holds:
  \begin{align}
    \left|\sum_{\T\in\TriM_{h,\M}}
    \int_{\T} (\DM^{\vec{\beta}}\bar{f}_h - \DM^{\vec{\beta}}\bar{f})\,\dM\right| & \leq C h^{\order+2-|\vec{\beta}|} \|\bar{f}\|_{C^{\order+2}(\M)}.
  \end{align}
\end{theorem}
\begin{proof}
If $|\vec{\beta}|=\order+2$, then $\DM^{\vec{\beta}}\bar{f}_h=0$ on each
element and the estimate follows immediately, since the right-hand side has
order $h^0$. We therefore consider $|\vec{\beta}|\leq\order+1$ below.

We start with $|\vec{\beta}|=0$.
In each element $\T\in\TriBar_{h,\M}$, we use a nodal basis expression for $\bar{f}_h$ and a Taylor expansion for $\bar{f}$ to obtain the pointwise representation of the error
 \begin{align*}
		\bar{f}_h(\xM) - \bar{f}(\xM) = \sum_{i=1}^{n_\order} \phiBar_i(\xM)\left( \sum_{|\vec{\alpha}|=\order+1} \frac{1}{\vec{\alpha}!} (\xM_i - \xM)^{\vec{\alpha}} \DM^{\vec{\alpha}}\bar{f}(\xM) + \sum_{|\vec{\alpha}|=\order+2} \frac{1}{\vec{\alpha}!} (\xM_i - \xM)^{\vec{\alpha}} \DM^{\vec{\alpha}}\bar{f}(\bar{y}^{(i)})\right),
 \end{align*}
 where $\phiBar_i = \phi^\order_{i}\circ F_{\T}^{-1}$ is the $i$th Lagrange basis function on $\T$ associated to the Lagrange node $\xM_i=F_{\T}(\xRef_{i,\order})$, and $\bar{y}^{(i)}\in\T$.

  Now we consider symmetric pairs of triangles, i.e., we assume that there is a mirrored triangle $\T'\in\TriBar_{h,\M}$ associated with $\T$ as in \cref{def:symmetric-triangles}. Without loss of generality, the common point is $\xM_0$. We write
  \[
	  \DM^{\vec{\alpha}}\bar{f}(\xM) = \DM^{\vec{\alpha}}\bar{f}(\xM_0) + \sum_{\beta=1}^d (\xM - \xM_0)^{(\beta)} \DM^{\vec{\alpha},\beta}\bar{f}(\bar{y})
  \]
	  for some $\bar{y}\in \T\cup \T'$, by another linear Taylor expansion around $\xM_0$.
  For $|\vec{\alpha}|=\order+1$ (which is odd), employing the symmetry property of \cref{def:symmetric-triangles}, we observe that
  \[
    \int_{\T} \phiBar_i(\xM)(\xM_i - \xM)^{\vec{\alpha}}\,\dM = -\int_{\T'} \phiBar'_i(\xM')(\xM'_i - \xM')^{\vec{\alpha}}\,\dM',
  \]
  with $\phiBar'_i = \phi^\order_{i}\circ F_{\T'}^{-1}$ and $\xM'_i=F_{\T'}(\xRef^\order_{i})$,
  since the triangles are symmetric, the basis functions $\phiBar_i$ (considered on the pair of elements) are even functions, while the Taylor monomial $(\xM_i - \xM)^{\vec{\alpha}}$ is odd. The sign flip comes from the mirrored coordinates.

  Thus, the $|\vec{\alpha}|=\order+1$ terms cancel in the integral over the union of the two triangles, and we obtain the estimate from the remainder term
  \begin{align*}
	    \left|\int_{\T\cup\T'} (\bar{f}_h(\xM) - \bar{f}(\xM))\,\dM\right| &\leq C h^{\order+2} \|\bar{f}\|_{C^{\order+2}(\T\cup\T')}(|\T| + |\T'|).
  \end{align*}
  For a single triangle, the $\order+1$ term remains, and we obtain the weaker estimate
  \begin{align*}
	    \left|\int_{\T} (\bar{f}_h(\xM) - \bar{f}(\xM))\,\dM\right| &\leq C h^{\order+1} \|\bar{f}\|_{C^{\order+1}(\T)} |\T|.
  \end{align*}

  As the unpaired triangles in $\M$ are located only at the boundary $\partial\M$, their number is of order $h^{-1}$, the number of symmetric pairs is of order $h^{-2}$, and the area scales like $|\T|\leq ch^2$, compare \cref{lem:existence-of-symmetric-triangles}, we indeed obtain order $h^{\order+2}$ for the integral over $\M$.

For derivatives, we apply the same argument to the differentiated leading
term. On a symmetric pair let $R\xM\colonequals2\xM_0-\xM$, so that
$R\T=\T'$. For mirrored Lagrange nodes and basis functions,
$\xM_i'=R\xM_i$ and $\phiBar_i'(R\xM)=\phiBar_i(\xM)$. Hence,
for $|\vec{\alpha}|=\order+1$,
\[
  D_y^{\vec{\beta}}
  \big(\phiBar_i'(y)(\xM_i'-y)^{\vec{\alpha}}\big)\big|_{y=R\xM}
  =
  (-1)^{|\vec{\alpha}|-|\vec{\beta}|}
  \DM^{\vec{\beta}}
  \big(\phiBar_i(\xM)(\xM_i-\xM)^{\vec{\alpha}}\big).
\]
Thus, the differentiated leading terms cancel over $\T\cup\T'$ whenever
$\order+2-|\vec{\beta}|$ is even. Taylor's formula with integral remainder
and affine scaling give
\[
  \DM^{\vec{\beta}}(\bar f_h-\bar f)
  =
  \sum_{i=1}^{n_\order}
  \sum_{|\vec{\alpha}|=\order+1}
  \frac{\DM^{\vec{\alpha}}\bar f(\xM_0)}{\vec{\alpha}!}
  \DM^{\vec{\beta}}
  \left(\phiBar_i(\xM)(\xM_i-\xM)^{\vec{\alpha}}\right)
  +\mathcal R_{\order+2-|\vec{\beta}|},
\]
where
\[
  \|\mathcal R_{\order+2-|\vec{\beta}|}\|_{L^\infty(\T\cup\T')}
  \le
  Ch^{\order+2-|\vec{\beta}|}
  \|\bar f\|_{C^{\order+2}(\T\cup\T')}.
\]
The paired elements are therefore estimated using this cancellation, while
the unpaired boundary elements are estimated by the standard interpolation
bound. Summing them as above proves the assertion.
\end{proof}

We can generalize \cref{thm:flat-interpolation-estimate} by introducing a weight function in the integrals:

\begin{corollary}\label{cor:weighted-flat-interpolation-estimate}
  Let the assumptions of \cref{thm:flat-interpolation-estimate} hold, let $\bar{g}\in W^{1,1}(\M)$, and let $\vec{\beta}\in\mathbb{N}_0^2$ satisfy $0\le |\vec{\beta}|\le \order+2$ with $\order+2-|\vec{\beta}|$ even. Then
  \begin{align}
    \left|\sum_{\T\in\TriM_{h,\M}}
    \int_{\T} \bar{g} \cdot (\DM^{\vec{\beta}}\bar{f}_h - \DM^{\vec{\beta}}\bar{f})\,\dM\right| & \leq C h^{\order+2-|\vec{\beta}|} \|\bar{g}\|_{W^{1,1}(\M)}\|\bar{f}\|_{C^{\order+2}(\M)}.
  \end{align}
\end{corollary}
\begin{proof}
  If $|\vec{\beta}|=\order+2$, then
  $\DM^{\vec{\beta}}\bar{f}_h=0$ on each element and the asserted estimate is
  the immediate bound
  \[
    \left|\int_\M \bar{g}\,\DM^{\vec{\beta}}\bar{f}\,\dM\right|
    \leq C\|\bar{g}\|_{L^1(\M)}\|\bar{f}\|_{C^{\order+2}(\M)}.
  \]
  We may therefore assume $|\vec{\beta}|\leq\order+1$.
  Let $\symT$ be either the union of a pair of symmetric triangles, $\symT=T_1\cup T_2$, or an unpaired triangle, $\symT=T$, and set
  $Q_{\symT}\colonequals\operatorname{conv}(\symT)$. Since $\M$ is convex,
  $Q_{\symT}\subseteq\M$, and shape regularity gives
  $\diam(Q_{\symT}) \leq C h$. Thus, the Poincar\'e inequality on the convex
  hull implies
	\begin{equation*}
    \|\bar{g}-\langle \bar{g}\rangle_{Q_{\symT}}\|_{L^1(\symT)}\leq C h \|\bar{g}\|_{W^{1,1}(Q_{\symT})},
  \end{equation*}
  where $\langle \bar{g}\rangle_{Q_{\symT}} = |Q_{\symT}|^{-1} \int_{Q_{\symT}} \bar{g}\,\dM$. We decompose
  \begin{equation*}
    \int_{\symT} \bar{g} \cdot (\DM^{\vec{\beta}}\bar{f}_h - \DM^{\vec{\beta}}\bar{f})\,\dM
    = \langle \bar{g}\rangle_{Q_{\symT}} \int_{\symT} (\DM^{\vec{\beta}}\bar{f}_h - \DM^{\vec{\beta}}\bar{f})\,\dM
    + \int_{\symT} (\bar{g}-\langle \bar{g}\rangle_{Q_{\symT}})(\DM^{\vec{\beta}}\bar{f}_h - \DM^{\vec{\beta}}\bar{f})\,\dM.
  \end{equation*}
  If $\symT$ is a paired union, we estimate using the arguments from \cref{thm:flat-interpolation-estimate} together with the Poincar\'e inequality:
  \begin{align*}
		\left|\int_{\symT} \bar{g} \cdot (\DM^{\vec{\beta}}\bar{f}_h - \DM^{\vec{\beta}}\bar{f})\,\dM\right|
		&\leq C h^{\order+2-|\vec{\beta}|}\|\bar{f}\|_{C^{\order+2}(\symT)} \left|\langle \bar{g}\rangle_{Q_{\symT}} \right| |\symT| + C h^{\order+1-|\vec{\beta}|} \|\bar{f}\|_{C^{\order+1}(\symT)} \|\bar{g}-\langle \bar{g}\rangle_{Q_{\symT}}\|_{L^1(\symT)}\\
     &\leq C h^{\order+2-|\vec{\beta}|}\|\bar{f}\|_{C^{\order+2}(\symT)} \|\bar{g}\|_{W^{1,1}(Q_{\symT})}.
  \end{align*}
  If $\symT=T$ is unpaired, we estimate:
  \begin{align*}
		\left|\int_{\symT} \bar{g} \cdot (\DM^{\vec{\beta}}\bar{f}_h - \DM^{\vec{\beta}}\bar{f})\,\dM\right|
		&\leq C h^{\order+1-|\vec{\beta}|}\|\bar{f}\|_{C^{\order+1}(T)} \|\bar{g}\|_{L^{1}(T)}.
  \end{align*}
  As all unpaired triangles are contained in a boundary strip $\omega_h\subset\M$ of width $h$, we estimate
  \begin{align*}
		\sum_{\T\in\TriBar_{h,\omega_h}}\|\bar{g}\|_{L^{1}(\T)}\leq \|\bar{g}\|_{L^{1}(\omega_h)}\leq Ch\|\bar{g}\|_{W^{1,1}(\M)},
  \end{align*}
  with $\TriBar_{h,\omega_h}\subset\TriBar_{h,\M}$ \citep[cf.][Lemma 2]{Mittelmann1977Gravitational}.
  It remains to sum the estimates over all paired unions. The sets
  $Q_{\symT}$ have uniformly bounded overlap. Indeed, each $Q_{\symT}$ is
  contained in a ball of radius $Ch$ around the vertex shared by the symmetric
  pair. By the quasi-uniformity and shape regularity in
  \cref{ass:admissible-triangulation}, only uniformly many such vertices and
  incident pairs can occur in any ball of radius $Ch$. Hence any point of
  $\M$ belongs to at most uniformly many such convex hulls. Therefore
  \begin{align*}
		\sum_{\symT\ \mathrm{paired}} \|\bar{g}\|_{W^{1,1}(Q_{\symT})} \leq C \|\bar{g}\|_{W^{1,1}(\M)}.
  \end{align*}
  Summing over all paired unions and unpaired triangles yields the claimed
  bound.
\end{proof}

\begin{remark}\label{R:weighted-flat-interpolation-estimate-odd}
  Without the parity condition in \cref{cor:weighted-flat-interpolation-estimate},
  or without a symmetric triangulation, the preceding cancellation argument is
  not available. In that case one still has the standard integral estimate
  for every $0\le |\vec{\beta}|\le \order+1$:
  \[
    \left|\sum_{\T\in\TriM_{h,\M}} \int_{\T} \bar{g} \cdot (\DM^{\vec{\beta}}\bar{f}_h - \DM^{\vec{\beta}}\bar{f})\,\dM\right|
    \leq C h^{\order+1-|\vec{\beta}|}\|\bar{g}\|_{L^{1}(\M)}\|\bar{f}\|_{C^{\order+1}(\M)}.
  \]
	In particular, the estimate above is the fallback bound when
  $\order+2-|\vec{\beta}|$ is odd.
\end{remark}

For the special case of \cref{cor:weighted-flat-interpolation-estimate} with $|\vec{\beta}|=1$, we obtain a further improvement if the weight function has higher regularity:

\begin{lemma}\label{lem:improved-first-derivative}
  Let the assumptions of \cref{thm:flat-interpolation-estimate} hold, and in addition assume that the boundary of the macro element has a uniform induced boundary triangulation as in \cref{rem:symmetric-boundary-triangulation}. Let $\bar{g}\in W^{\ell,1}(\M)$, $\ell\in\{1,2\}$. Then, for $|\vec{\beta}|=1$, we have
  \begin{align}
    \left|\int_{\M} \bar{g} \cdot (\DM^{\vec{\beta}}\bar{f}_h - \DM^{\vec{\beta}}\bar{f})\,\dM\right|
		&\leq \begin{cases}
	  	C h^{\order+\ell} \|\bar{g}\|_{W^{\ell,1}(\M)}\|\bar{f}\|_{C^{\order+\ell}(\M)}, & \text{ for }\order\text{ even},\\
	  	C h^{\order+1} \|\bar{g}\|_{W^{1,1}(\M)}\|\bar{f}\|_{C^{\order+1}(\M)}, & \text{ for }\order\text{ odd}.
		\end{cases}\label{eq:improved-first-derivative}
  \end{align}
\end{lemma}
\begin{proof}
  We use integration by parts and write
  \begin{align*}
		\int_{\M} \bar{g} \cdot (\partial_i\bar{f}_h - \partial_i\bar{f})\,\dM &= -\int_{\M} \partial_i \bar{g} \cdot (\bar{f}_h - \bar{f})\,\dM + \int_{\partial\M} \bar{g} \cdot (\bar{f}_h - \bar{f}) \nu^i \,d \bar{s},\quad i=1,2,
  \end{align*}
  where $\nu$ denotes the outer conormal to the macro element $\M$. For $\order$ odd or $\ell=1$, \eqref{eq:improved-first-derivative} follows from the standard estimate in \cref{R:weighted-flat-interpolation-estimate-odd} with $\vec\beta=0$ for the volume term, and from the corresponding one-dimensional standard interpolation estimate on each macro edge for the boundary term. For $\order$ even and $\ell=2$, we use \cref{cor:weighted-flat-interpolation-estimate} to estimate
  \begin{align*}
		\left|\int_{\M} \DM^{\vec\beta} \bar{g} \cdot (\bar{f}_h - \bar{f})\,\dM \right| \leq C h^{\order+2} \|\bar{g}\|_{W^{2,1}(\M)}\|\bar{f}\|_{C^{\order+2}(\M)}.
  \end{align*}
  On each macro edge, $\nu^i$ is constant and the induced boundary subdivision is uniform. Note that the paired equal-length boundary intervals yield the same parity cancellation as in the two-dimensional argument. Thus, we get
  \begin{align*}
		\left|\int_{\partial\M} \bar{g} \cdot (\bar{f}_h - \bar{f}) \nu^i \,d \bar{s}\right| \leq C h^{\order+2} \|\bar{g}\|_{W^{1,1}(\partial \M)}\|\bar{f}\|_{C^{\order+2}(\M)} \leq C h^{\order+2} \|\bar{g}\|_{W^{2,1}(\M)}\|\bar{f}\|_{C^{\order+2}(\M)},
  \end{align*}
  where the last estimate follows from the trace theorem for Sobolev functions.
\end{proof}

\subsection{Interpolation errors over approximated surface patches}

In this section, we assume that we have a surface $\GM$ parametrized over a single macro element, and a discrete approximation  $\GhM$ as described in \cref{sec:surface-parametrization}. We will consistently write $\kg$ for the interpolation order in the geometry, and in some later theorems use $\ku$ for the interpolation order of functions on that discrete surface. The symmetry assumptions below always refer to the flat induced triangulation $\TriBar_{h,\M}$ of the macro element $\M$. No separate symmetry condition is imposed on the curved surface elements.

Throughout the improved estimates of this subsection we assume that
$\M\in\TriBar_{h_0}$, that the induced flat triangulation
$\TriBar_{h,\M}$ is symmetric in the sense of
\cref{def:symmetric-triangulation}, that
$\PhiM\in C^{\kg+2}(\M;\R^3)$, that $\PhihM$ is the elementwise
degree-$\kg$ Lagrange interpolation of $\PhiM$, and that $h$ is
sufficiently small. In estimates that use
\cref{lem:improved-first-derivative}, we further assume that the boundary of
$\M$ has a uniform induced boundary triangulation as in
\cref{rem:symmetric-boundary-triangulation}. Without these symmetry
hypotheses, the corresponding arguments give the standard estimates from
\cref{lem:standard-flat-interpolation-estimate,R:weighted-flat-interpolation-estimate-odd}.

In the following estimates, the constant $C$ might depend on up to $\kg+2$ derivatives of the surface parametrization $\Phi$ and on the coarse grid size $h_0$, but not on the refined grid size $h$.

In order to compare $\GM$ and $\GhM$, we first estimate the integral over the signed-distance function.

\begin{lemma}\label{lem:dist_estimate}
  Under the hypotheses of this subsection, let $\kg$ be even. If
  $f_h\in W^{1,1}(\GhM)$, then
  \begin{align}
    \left|\int_{\GhM} \dist f_h\,\dGh \right| &\leq C h^{\kg+2}\Norm{f_h}{W^{1,1}(\GhM)},
  \end{align}
  with $\dist\colon U_{\delta}(\G)\to\R$ the surface signed-distance function.
\end{lemma}
\begin{proof}
	We lift the integral over $\GhM$ to the macro element $\M$ using the abbreviation $\Phi_h\colonequals\PhihM$, apply \cref{lem:surface-approximation-estimates,lem:surface-elements}, and estimate
	\begin{align*}
		\left|\int_{\GhM} \dist f_h\,\dGh\right| & =  \left|\int_{\M} \dist(\Phi_h(\xM)) f_h(\Phi_h(\xM))\,\mGh(\xM)\dM\right|\\
		& \leq \left|\int_{\M} \dist(\Phi_h(\xM)) f_h(\Phi_h(\xM))\,\mG(\xM)\dM\right|
		+ \underbrace{\|\dist\|_{L^{\infty}(\GhM)}}_{\leq C h^{\kg+1}} \|f_h\circ \Phi_h\|_{{L^1}(\M)}  \underbrace{\|\mG - \mGh\|_{L^{\infty}(\M)}}_{\leq Ch^{\kg}}.
	\end{align*}
	We now set $\bar{f} \colonequals (f_h\circ \Phi_h)\cdot\mG \in W^{1,1}(\M)$,
	and consider the integral $\int_{\M} \dist(\Phi_h(\xM)) \bar{f}(\xM)\,\dM$.

	Let $\xM\in \M$ be a point in the macro element. In order to rewrite $\dist(\Phi_h(\xM))$ in terms of an interpolation error, we introduce symbols for the difference in the discrete and smooth parametrization and lifted normal vectors,
	\begin{align*}
		\bm{e}(\xM)\colonequals \Phi_h(\xM)-\Phi(\xM), \qquad
		\nM(\xM)   \colonequals \bm{n}(\Phi(\xM)), \qquad
		\nhM(\xM)  \colonequals \bm{n}(\Phi_h(\xM)).
	\end{align*}
  From the closest-point projection formula \eqref{eq:closest-point-projection}, we have
  \begin{align*}
    \dist(\Phi_h(\xM))\nhM(\xM)
    &= \Phi_h(\xM) - \pi(\Phi_h(\xM))
     = \bm{e}(\xM)+ \Phi(\xM) - \pi(\Phi_h(\xM)).
  \end{align*}
	A Taylor expansion of $\pi$ at $\Phi(\xM)\in \G$ gives
  \begin{align*}
    \pi(\Phi_h(\xM))
    &= \pi(\Phi(\xM))+D\pi(\Phi(\xM))\cdot\bm{e}(\xM)+\OO(\|\bm{e}(\xM)\|^2)\\
    &= \Phi(\xM)+(\Id - \nM(\xM)\otimes \nM(\xM))\bm{e}(\xM)+\OO(\|\bm{e}(\xM)\|^2).
  \end{align*}
	Thus we obtain
  \begin{align*}
		\dist(\Phi_h(\xM)) & = \dist(\Phi_h(\xM))(\nhM(\xM)\cdot \nM(\xM) + \frac{1}{2}\|\nhM(\xM)-\nM(\xM)\|^2)\\
		&= \nM(\xM)\cdot\bm{e}(\xM)  + \OO(\|\bm{e}(\xM)\|^2).
  \end{align*}

	For the function $\bar{f}\in W^{1,1}(\M)$, with
  $\|\bar f\|_{W^{1,1}(\M)}\le C\|f_h\|_{W^{1,1}(\GhM)}$ by the uniform
  pullback norm equivalences, we rewrite the integral as
  \begin{align*}
    \int_\M \dist(\Phi_h(\xM))\bar{f}(\xM)\,\dM
    = \int_\M (\nM\cdot\bm{e}) \bar{f}\,\dM
      + \int_\M \OO(\|\bm{e}\|^2)\bar{f}\,\dM.
  \end{align*}
	For the first term, we apply \cref{cor:weighted-flat-interpolation-estimate} (with $\vec{\beta}=0$ and $\order=\kg$) componentwise to the geometry interpolation error, using weights $\bar{g}_j\colonequals \bar{f}\,\nM_j\in W^{1,1}(\M)$.
	For the second term, we use the standard interpolation estimate from \cref{lem:standard-flat-interpolation-estimate}, namely $\|\bm{e}\|_{L^\infty(\M)}\le Ch^{\kg+1}$, to get
  \begin{align*}
    \left|\int_\M \OO(\|\bm{e}\|^2)\bar{f}\,\dM\right|
    \leq C h^{2\kg+2}\|\bar{f}\|_{L^1(\M)}
    \leq C h^{\kg+2}\|\bar{f}\|_{W^{1,1}(\M)}.
  \end{align*}
  Hence,
  \begin{align*}
    \left|\int_\M \dist(\Phi_h(\xM))\bar{f}(\xM)\,\dM\right|
    \leq C h^{\kg+2}\|\bar{f}\|_{W^{1,1}(\M)},
  \end{align*}
  from which the assertion follows.
\end{proof}

In the following assertions we consider functions on $\G$ lifted to the discrete surface by either using the closest-point projection $\pi$ or the lifting $L$. By \cref{lem:regularity-interpolated-geometry}, both maps are well defined on $\Gh$ for $h$ sufficiently small. On a single macro patch, however, they need not have the same image: $L(\GhM)=\GM$, while $\pi(\GhM)$ is a possibly different subsurface of $\G$. Thus, we consider functions on $\GM$ in the case of the lifting $L$, and functions on $\pi(\GhM)$ in the case of the closest-point projection.

\begin{corollary}\label{cor:weighted-diff-of-area}
	Under the hypotheses of this subsection, let $\kg$ be even. For every
	$g_1\in W^{1,1}(\pi(\GhM))$,
	\begin{align}\label{eq:weighted-diff-of-area-g1}
		\left|\int_{\pi(\GhM)}g_1\,\dG - \int_{\GhM} g_1\circ \pi\,\dGh\right| &\leq C h^{\kg+2}\|g_1\|_{W^{1,1}(\pi(\GhM))}.
	\end{align}
	If, in addition, the boundary hypothesis of this subsection holds, then for
	$g_2\in W^{\ell,1}(\GM)$, $\ell\in\{1,2\}$,
	\begin{align}\label{eq:weighted-diff-of-area-g2}
		\left|\int_{\GM}g_2\,\dG - \int_{\GhM} g_2\circ  L \,\dGh\right| &\leq C h^{\kg+\ell}\|g_2\|_{W^{\ell,1}(\GM)}.
	\end{align}
\end{corollary}
\begin{proof}
	Set $\mathfrak{L}_1 = \pi$ and $\mathfrak{L}_2 = L$. Then $\GM = \mathfrak{L}_2(\GhM)$, and we write for $i=1,2$
	\begin{align*}
		\left|\int_{\mathfrak{L}_i(\GhM)}g_i\,\dG - \int_{\GhM} g_i\circ  \mathfrak{L}_i \,\dGh\right|
		&= \left|\int_{\M}g_i\circ\mathfrak{L}_i\circ \Phih \,(\mu[\mathfrak{L}_i\circ \Phih]-\mGh)\dM\right|.
	\end{align*}
	As in the proof of \citet[Lemma 3]{Nedelec1976Curved}, we write pointwise
	\begin{align*}
	\mu[\mathfrak{L}_i\circ \Phih]^2 - (\mGh)^2
		&= \|\partial_1(\mathfrak{L}_i\circ \Phih)\times\partial_2(\mathfrak{L}_i\circ \Phih)\|^2 - \|\partial_1\Phih\times\partial_2\Phih\|^2 \\
	&= 2\langle\partial_1(\mathfrak{L}_i\circ \Phih-\Phih),\,\partial_2(\mathfrak{L}_i\circ \Phih)\times(\partial_1(\mathfrak{L}_i\circ \Phih)\times\partial_2(\mathfrak{L}_i\circ \Phih))\rangle\\
			&\phantom{=}- 2\langle\partial_2((\mathfrak{L}_i\circ \Phih)-\Phih),\,\partial_1(\mathfrak{L}_i\circ \Phih)\times(\partial_1(\mathfrak{L}_i\circ \Phih)\times\partial_2(\mathfrak{L}_i\circ \Phih))\rangle + \mathcal{O}(h^{2\kg}).
	\end{align*}

	For $i=1$, we have $\mathfrak{L}_1\circ \Phih = \pi\circ \Phih$. By definition of $\pi$, we have pointwise
	\[
			\pi(\Phih(\xM))-\Phih(\xM) = - \dist(\Phih(\xM))\nn(\Phih(\xM)),
		\]
	and the vectors $\partial_k(\pi\circ \Phih)\times(\partial_1(\pi\circ \Phih)\times\partial_2(\pi\circ \Phih))$, $k=1,2$, are tangent to $\G$ at $\pi(\Phih(\xM))$ and thus orthogonal to $\nn(\Phih(\xM))$. Hence, we can write pointwise for $j,k=1,2$
	\begin{multline*}
		\langle\partial_j(\pi\circ \Phih-\Phih),\,\partial_k(\pi\circ \Phih)\times(\partial_1(\pi\circ \Phih)\times\partial_2(\pi\circ \Phih))\rangle\\
		=  \dist(\Phih(\xM)) \langle \nn \circ\Phih ,\,\partial_j(\partial_k(\pi\circ \Phih)\times(\partial_1(\pi\circ \Phih)\times\partial_2(\pi\circ \Phih)))\rangle.
	\end{multline*}
	Inserting this into the integral, and applying \cref{lem:dist_estimate} yields the assertion \eqref{eq:weighted-diff-of-area-g1}.

	For $i=2$, we have $\mathfrak{L}_2\circ \Phih = \Phi$.
	Inserting this into the integral,
	\begin{align*}
		\left|\int_{\M}g_2\circ\Phi\,(\mu[\Phi]-\mGh)\dM\right|
			&= \left|\int_{\M}\frac{g_2\circ\Phi}{\mu[\Phi] + \mGh}\,(\mu[\Phi]^2 - (\mGh)^2)\dM\right| \\
			&= \left|\int_{\M}\frac{g_2\circ\Phi}{2\mu[\Phi]}\,(\mu[\Phi]^2 - (\mGh)^2)\dM\right| + \mathcal{O}(h^{2\kg}),
	\end{align*}
  	and applying \cref{lem:improved-first-derivative} for either $\ell=1$ or $\ell=2$, leads to the assertion \eqref{eq:weighted-diff-of-area-g2}.

\end{proof}

The following estimate improves the order in the difference between the two
liftings $L$ and $\pi$ from \cref{lem:interpolation-error-estimates}. For $\ell=2$, the improvement again
comes from the symmetric triangulation through
\cref{cor:weighted-flat-interpolation-estimate}.

\begin{lemma}\label{lem:F-L-F-pi-estimate}
	Under the hypotheses of this subsection, let $\kg$ be even. Let $p\in[1,\infty]$ and $f\in W^{\ell,p}(\G)$, $\ell\in\{1,2\}$, $\bar{g}\in W^{\ell-1,\bar{p}}(\M)$, with $\frac{1}{p}+\frac{1}{\bar{p}}=1$. Then we have the estimate
	\begin{align}
		\left|\int_{\M} (f\circ \pi\circ \Phih - f\circ \Phi)\cdot \bar{g}\, \dM\right|
		&\leq C h^{\kg+\ell} \|f\|_{W^{\ell,p}(\G)}\|\bar{g}\|_{W^{\ell-1,\bar{p}}(\M)}.
	\end{align}
\end{lemma}
\begin{proof}
	Note that for $1\leq p<\infty$, by density it is enough to show the estimates for $f\in C^{\infty}(\G)$.
	For $\ell=1$, we use Hölder's inequality to estimate
	\begin{equation*}
		\left|\int_{\M} (f\circ \pi\circ \Phih - f\circ \Phi)\cdot \bar{g}\, \dM\right| \leq \|f\circ \pi\circ \Phih - f\circ \Phi\|_{L^p}\|\bar{g}\|_{L^{\bar{p}}(\M)}.
	\end{equation*}
	For any $\xM\in \Omega$, we define the projected straight-line homotopy
	\[
		\Gamma_t(\xM)\colonequals \pi\big(\Phi(\xM)+t(\pi\circ\Phih(\xM)-\Phi(\xM))\big),
		\qquad 0\le t\le 1 .
	\]
	By \cref{lem:surface-approximation-estimates}, we have $\|\pi\circ \Phih-\Phi\|_{L^\infty(\M)}\le Ch^{\kg+1}$, and thus $\Gamma_t$ is indeed well-defined for $h$ small enough, and
	\begin{equation*}
		|\partial_{t}\Gamma_t(\xM)|\leq C h^{\kg+1}\quad \forall \xM\in \M,\ t\in[0,1].
	\end{equation*}
	Note that for each $t\in [0,1]$ the mapping $\Gamma_t:\M\to\Gamma$ is a
	uniformly regular parametrization of a patch of $\G$. In particular its
	area element $\mu[\Gamma_t]$ is bounded above and below independently of
	$h$ and $t$. Consequently, for every $v\in L^p(\G)$,
	\[
		\|v\circ\Gamma_t\|_{L^p(\M)}\le C\|v\|_{L^p(\G)},\qquad 0\le t\le1,
	\]
	where $v\circ\Gamma_t$ is the usual pullback of $v$ from the surface to the macro element. Below we use this bound with $v=\gradG f$.
	By the fundamental theorem of calculus, we can thus estimate pointwise
	\begin{equation*}
		\left|f(\pi(\Phih(\xM))) - f(\Phi(\xM))\right| = \left|\int_0^1 \frac{d}{dt} f(\Gamma_t(\xM)) \;\dt\right| \leq Ch^{\kg+1}\int_0^1 |\gradG f(\Gamma_t(\xM))|\;\dt.
	\end{equation*}
	Taking the $L^p$-norm, we indeed obtain
	\begin{equation*}
		\|f\circ \pi\circ \Phih - f\circ \Phi\|_{L^p} \leq Ch^{\kg+1}\sup_{t\in [0,1]}\|\gradG f \circ\Gamma_t\|_{L^p(\M)}\leq Ch^{\kg+1}\|\gradG f\|_{L^p(\G)}.
	\end{equation*}
	For $p=\infty$ the same estimate follows from the Lipschitz representative of $f\in W^{1,\infty}(\G)$.

	For $\ell=2$ and $1\leq p < \infty$, we again use a density argument and a Taylor expansion of the smooth function $f\circ \pi$ to write
	\begin{align*}
		f(\pi(\Phih(\xM)))
		& = f(\Phi(\xM)) + \D(f\circ \pi)(\Phi(\xM))\cdot \bm{e}(\xM)+ \frac{1}{2}\bm{e}^{T}(\xM)\D^2(f\circ \pi)(\xi)\bm{e}(\xM),
	\end{align*}
	where $\bm{e}(\bar{x})= \Phih(\xM) -\Phi(\xM)$, and $\xi\in \R^3$ with $|\xi-\Phi(\xM)|\leq Ch^{\kg+1}$.
	Thus, by elementwise integration we obtain
	\begin{multline*}
		\int_{\M} (f\circ \pi\circ \Phih - f\circ \Phi)\cdot g\, \dM
		= \int_{\M} (D(f\circ \pi) \circ\Phi\cdot \bm{e}+\frac{1}{2}\bm{e}^{T} D^2(f\circ \pi)(\xi)\bm{e} )\cdot \bar{g}\, \dM\\
		\leq C h^{\kg+2} \|\Phi\|_{C^{\kg+2}(\M)} \|G\|_{W^{1,1}(\M)} + Ch^{2m+2}\|\Phi\|_{C^{\kg+1}(\M)}^2 \|f\|_{W^{2,p}(\G)}\|\bar{g}\|_{L^{\bar{p}}(\M)},
	\end{multline*}
	where we used \cref{cor:weighted-flat-interpolation-estimate} component-wise with the weight functions $G_j=  \bar{g} D(f\circ \pi)_j \circ\Phi$ for the first term, and standard interpolation error estimates for the second term.
	The assertion then follows from estimating $\|G\|_{W^{1,1}(\M)}$.
 	For $p=\infty$, the same argument is interpreted with a $C^{1,1}$-representative and the one-dimensional Taylor formula along almost every path $t\mapsto \Phi(\xM)+t\bm e(\xM)$.
\end{proof}

The following theorems characterize the interpolation error of discrete functions on the discrete surface and their derivatives. These results extend the classical findings of \citet{Chien1993Piecewise,Chien1995Numerical} for quadrature errors on discrete surfaces and are partially reproduced in \citet{zavalani2024note}. Our proofs build on the estimates on the macro elements and use the parametrization as an additional weight function. The assumptions needed for these applications are the hypotheses of this subsection.

In the assertions of the theorems, we combine even and odd polynomial order statements, by introducing the following notation: for an integer $\order\geq 1$, set
\begin{align}
	\widehat{\order} &\colonequals \begin{cases}
		\order+1 & \text{ if }\order\text{ is odd}, \\
		\order+2 & \text{ if }\order\text{ is even},
	\end{cases} &
	\widehat{\order}^{(\ell)} &\colonequals \begin{cases}
		\order+1 & \text{ if }\order\text{ is odd}, \\
		\order+\ell & \text{ if }\order\text{ is even}.
	\end{cases}
	\label{eq:even-odd-order}
\end{align}

\begin{theorem}\label{thm:interpolation_estimate_-1}
		Assume the hypotheses of this subsection, including the boundary hypothesis. Let $f\in C^{\ku+2}(\GM;\R)$ be a scalar function. Let $f_h \colonequals \II_{h,\kg}^\ku f\in C^0(\GhM;\R)$ be the interpolation of $f$, and $g\in W^{\ell,1}(\GM)$, $\ell\in\{1,2\}$.
		Then
	\begin{align}
		\left| \int_{\GM} f\cdot g \,\dG - \int_{\GhM} f_h\cdot (g\circ L)\,\dGh \right|
			&\leq Ch^{\min\{\widehat{\ku}, \widehat{\kg}^{(\ell)}\}} \|f\|_{C^{\widehat{\ku}}(\GM)} \|g\|_{W^{\ell,1}(\GM)},
	\end{align}
		with $\widehat{\ku}$ and $\widehat{\kg}^{(\ell)}$ as in \eqref{eq:even-odd-order}.
\end{theorem}
\begin{proof}
	We set $\bar{g}\colonequals g\circ \Phi$ and $\bar{f}\colonequals f\circ\Phi$. Then $f_h\circ L^{-1} = \bar{f}_h \circ \Phi^{-1}$ with $\bar{f}_h\colonequals\IBar_h^\ku(\bar{f})$.
	We estimate
	\begin{multline*}
		\left| \int_{\GM} f\cdot g \,\dG - \int_{\GhM} f_h\cdot (g\circ L)\,\dGh \right| \\
		\begin{aligned}
		& \leq
		\left| \int_{\GM} (f-f_h\circ L^{-1})\cdot g \,\dG \right| + \left| \int_{\GM} (f_h \circ L^{-1}) \cdot g \,\dG - \int_{\GhM} f_h\cdot (g\circ L)\,\dGh \right|\\
		& =
		\left| \int_{\M}  (f-f_h\circ L^{-1})\circ \Phi\cdot \bar{g}\; \mG \,\dM \right| + \left| \int_{\M}  f_h\circ \PhihM\cdot \bar{g} \; (\mG-\mGh) \dM \right|.
		\end{aligned}
	\end{multline*}
	Note that $f_h \circ L^{-1}\circ \Phi = \bar{f}_h = f_h\circ \PhihM$. As $\bar{g}\mG \in W^{1,1}(\M)$, for even $\ku$ we estimate the first term using \cref{cor:weighted-flat-interpolation-estimate}, while for odd $\ku$ we use the standard weighted estimate from \cref{R:weighted-flat-interpolation-estimate-odd}:
	\begin{align*}
			\left| \int_{\M} (\bar{f} - \bar{f}_h)\cdot \bar{g}\mG \,\dM \right| \leq Ch^{\widehat{\ku}}\|\bar{f}\|_{C^{\widehat{\ku}}(\M)}\|\bar{g}\|_{W^{1,1}(\M)}\leq Ch^{\widehat{\ku}}\|f\|_{C^{\widehat{\ku}}(\GM)}\|{g}\|_{W^{1,1}(\GM)}
	\end{align*}
	with $\widehat{\ku}$ as above.
	We add and subtract a smooth term $f\circ\Phi=\bar{f}$, split up the integral, and use standard interpolation error estimates, \cref{R:weighted-flat-interpolation-estimate-odd},
	\begin{align*}
		\left| \int_{\M} \bar{f}_h \cdot \bar{g}\;(\mG - \mGh) \,\dM\right|
		 \leq \left| \int_{\M} \bar{f} \cdot \bar{g}\;(\mG - \mGh) \,\dM\right|+ \underbrace{\left| \int_{\M} (\bar{f}_h - \bar{f}) \cdot \bar{g}\;(\mG - \mGh) \,\dM\right|}_{\leq Ch^{\ku+\kg+1}\|f\|_{C^{\ku+1}(\GM)}\|g\|_{L^{1}(\GM)}}.
	\end{align*}
	For even $\kg$, the first term is estimated using \cref{cor:weighted-diff-of-area}. For odd $\kg$, we use the same area-element expansion as in the proof of that corollary and apply \cref{lem:improved-first-derivative} with interpolation order $\kg$. Since $f \cdot g \in W^{\ell,1}(\GM)$, both cases give
	\begin{align*}
		\left| \int_{\M} \bar{f} \cdot \bar{g}\;(\mG - \mGh) \,\dM\right|
		&= \left| \int_{\GM} f\cdot g \,\dG - \int_{\GhM} (f\cdot g)\circ L\,\dGh\right|\\
		&\leq C h^{\widehat{\kg}^{(\ell)}}\| f\cdot g\|_{W^{\ell,1}(\GM)}\\
		&\leq C h^{\widehat{\kg}^{(\ell)}}\|f\|_{C^{\ell}(\GM)}\| g\|_{W^{\ell,1}(\GM)},
	\end{align*}
	with $\widehat{\kg}^{(\ell)}$ as above.
	Since $\ku,\kg\geq 1$ we have $\ku+\kg+1\geq \min\{\widehat{\ku},\widehat{\kg}^{(\ell)}\}$ and the assertion follows.
\end{proof}

If we lift with $\pi$ instead of $L$, we get a slightly better estimate, but have to assume that the functions are smooth on a larger subsurface of $\G$ than $\GM$.

\begin{theorem}\label{thm:interpolation_estimate_0}
	Assume the hypotheses of this subsection. Let $f\in C^{\ku+2}(\G;\R)$ be a scalar function. Let $f_h \colonequals \II_{h,\kg}^\ku f\in C^0(\Gh;\R)$ be the interpolation of $f$, and $g\in W^{1,1}(\G)$.
	Then
	\begin{equation}
		\left| \int_{\pi(\GhM)} f\cdot g \,\dG - \int_{\GhM} f_h\cdot (g\circ \pi)\,\dGh \right|
		\leq Ch^{\min\{\widehat{\ku}, \widehat{\kg}\}} \|f\|_{C^{\widehat{\ku}}(\G)} \|g\|_{W^{1,1}(\G)},
	\end{equation}
	with $\widehat{\ku}$ and $\widehat{\kg}$ as in \eqref{eq:even-odd-order}.
\end{theorem}
\begin{proof}
	We set $\bar{g}\colonequals g\circ \pi\circ \Phih$ and $\bar{f}\colonequals f\circ\Phi$. Then $f_h\circ \Phih = \bar{f}_h$ with $\bar{f}_h\colonequals\IBar_h^\ku(\bar{f})$. By \cref{lem:regularity-interpolated-geometry}, the map $\pi\circ\Phih\colon\M\to\pi(\GhM)$ is uniformly regular for $h$ sufficiently small.
	We estimate
	\begin{align*}
		\Big| \int_{\pi(\GhM)} f\cdot g \,\dG &- \int_{\GhM} f_h\cdot (g\circ \pi)\,\dGh \Big|
		= \left| \int_{\M} f\circ \pi\circ \Phih \cdot \bar{g} \;\mu[\pi\circ \Phih]\,\dM - \int_{\M} \bar{f}_h\cdot \bar{g}\;\mGh\,\dM \right| \notag\\
		& \leq \underbrace{\left| \int_{\M} (f\circ \pi\circ \Phih -\bar{f}) \cdot \bar{g} \;\mu[\pi\circ \Phih]\,\dM\right|}_{(i)}
		+ \underbrace{\left|\int_{\M} \bar{f}\cdot \bar{g}\;(\mu[\pi\circ \Phih]-\mGh)\,\dM\right|}_{(ii)} \\
		&\phantom{\leq} + \underbrace{\left| \int_{\M} (\bar{f}- \bar{f}_h)\cdot \bar{g}\;\mGh\,\dM \right|}_{(iii)}.
	\end{align*}
	The third term $(iii)$ can be further estimated as
	\begin{align*}
		\left| \int_{\M} (\bar{f}- \bar{f}_h)\cdot \bar{g}\;\mGh\,\dM \right|
		&\leq \left| \int_{\M} (\bar{f}- \bar{f}_h)\cdot \bar{g}\;\mG\,\dM \right| + \left| \int_{\M} (\bar{f}- \bar{f}_h)\cdot \bar{g}\;(\mG-\mGh)\,\dM \right|.
	\end{align*}
		As $\bar{g}\mG \in W^{1,1}(\M)$, for even $\ku$ we use \cref{cor:weighted-flat-interpolation-estimate}, while for odd $\ku$ we use the standard weighted estimate from \cref{R:weighted-flat-interpolation-estimate-odd}:
	\begin{align*}
		\left| \int_{\M} (\bar{f} - \bar{f}_h)\cdot \bar{g}\mGh \,\dM \right| &\leq Ch^{\widehat{\ku}}\|\bar{f}\|_{C^{\widehat{\ku}}(\M)}\|\bar{g}\|_{W^{1,1}(\M)} + Ch^{\ku+\kg+1}\|\bar{f}\|_{C^{\ku+1}(\M)}\|\bar{g}\|_{L^{1}(\M)}\\
		&\leq Ch^{\widehat{\ku}}\|f\|_{C^{\widehat{\ku}}(\G)}\|{g}\|_{W^{1,1}(\G)}.
	\end{align*}

	For the second term $(ii)$, define $g_3\colonequals \bar{f}\circ(\pi\circ\Phih)^{-1}\cdot g$ on $\pi(\GhM)$.
	Then $g_3\circ\pi\circ\Phih=\bar f\,\bar g$, and the uniform regularity of $\pi\circ\Phih$ gives
	\[
		\|g_3\|_{W^{1,1}(\pi(\GhM))}
		\leq C\|f\|_{C^1(\G)}\|g\|_{W^{1,1}(\G)}.
	\]
	We apply \cref{cor:weighted-diff-of-area} for even $\kg$ and the standard estimate \eqref{eq:lem-surface-elements2} for odd $\kg$, to estimate
	\begin{align*}
		\left|\int_{\M} \bar{f}\cdot \bar{g}\;(\mu[\pi\circ \Phih]-\mGh)\,\dM\right|
		&\leq C h^{\widehat{\kg}}\|g_3\|_{W^{1,1}(\pi(\GhM))}
		\leq Ch^{\widehat{\kg}} \|f\|_{C^{1}(\G)}\|g\|_{W^{1,1}(\G)}.
	\end{align*}
	For the first term $(i)$ we estimate
	\begin{multline*}
		\left| \int_{\M} (f\circ \pi\circ \Phih -\bar{f}) \cdot \bar{g} \;\mu[\pi\circ \Phih]\,\dM\right|\\
		\leq \left| \int_{\M} (f\circ \pi\circ \Phih -\bar{f}) \cdot \bar{g} \;\mG\,\dM\right| + \left| \int_{\M} (f\circ \pi\circ \Phih -\bar{f}) \cdot \bar{g} \;(\mu[\pi\circ \Phih]-\mG)\,\dM\right|.
	\end{multline*}
	The first term is estimated using either the standard lifting estimate
	\eqref{eq:standard-interpolation-1} for odd $\kg$, or
	\cref{lem:F-L-F-pi-estimate} for even $\kg$. The second term is estimated
	with \cref{lem:surface-elements} and \eqref{eq:standard-interpolation-1}.
	Thus, we obtain
	\begin{align*}
		\left| \int_{\M} (f\circ \pi\circ \Phih -\bar{f}) \cdot \bar{g} \;\mu[\pi\circ \Phih]\,\dM\right|
		&\leq Ch^{\widehat{\kg}} \|f\|_{C^{2}(\G)}\|\bar{g}\|_{W^{1,1}(\M)} + Ch^{2\kg+1} \|f\|_{C^{1}(\G)}\|\bar{g}\|_{L^{1}(\M)}.
	\end{align*}
	Combining all the estimates yields the assertion.
\end{proof}

In the following two estimates, $\gradG f\cdot g$ and
$\gradGh f_h\cdot(g\circ L)$ denote multiplication of the vector-valued
gradients by the scalar weight, and the norm is the Euclidean norm of the
resulting vector-valued integral.

\begin{theorem}\label{thm:interpolation_estimate_1st_derivatives_standard}
	Assume the hypotheses of this subsection, including the boundary hypothesis. Let $f\in C^{\ku+1}(\GM;\R)$ be a scalar function. Let $f_h \colonequals \II^{\ku}_{h,\kg}(f)$ be the interpolation of $f$, and $g\in W^{1,1}(\GM)$.
	Then
	\begin{equation}
		\left\| \int_{\GM} \gradG f\cdot g \,\dG - \int_{\GhM} \gradGh f_h\cdot (g\circ L) \dGh \right\|
		\leq Ch^{\min\{\ku+1,\kg+1\}} \|f\|_{C^{\ku+1}(\GM)} \|g\|_{W^{1,1}(\GM)},
	\end{equation}
\end{theorem}
\begin{proof}
	We estimate
  \begin{multline*}
    \left\| \int_{\GM} \gradG f\cdot g \,\dG - \int_{\GhM} \gradGh f_h\cdot (g\circ L)\,\dGh \right\|\\
   	\begin{aligned}
   		&\leq \left\| \int_{\GM} \gradG (f-f_h\circ L^{-1})\cdot g \,\dG \right\| + \left\| \int_{\GM} \gradG (f_h \circ L^{-1}) \cdot g \,\dG - \int_{\GhM} \gradGh f_h\cdot (g\circ L)\,\dGh \right\|\\
			&= \left\| \int_{\M} \gradG (f-f_h\circ L^{-1})\circ \Phi\cdot \bar{g}\mG \,\dM \right\|
			+ \left\| \int_{\M} \left(\gradG (f_h \circ L^{-1})\circ \Phi\mG - \gradGh f_h\circ \PhihM\mGh \right)\bar{g}\dM \right\|.
    \end{aligned}
  \end{multline*}
	We set $\bar{g}\colonequals g\circ \Phi$ and $\bar{f}\colonequals f\circ\Phi$ and write with $\bar{f}_h\colonequals\IBar_h^\ku(\bar{f})$
	\begin{align*}
		\gradG(f-f_h\circ L^{-1}) \circ \Phi & = \gradM(\bar{f} - \bar{f}_h) \cdot (\DM\Phi)^+
	\end{align*}
	and proceed as in \cref{thm:interpolation_estimate_-1} with the estimate
	\begin{align}
		\left\| \int_{\M} \gradM(\bar{f} - \bar{f}_h)\cdot(\DM\Phi)^+\cdot\bar{g}\mG \,\dM \right\|
		\leq C h^{\ku+1}\|f\|_{C^{\ku+1}(\GM)}\| g\|_{W^{1,1}(\GM)}, \label{eq:estimate-df-dfh-DPhi-g_mu}
	\end{align}
	where we used \cref{lem:improved-first-derivative} with $\ell=1$ and the weight function $(\DM\Phi)^+\cdot \bar{g}\mG$ componentwise. For the remaining terms, we write analogously
	\begin{align*}
		\gradG (f_h \circ L^{-1})\circ \Phi
		&= \gradM\bar{f}_h\cdot(\DM\Phi)^+
		&\text{and}&&
		\gradGh f_h\circ \PhihM
		&= \gradM\bar{f}_h \cdot (\DM\PhihM)^+.
	\end{align*}
	For the difference we obtain
	\begin{multline}\label{eq:estimate-differences-of-first-order-derivatives}
		\int_{\M} \left(\gradG (f_h \circ L^{-1})\circ \Phi\mG - \gradGh f_h\circ \PhihM\mGh \right)\bar{g}\dM\\
		  = \int_{\M} \gradM\bar{f}_h\cdot(\DM\Phi)^+ (\mG -\mGh) \bar{g} \,\dM
		  + \int_{\M} \gradM\bar{f}_h\cdot\left((\DM\Phi)^+ - (\DM\PhihM)^+\right)\cdot\bar{g}\mGh\,\dM.
	\end{multline}

	To apply \cref{cor:weighted-diff-of-area} or
	\cref{lem:improved-first-derivative}, respectively, the weights must be
	sufficiently smooth. We therefore replace the discrete factor
	$\gradM\bar f_h$ by $\gradM\bar f$ and, where needed, replace $\mGh$ by
	$\mG$. The errors are controlled by
	\cref{lem:standard-flat-interpolation-estimate,lem:surface-elements,lem:pseudo-inverse-geometry-mapping}:
	\[
		\|\gradM(\bar f_h-\bar f)\|_{L^\infty(\M)}
		\leq Ch^\ku\|f\|_{C^{\ku+1}(\GM)},\quad
		\|\mGh-\mG\|_{L^\infty(\M)} +\|(\DM\Phi)^+-(\DM\PhihM)^+\|_{L^\infty(\M)}
		\leq Ch^\kg .
	\]
	For example, the second term in \eqref{eq:estimate-differences-of-first-order-derivatives} is written as
	\begin{equation*}
    \int_{\M} \gradM\bar{f}_h \cdot \left((\DM\Phi)^+ - (\DM\PhihM)^+\right)\bar{g}\mGh\,\dM
		= \int_{\M} \gradM\bar{f} \cdot \left((\DM\Phi)^+ - (\DM\PhihM)^+\right)\bar{g}\mG\,\dM
		+ \mathcal{O}(h^{\kg+\min\{\ku,\kg\}}).
	\end{equation*}

	The estimate of the first term follows by expressing the pseudo-inverse difference in terms of differences in the Jacobians, $\DM\Phi-\DM\Phih$, with a smooth enough prefactor to apply \cref{lem:improved-first-derivative} with $\ell=1$ componentwise and standard estimates \cref{lem:surface-approximation-estimates,lem:pseudo-inverse-geometry-mapping} for the remaining products of differences. This yields an order $\kg+1$.
	The term in
	\eqref{eq:estimate-differences-of-first-order-derivatives} containing
	$\mG-\mGh$ is treated in the same way. After replacing
	$\gradM\bar f_h$ by $\gradM\bar f$,
	\cref{cor:weighted-diff-of-area} for even $\kg$ and the standard estimate
	from \cref{lem:surface-elements} for odd $\kg$ give order $\kg+1$. The
	replacement error is of order $\ku+\kg$ by
	\cref{lem:standard-flat-interpolation-estimate}. Since $\ku,\kg\geq1$,
	these bounds are of order
	$\min\{\ku+1,\kg+1\}$.
\end{proof}

Similar to the improved estimates for first derivatives on the macro element \cref{lem:improved-first-derivative}, we can look at the surface gradients multiplied with a weight function of higher regularity, and obtain an improved estimate:

\begin{theorem}\label{thm:interpolation_estimate_1st_derivatives}
	Assume the hypotheses of this subsection, including the boundary hypothesis. Let $f\in C^{\ku+2}(\GM;\R)$ be a scalar function. Let $f_h \colonequals \II^{\ku}_{h,\kg}(f)$ be the interpolation of $f$, and $g\in W^{2,1}(\GM)$.
	Then
	\begin{equation}
		\left\| \int_{\GM} \gradG f\cdot g \,\dG - \int_{\GhM} \gradGh f_h\cdot (g\circ L) \dGh \right\|
		\leq Ch^{\min\{\widehat{\ku} , \widehat{\kg}\}} \|f\|_{C^{\ku+2}(\GM)} \|g\|_{W^{2,1}(\GM)},
	\end{equation}
	with $\widehat{\ku}$ and $\widehat{\kg}$ as in \eqref{eq:even-odd-order}.
\end{theorem}
\begin{proof}
  The proof follows mostly the one of \cref{thm:interpolation_estimate_1st_derivatives_standard}. We only highlight the differences:
  We replace the estimate \eqref{eq:estimate-df-dfh-DPhi-g_mu} by
	  \begin{align}
			\left\| \int_{\M} \gradM(\bar{f} - \bar{f}_h)\cdot(\DM\Phi)^+\cdot\bar{g}\mG \,\dM \right\|
			\leq C h^{\widehat{\ku}}\|f\|_{C^{\ku+2}(\GM)}\| g\|_{W^{2,1}(\GM)}, \label{eq:estimate-df-dfh-DPhi-g_mu_improved}
	\end{align}
	where we used \cref{lem:improved-first-derivative} with the weight function $(\DM\Phi)^+\cdot \bar{g}\mG$ componentwise.

	For the estimate of the integration element difference and Jacobian pseudo-inverse difference in \eqref{eq:estimate-differences-of-first-order-derivatives}, we use the same decompositions as in the preceding proof. For even $\kg$, the smooth-weight terms are estimated by \cref{cor:weighted-diff-of-area} and \cref{lem:improved-first-derivative} with $\ell=2$. For odd $\kg$, the standard estimates from \cref{lem:surface-elements,lem:pseudo-inverse-geometry-mapping} give order $\kg+1=\widehat{\kg}$. The replacement errors again have orders $\ku+\kg$ or $2\kg$ by \cref{lem:standard-flat-interpolation-estimate,lem:surface-elements,lem:pseudo-inverse-geometry-mapping}, and these are at least of order $\min\{\widehat{\ku},\widehat{\kg}\}$.
	\end{proof}

\begin{remark}
	One may formulate analogues of
	\cref{thm:interpolation_estimate_1st_derivatives_standard,thm:interpolation_estimate_1st_derivatives}
	with the lifting $L$ replaced by the closest-point projection $\pi$.
	In this case $\GM$ is replaced by $\pi(\GhM)$, and the functions have to
	have the corresponding regularity on this subsurface. By
	\cref{lem:regularity-interpolated-geometry}, the map
	$\pi\circ\Phih\colon\M\to\pi(\GhM)$ is uniformly regular for $h$
	sufficiently small, so the same pullback arguments are available.
	However, the leading interpolation term corresponding to
	\eqref{eq:estimate-df-dfh-DPhi-g_mu_improved} has the same estimate as in
	the case of $L$. The additional terms caused by replacing $L$ by $\pi$ are
	controlled by the lifting estimate \eqref{eq:standard-interpolation-1} and
	by \cref{lem:F-L-F-pi-estimate}. Thus the final order is the same as for
	the estimates with $L$, and no analogue of the extra improvement in
	\cref{thm:interpolation_estimate_0} is obtained. We therefore keep the
	statements with $L$ only.
\end{remark}

\subsection{Global interpolation errors over approximated surfaces}

So far, all estimates have been derived on a single macro patch, i.e., on subsets $\GM$ or $\GhM$ associated with one macro element $\M$. We now assume that $\G$ is represented by finitely many continuously glued macro-element parametrizations $\Phi_\M$ and that $\Gh$ is obtained by the corresponding elementwise polynomial parametrizations $\PhihM$. Whenever a global estimate uses one of the improved local estimates, the hypotheses of the preceding subsection are imposed on every macro element. For estimates involving \cref{lem:improved-first-derivative}, the additional boundary hypothesis is imposed on every macro element as well. The global estimates follow by summing the local estimates over the fixed set of macro patches. If a local estimate already uses a norm on the whole surface $\G$, this norm also controls the sum, since the number of macro patches is fixed and the relevant closest-point neighbourhoods have uniformly finite overlap.

For Sobolev norms of scalar, vector, or tensor fields on $\G$, derivatives are understood as covariant derivatives. For instance, $\|g\|_{W^{2,1}(\G)}$ contains the $L^1$-norms of $g$, $\gradG g$, and $\gradG^{(2)}g$.

The first estimate is a global version of the local distance estimate \cref{lem:dist_estimate}.

\begin{corollary}\label{cor:dist_estimate_global}
  Let $f\in W^{1,1}(\G)$, and let $\Phi\colon\GBar\to\G$ be a continuous and patchwise $C^{\kg+2}$ parametrization of $\G$. Let $\Phih$ be its elementwise polynomial approximation. Assume that $\kg$ is even and that the hypotheses of the preceding subsection hold on every local patch $\M\in\TriRef$. Let $\mathfrak{L}\colonequals L$ or $\mathfrak{L}\colonequals \pi$. Then
  \begin{align}
    \Big|\int_{\Gh} \dist (f\circ \mathfrak{L})\,\dGh \Big| &\leq C h^{\kg+2}\Norm{f}{W^{1,1}(\G)},
  \end{align}
  with $\dist\colon U_{\delta}(\G)\to\R$ the surface signed-distance function.
\end{corollary}
\begin{proof}
	For each macro element $\M\in \TriRef$, the uniform regularity of
	$\mathfrak L|_{\GhM}$ gives
	$f\circ \mathfrak L\in W^{1,1}(\GhM)$ and
	\[
		\|f\circ \mathfrak L\|_{W^{1,1}(\GhM)}
		\leq C_\M\|f\|_{W^{1,1}(\mathfrak L(\GhM))}.
	\]
	Thus, we can apply \cref{lem:dist_estimate} on each macro element, and obtain
	\begin{align*}
    \Big|\int_{\Gh} \dist (f\circ \mathfrak{L})\,\dGh \Big|
		&= \Big|\sum_{\M\in \TriRef} \int_{\GhM} \dist (f\circ \mathfrak{L})\,\dGh \Big|
		\leq  \sum_{\M\in \TriRef}  C_{\M} h^{\kg+2}\Norm{f\circ \mathfrak{L}}{W^{1,1}(\GhM)}.
  \end{align*}
	For $\mathfrak L=L$, the sets $L(\GhM)=\GM$ form the exact macro patches.
	For $\mathfrak L=\pi$, the sets $\pi(\GhM)$ cover $\G$ with uniformly
	finite overlap. Since the number of macro elements is fixed, the local
	constants are bounded uniformly and the sum of the local norms is
	controlled by $\|f\|_{W^{1,1}(\G)}$.
\end{proof}

\begin{corollary}\label{cor:global_interpolation_estimate}
	Let $\Phi\colon\GBar\to\G$ be continuous and patchwise sufficiently smooth, and let $\Gh$ be parametrized elementwise over $\T\in\TriBar_h$. Let $f\in C^{\ku+2}(\G;\R)$ be a scalar function and let $f_h \colonequals \II_{h,\kg}^\ku f\in C^0(\Gh;\R)$ be the interpolation of $f$. Assume that the local hypotheses of the corresponding macro-patch estimates hold on every macro element, including the boundary hypothesis where it is required. Then the following global estimates hold, with $\widehat{\ku}$, $\widehat{\kg}$, and $\widehat{\kg}^{(\ell)}$ as in \eqref{eq:even-odd-order}. Let $g_\ell\in W^{\ell,1}(\G)$, $\ell\in\{1,2\}$, then

	\begin{align}
		\Big| \int_{\G} f\cdot g_\ell \,\dG - \int_{\Gh} f_h\cdot (g_\ell\circ L)\,\dGh \Big|
		&\leq Ch^{\min\{\widehat{\ku},\widehat{\kg}^{(\ell)}\}} \|f\|_{C^{\widehat{\ku}}(\G)} \|g_\ell\|_{W^{\ell,1}(\G)},
		\label{eq:global-interpolation-L} \\
		\Big| \int_{\G} f\cdot g_1 \,\dG - \int_{\Gh} f_h\cdot (g_1\circ \pi)\,\dGh \Big|
		&\leq Ch^{\min\{\widehat{\ku}, \widehat{\kg}\}} \|f\|_{C^{\widehat{\ku}}(\G)} \|g_1\|_{W^{1,1}(\G)},
		\label{eq:global-interpolation-pi} \\
		\Big\| \int_{\G} \gradG f\cdot g_1 \,\dG - \int_{\Gh} \gradGh f_h\cdot (g_1\circ L)\,\dGh \Big\|
		&\leq Ch^{\min\{\ku+1,\kg+1\}} \|f\|_{C^{\ku+1}(\G)} \|g_1\|_{W^{1,1}(\G)},
		\label{eq:global-interpolation-first-derivative-standard} \\
		\Big\| \int_{\G} \gradG f\cdot g_2 \,\dG - \int_{\Gh} \gradGh f_h\cdot (g_2\circ L)\,\dGh \Big\|
		&\leq Ch^{\min\{\widehat{\ku}, \widehat{\kg}\}} \|f\|_{C^{\ku+2}(\G)} \|g_2\|_{W^{2,1}(\G)}.
		\label{eq:global-interpolation-first-derivative-improved}
	\end{align}
	The constant $C$ may depend on the macro parametrizations but it is independent of $h$.
\end{corollary}
\begin{proof}
	The continuity of $\Phi$ and the matching Lagrange nodes on macro-element
	interfaces imply that $\G$ and $\Gh$ decompose into macro patches up to
	codimension-one interfaces. Therefore each global integral with the lifting
	$L$ is the sum of the corresponding local integrals over $\GM$ and $\GhM$.
	For the estimate with $\pi$, the local exact patches are $\pi(\GhM)$.
	These patches cover $\G$ with uniformly finite overlap by the regularity of
	the closest-point projection on $\Gh$. Applying
	\cref{thm:interpolation_estimate_-1,thm:interpolation_estimate_0,thm:interpolation_estimate_1st_derivatives_standard,thm:interpolation_estimate_1st_derivatives}
	patchwise gives the estimates on every macro patch.

	It remains only to collect constants. Since the number of macro elements is fixed, the local constants can be replaced by their maximum. The sums of the local Sobolev norms are bounded by the corresponding global norms on $\G$, and the estimates that already use norms on the whole surface are controlled directly by the finite number of patches and the finite overlap of the closest-point neighbourhoods. This proves all assertions.
\end{proof}

In \cref{sec:appendix-higher-order-derivatives} we have extended these results to higher-order lifted derivative estimates, including a derivative-order dependent parity. The arguments in the proofs follow the same structure as
for the first derivative estimates, but include a lot more extra differences that all must be handled by the improved or standard estimates.

\section{Application to geometric quantities}\label{sec:application}

In surface finite element discretizations and their a priori error estimates, a crucial step is usually to estimate the gap between the geometric quantities of the discrete surface $\Gh$ and those of the continuous surface $\G$. In this section we show that, on even geometries, some of these estimates can be improved using the techniques described in \cref{sec:interpolation_errors}. Although not always reported in articles, it seems to be a common observation among researchers working with higher order surface finite elements that even geometries can produce better discretization errors than expected from the theory. We comment on one of these cases below.

We use the notation for normals, projections, surface derivatives, divergence,
and curvatures introduced in \cref{sec:discrete surface}. The estimates below
concern the improved bounds for even geometry order. Therefore, on every macro
element $\M$, we assume the local hypotheses from
\cref{sec:interpolation_errors}: the induced flat triangulation
$\TriBar_{h,\M}$ is symmetric in the sense of
\cref{def:symmetric-triangulation}, the discrete geometry is the elementwise
degree-$\kg$ interpolation of the exact parametrization, and $h$ is
sufficiently small. The symmetry assumption always refers to the flat
triangulation over the macro element. No separate symmetry condition is imposed
on the curved surface elements.

The following so-called non-standard estimate improves the estimates in \citet[Lemma 4.2]{HLL2020Analysis}, which are used in the a priori estimates for surface finite element discretizations of the vector Laplace equation.

\begin{lemma}[Non-standard estimates for even $\kg$]\label{lem:non-standard-estimates}
  Let $\bm{v}\in W^{2,1}(\G;\R^3)$ be a vector field and $\kg\geq 2$ even. For $\mathfrak{L}\colonequals L$ or $\mathfrak{L}\colonequals\pi$ we have the estimates
  \begin{align}
    \left|\Inner{\PPh(\nn\circ\pi)}{\bm{v}\circ \mathfrak{L}}{\Gh}\right| &\leq C h^{\kg+2} \|\bm{v}\|_{W^{2,1}(\G;\R^3)},  \label{eq:Pnh} \\
    \left|\Inner{\nn\circ\pi-\nnh}{\bm{v}\circ \mathfrak{L}}{\Gh}\right| &\leq C h^{\kg+2}\|\bm{v}\|_{W^{2,1}(\G;\R^3)}.  \label{eq:n-nh}
  \end{align}
\end{lemma}
\begin{proof}
  For the first estimate \eqref{eq:Pnh}, we follow \citet[Lemma 4.2]{HLL2020Analysis} by representing $\nn\doteq\nn\circ\pi$ as the gradient of a distance function and then performing integration by parts to obtain
  \begin{equation*}
    \Inner{\PPh\nn}{\bm{v}\circ \mathfrak{L}}{\Gh} =  \Inner{\gradGh\dist}{\bm{v}\circ \mathfrak{L}}{\Gh}
		= -\sum_{\T\in \TriBar_h}\Inner{\dist}{g_h}{\Phih(\T)} + \Inner{\dist[\nu]}{\bm{v}\circ \mathfrak{L}}{\mathcal{E}_h},
	\end{equation*}
	where $g_h\colonequals \MeanCurvatureh\nnh\cdot \bm{v}\circ \mathfrak{L} + \gradGh\cdot\bm{v}\circ \mathfrak{L}$, $\mathcal{E}_h$ denotes the union of interior edges of the surface triangulation, and $[\nu]$ the jumps of the conormals of adjacent surface elements over the common edge.
	In order to apply \cref{cor:dist_estimate_global} to estimate the first term, we replace $g_h$ by $g\circ \mathfrak{L}$, where $g\colonequals \MeanCurvature\; \nn\cdot \bm{v} + \gradG\cdot\bm{v}\in W^{1,1}(\G)$, i.e., we write
	\begin{equation*}
    \Inner{\PPh\nn}{\bm{v}\circ \mathfrak{L}}{\Gh}
		= \underbrace{-\sum_{\T\in \TriBar_h}\Inner{\dist}{g  \circ \mathfrak{L}}{\Phih(\T)}}_{(i)} +\underbrace{\sum_{\T\in \TriBar_h}\Inner{\dist}{g  \circ \mathfrak{L}-g_h}{\Phih(\T)}}_{(ii)} + \underbrace{\Inner{\dist[\nu]}{\bm{v}\circ \mathfrak{L}}{\mathcal{E}_h}}_{(iii)}.
	\end{equation*}
	As $\|g\|_{W^{1,1}(\G)}\leq C \|\bm{v}\|_{W^{2,1}(\G;\R^3)}$, we apply \cref{cor:dist_estimate_global} to estimate $(i)$ by
	\begin{equation*}
 		\left|\Inner{\dist}{g}{\Gh}\right|
 		\leq C h^{\kg+2} \|g\|_{W^{1,1}(\G)}
 		\leq C h^{\kg+2} \|\bm{v}\|_{W^{2,1}(\G;\R^3)}.
  \end{equation*}
	For $(ii)$ we note that the difference $g \circ \mathfrak{L}-g_h$ can be expressed in differences of the form $\bm{n}\circ\mathfrak{L}- \nnh$ and $\MeanCurvature\circ\mathfrak{L} - \MeanCurvatureh$. Thus, using standard geometric estimates and \cref{lem:interpolation-error-estimates}, we can estimate $(ii)$ by
	\begin{equation*}
		\left|\sum_{\T\in \TriBar_h}\Inner{\dist}{g  \circ \mathfrak{L}-g_h}{\Phih(\T)}\right|\leq \|\dist\|_{L^{\infty}(\Gh)}\|g  \circ \mathfrak{L}-g_h\|_{L^1(\Gh)}\leq Ch^{2\kg}\|\bm{v}\|_{W^{1,1}(\G;\R^3)}.
	\end{equation*}
  As the jumps satisfy the estimate $\|[\nu]\|_{L^{\infty}(\mathcal{E}_h)}\leq C h^{\kg}$, by the trace inequality, the term $(iii)$ satisfies
  \begin{align*}
		\Inner{\dist[\nu]}{\bm{v}\circ \mathfrak{L}}{\mathcal{E}_h}
 		&\leq \|[\nu]\|_{L^{\infty}(\mathcal{E}_h)}\|\dist\|_{L^{\infty}(\mathcal{E}_h)}\|\bm{v}\circ \mathfrak{L}\|_{L^{1}(\mathcal{E}_h)}\\
 		&\leq C h^{2\kg+1}\left(h^{-1}\|\bm{v}\circ \mathfrak{L}\|_{L^{1}(\Gh)}+\|\gradGh(\bm{v}\circ \mathfrak{L})\|_{L^{1}(\Gh)}\right)\\
 		&\leq C h^{2\kg}\|\bm{v}\|_{W^{1,1}(\G;\R^3)}.
  \end{align*}
  As $2\kg\geq \kg+2$ for $\kg\geq 2$, we obtain \eqref{eq:Pnh}.

  For the second estimate~\eqref{eq:n-nh}, we write
  \begin{align*}
 		\Inner{\nn-\nnh}{\bm{v}\circ \mathfrak{L}}{\Gh} &= \Inner{\nn-\nnh}{\PPh(\bm{v}\circ \mathfrak{L}) + (\bm{v}\circ \mathfrak{L}\cdot \nnh)\,\nnh}{\Gh}\\
 		&=  \Inner{\PPh\nn}{\bm{v}\circ \mathfrak{L}}{\Gh} -\Inner{\frac{1}{2}|\nn-\nnh|^2\nnh}{\bm{v}\circ \mathfrak{L}}{\Gh}.
  \end{align*}
  Thus, we obtain the assertion using $\|\nn-\nnh\|_{L^{\infty}(\Gh)}^2\leq C h^{2\kg}\leq C h^{\kg+2}$ and \eqref{eq:Pnh}.
\end{proof}

\begin{remark}
	Lifting the normal vector $\nn$ by $L$ instead of $\pi$ in \cref{lem:non-standard-estimates} results in the same estimates, by employing \cref{lem:F-L-F-pi-estimate}.
\end{remark}

Although these estimates are of higher order in $h$ than those in \citet[Lemma 4.2]{HLL2020Analysis}, they cannot be used directly to infer higher order discretization error estimates for the vector Laplace equation. The main obstruction is that the higher order derivatives required of the vector field $\bm{v}$ cannot be handled appropriately in the existing proofs.

\begin{lemma}\label{lem:weingarten_est}
  For 2-tensor fields $\bm{A}\in W^{1,1}(\G;\R^{3\times 3})$ and even $\kg\geq 2$, with $\mathfrak{L}\colonequals L$ or $\mathfrak{L}\colonequals\pi$, we have the estimate
  \begin{align}
    \left|\Inner{\gradG\nn}{\bm{A}}{\G} - \Inner{\gradGh\nnh}{\bm{A} \circ \mathfrak{L}}{\Gh}\right| &\leq C h^{\kg} \|\bm{A}\|_{W^{1,1}(\G;\R^{3\times 3})}. \label{eq:dn-dnh}
  \end{align}
\end{lemma}
\begin{proof}
	Note that $\Weingarten=-\gradG\nn$ is a representation of the Weingarten map. We use the local parametrizations $\Phi$ to express this. For $\xM\in \M$ and $\xG=\Phi(\xM)\in \G$, we write
	the second fundamental form as
	\begin{align*}
		II(\xM) = \begin{pmatrix}
			\partial_1^2\Phi(\xM) \cdot  \nn(\Phi(\xM)) & \partial_1\partial_2\Phi(\xM) \cdot  \nn(\Phi(\xM))\\
			\partial_1\partial_2\Phi(\xM) \cdot  \nn(\Phi(\xM)) & \partial_2^2\Phi(\xM) \cdot  \nn(\Phi(\xM))
		\end{pmatrix}.
	\end{align*}
	The Weingarten map is then given by
	\[
		\gradG\nn(\Phi(\xM)) = - (\DM\Phi(\xM)^{+})^\top II \DM\Phi(\xM)^{+},
	\]
	where $\DM\Phi(\xM)^{+}$ denotes the left pseudo-inverse of $\DM\Phi$, as before.

	Replacing all occurrences of $\Phi$ by $\Phih$, we obtain the discrete second fundamental form $II_h$, and the Weingarten map $\Weingartenh$ at the point $\Phih(\xM)$.
	Now we can write patchwise
	\begin{align}\label{eq:weingarte_aux_eq}
		\big(\gradGh\nnh &,\,\bm{A} \circ L\big)_{\GhM} -\Inner{\gradG\nn}{\bm{A}}{\GM}\notag \\
		& = \Inner{(\DM\Phi(\xM)^{+})^\top II \DM\Phi(\xM)^{+}}{\bm{A} \circ \Phi\,\mG}{\M} - \Inner{((\DM\Phih)^{+})^\top II_{h} (\DM\Phih)^{+}}{\bm{A}\circ\Phi\,\mGh}{\M}\notag \\
		& = \underbrace{\Inner{(\DM\Phi^{+})^{\top}(II -II_h)\DM\Phi^{+}}{\bm{A} \circ \Phi\,\mG}{\M}}_{(i)}\notag\\
		&\quad + \underbrace{\Inner{(\DM\Phi^{+})^{\top} II_h \DM\Phi^{+}}{\bm{A} \circ \Phi\,\mG}{\M}
		- \Inner{((\DM\Phih)^{+})^\top II_h (\DM\Phih)^{+}}{\bm{A} \circ \Phi\,\mGh}{\M}}_{(ii)}
	\end{align}
	The second term $(ii)$ is of order $h^\kg$ by standard estimates. More explicitly, the differences of area elements and pseudo-inverses are controlled by \cref{lem:surface-elements,lem:pseudo-inverse-geometry-mapping}, and all remaining factors are uniformly bounded. The first term $(i)$ is dominated by the difference in the second fundamental form. We write
	\begin{align*}
		II - II_h & = \nn(\Phi(\xM))^i \DM^2\Phi^i(\xM) - \nnh(\Phih(\xM))^i \DM^2(\Phih)^i(\xM)\\
		& = \nn(\Phi(\xM))^i \left(\DM^2\Phi^i(\xM)-\DM^2(\Phih)^i(\xM)\right) + \left(\nn(\Phi(\xM))^i - \nnh(\Phih(\xM))^i\right) \DM^2(\Phih)^i(\xM).
	\end{align*}
	Again, the second term is of order $h^\kg$, using \cref{lem:surface-approximation-estimates} together with the uniform boundedness of $\DM^2\Phih$. The first term contains the difference of second order derivatives of $\Phi$ and $\Phih$ with the smooth weight function $\nn\circ\Phi\in W^{1,1}(\M)$. If we insert these representations and estimates back into \eqref{eq:weingarte_aux_eq}, we can apply \cref{cor:weighted-flat-interpolation-estimate} to obtain the estimate on each macro element. We then sum over all macro elements. This gives the assertion for $\mathfrak{L}=L$.

	For $\mathfrak{L}=\pi$, we use a triangle inequality argument and estimate
	\begin{align*}
	   \left|\Inner{\gradGh\nnh}{\bm{A} \circ L-\bm{A} \circ \pi}{\Gh} \right| &\leq C h^{\kg+1} \|\bm{A}\|_{W^{1,1}(\G;\R^{3\times 3})}
	\end{align*}
	using \cref{lem:F-L-F-pi-estimate}.
\end{proof}

\begin{corollary}\label{cor:gauss-curvature-estimate}
  For $\bm{v},\bm{w}\in W^{1,2}(\G;\R^3)$ and even $\kg\geq 2$, with $\mathfrak{L}\colonequals L$ or $\mathfrak{L}\colonequals\pi$, we have
  \begin{equation}\label{eq:Gauss-curv-estimate}
		\left|\Inner{\GaussCurvature \bm{v}}{\bm{w}}{\G} - \Inner{\GaussCurvatureh \bm{v}\circ \mathfrak{L}}{\bm{w}\circ \mathfrak{L}}{\Gh} \right| \leq C h^{\kg} \|\bm{v}\|_{W^{1,2}(\G;\R^3)}\|\bm{w}\|_{W^{1,2}(\G;\R^3)},
  \end{equation}
  where $\GaussCurvature$ and $\GaussCurvatureh$ denote the Gaussian curvatures of $\G$ and $\Gh$, respectively.
\end{corollary}
\begin{proof}
  This is a direct consequence of \cref{lem:weingarten_est} using the identity $K = \frac{1}{2}(\tr(\gradG\nn)^2 - \tr((\gradG\nn)^2))$ and its discrete analogue. After expanding the difference of the quadratic terms, the part that is linear in $\gradG\nn-\gradGh\nnh$ is tested with tensor weights of the form $\bm{v}\otimes\bm{w}$ multiplied by smooth factors depending on $\gradG\nn$. These weights belong to $W^{1,1}(\G;\R^{3\times3})$ by Hölder's inequality. The remaining part is quadratic in the standard $L^\infty$ error for $\gradG\nn-\gradGh\nnh$ and is therefore of order $h^{2\kg-2}$, which is bounded by $h^\kg$ for $\kg\geq 2$.
\end{proof}

\begin{remark}
  In the analysis of discretization schemes for the surface Stokes equations, \citet[Remark 2.4]{HP2024Parametric} state the estimate \eqref{eq:Gauss-curv-estimate} as an observation. For certain schemes in that work, this new proof improves the a priori error theory to the observed orders of the investigated schemes.
\end{remark}

\section{Numerical experiments}\label{sec:numerical-experiments}

We now compare the convergence orders predicted in the preceding sections with numerical experiments. The main tests use the macro-element construction from the analysis together with symmetric red refinements of the reference mesh. Additional tests outside the scope of the main theorems are collected in \cref{sec:additional-numerical-experiments}.

All numerical experiments are implemented in the \textsc{Dune} framework \citep{BastianEtAl2008GenericGridInterfaceI,BastianEtAl2008GenericGridInterfaceII,BBD2021Dune,Dune2-10,EGMPS2025Concepts}. The discrete surfaces $\Gh$ are represented with \textsc{Dune-CurvedGrid} \citep{PS2020DuneCurvedGrid}, using \textsc{Dune-FoamGrid} \citep{SKSF2017DuneFoamGrid} for the reference grid and red refinement. The newest-vertex bisection tests in \cref{sec:additional-numerical-experiments} use \textsc{Dune-ALUGrid} \citep{ADKN2016DuneALUGrid,AK2017NewestVertex}. Integrals are computed with Gauss quadrature rules of sufficiently high degree. For $L^\infty$ norms, we use evaluation points combining the points from a Gauss quadrature rule and additional points on the boundary of the elements. All numerical tests perform 6 levels of uniform refinement steps of the grid and approximate the stepwise experimental order of convergence of the error quantities specified below.
The code and data needed to reproduce the results are available from Zenodo \citep{Praetorius2026OddGeometriesCode}.

\subsection{Test surface}

The test surface is a deformed sphere surface, see \cref{fig:flower_surface}. It is given as the zero level set of an implicit function $\phi\colon\R^3\to\R$. Let $(r,\theta,\varphi)$ be the spherical coordinates associated with $\xG(r,\theta,\varphi)=(r\sin\theta\cos\varphi, r\sin\theta\sin\varphi,r\cos\theta)\in\R^3$. We introduce a variation of a unit radius,
\[
	R(\theta,\varphi)
		\colonequals 1-0.1\cos(2\theta)-0.2\sin(\theta)^3\sin(3\varphi).
\]
This perturbation is related to low-order real spherical-harmonic modes. With this radius, we define an implicit representation of the deformed sphere surface, by
\[
	\phi(\xG(r,\theta,\varphi))\colonequals r^2-R(\theta,\varphi)^2,\qquad
	\G\colonequals\{\xG\in\R^3:\phi(\xG)=0\}.
\]
The normal vector, Weingarten map, and curvatures used in the computations are obtained from the first and second derivatives of $\phi$. As reference surface $\GBar$, we use a polytopal approximation of the sphere in the proximity $U_\delta$ of $\G$, which is generated with the \textsc{gmsh} meshing tool \citep{GR2009Gmsh}.

\begin{figure}[ht]
    \begin{subfigure}{0.24\linewidth}
        \includegraphics[width=.95\textwidth]{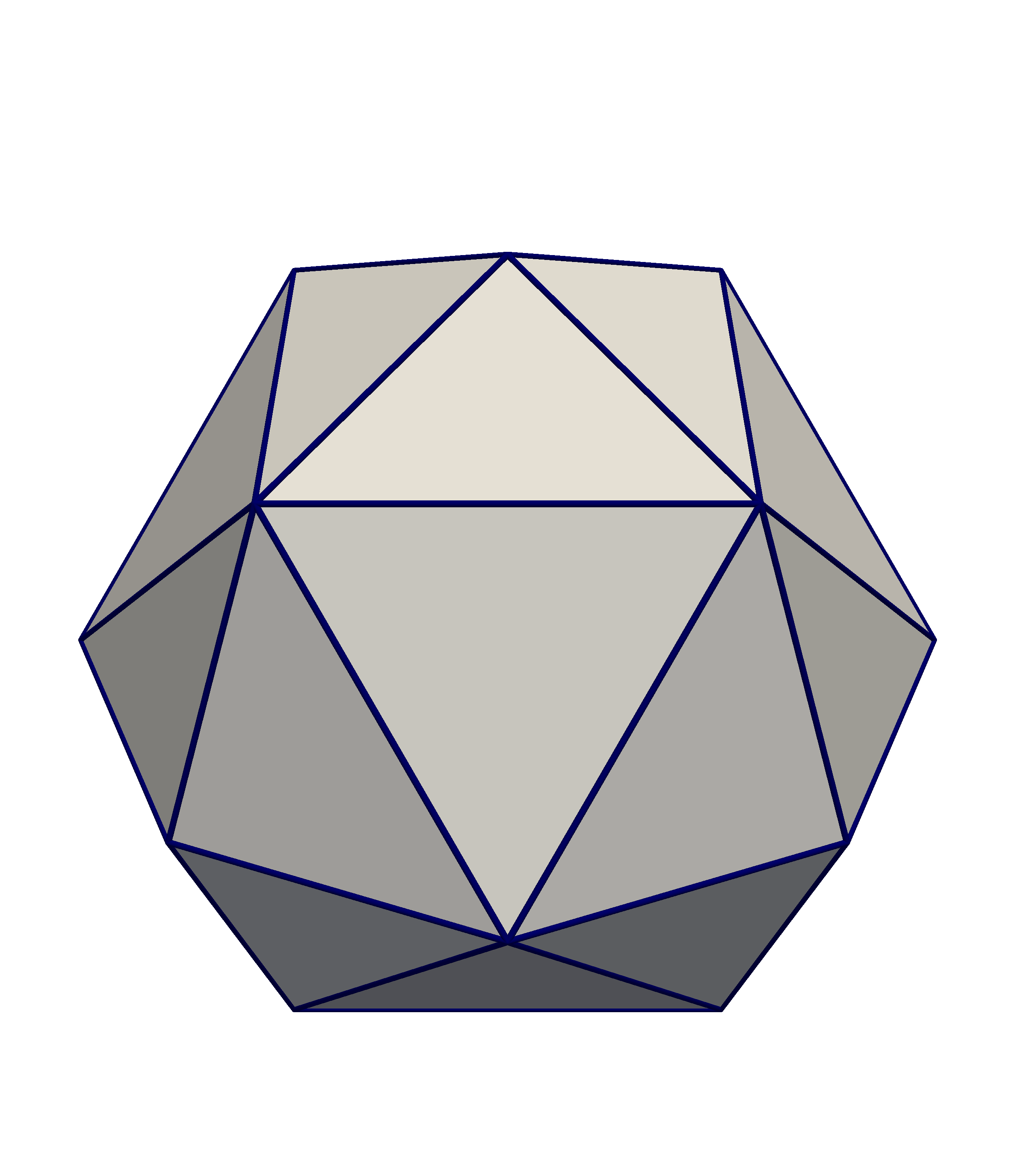}%
    \end{subfigure}\hfill%
    \begin{subfigure}{0.24\linewidth}
        \includegraphics[width=.95\textwidth]{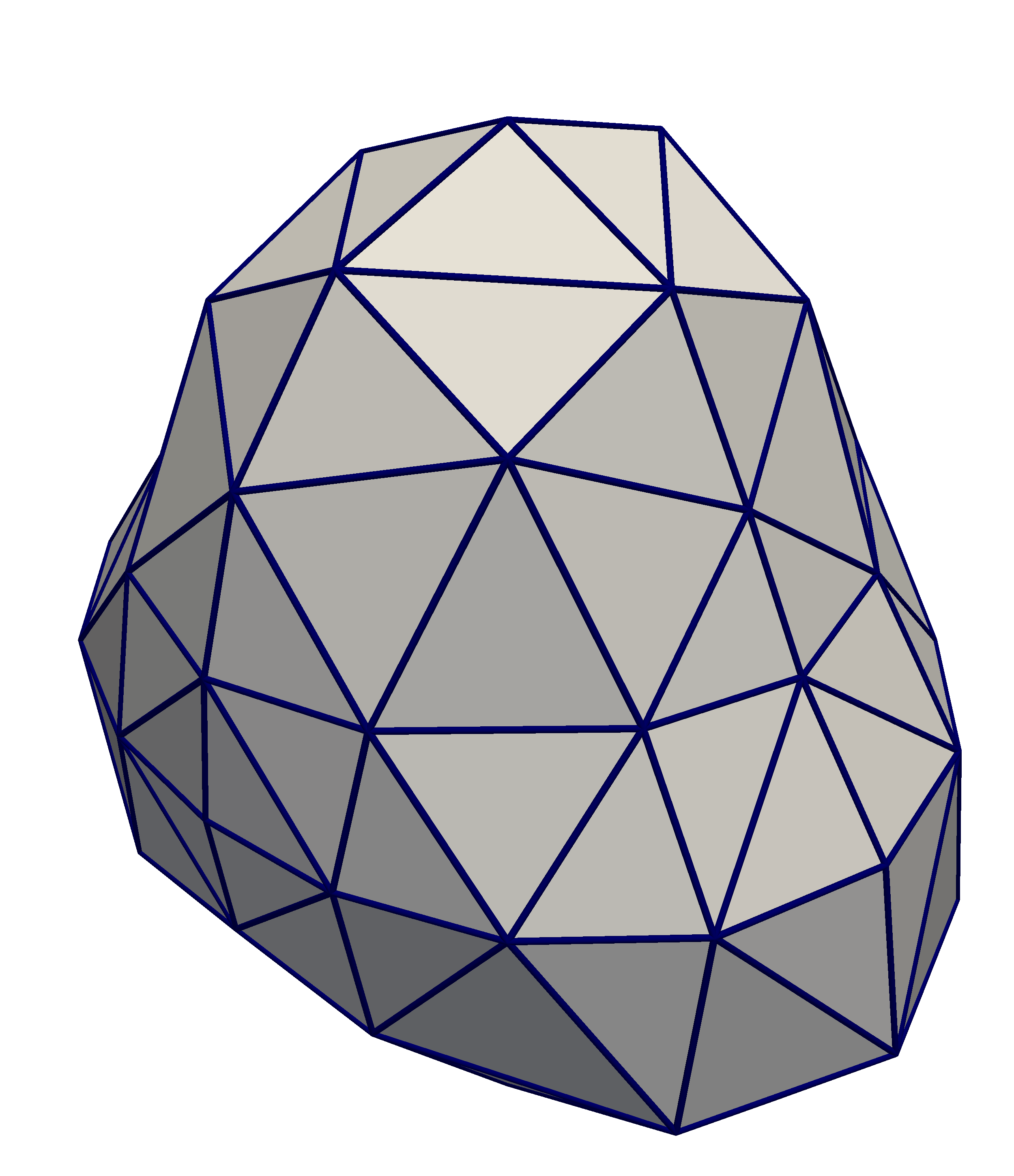}%
    \end{subfigure}\hfill%
    \begin{subfigure}{0.24\linewidth}
        \includegraphics[width=.95\textwidth]{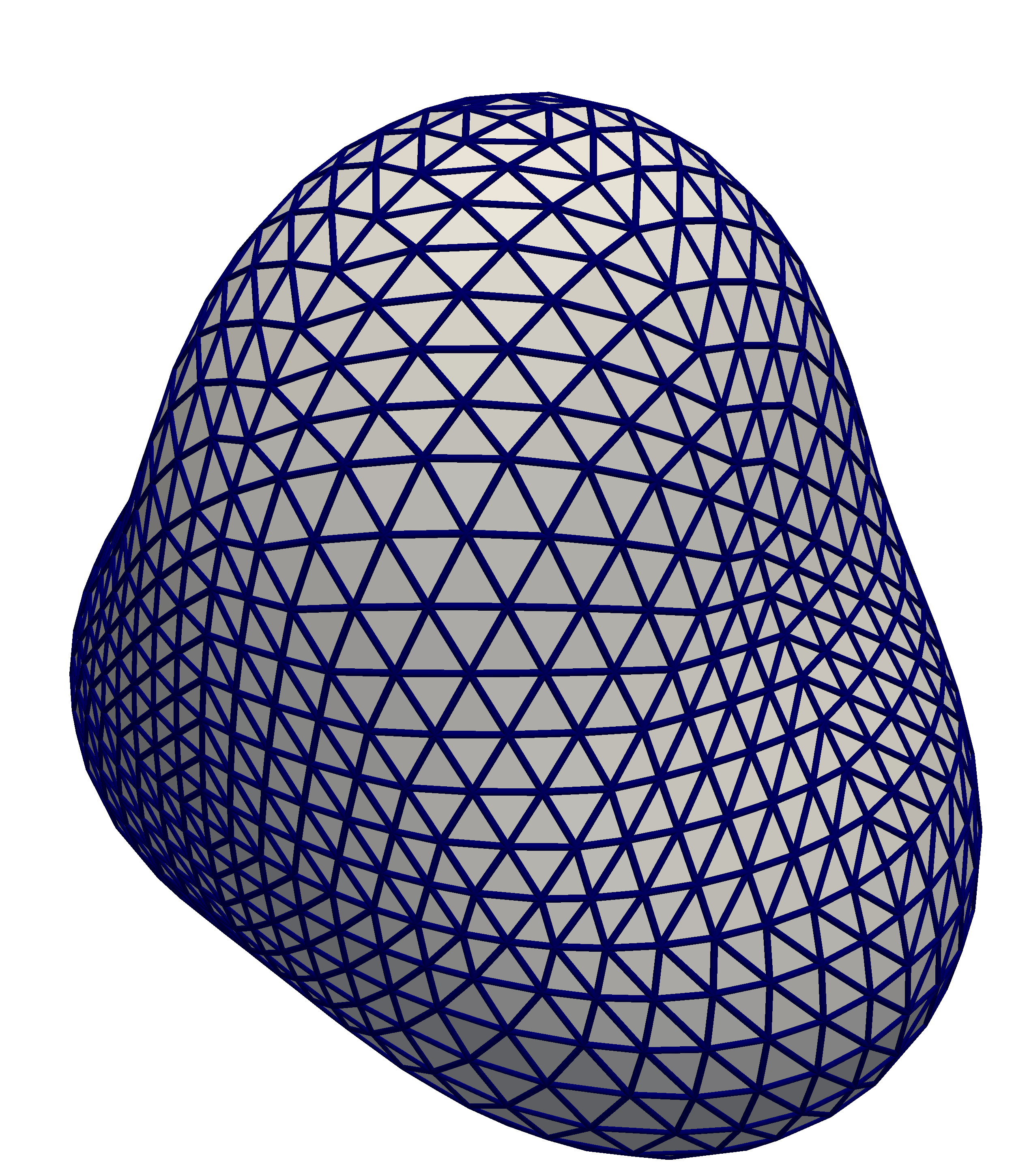}%
    \end{subfigure}\hfill%
    \begin{subfigure}{0.24\linewidth}
        \includegraphics[width=.95\textwidth]{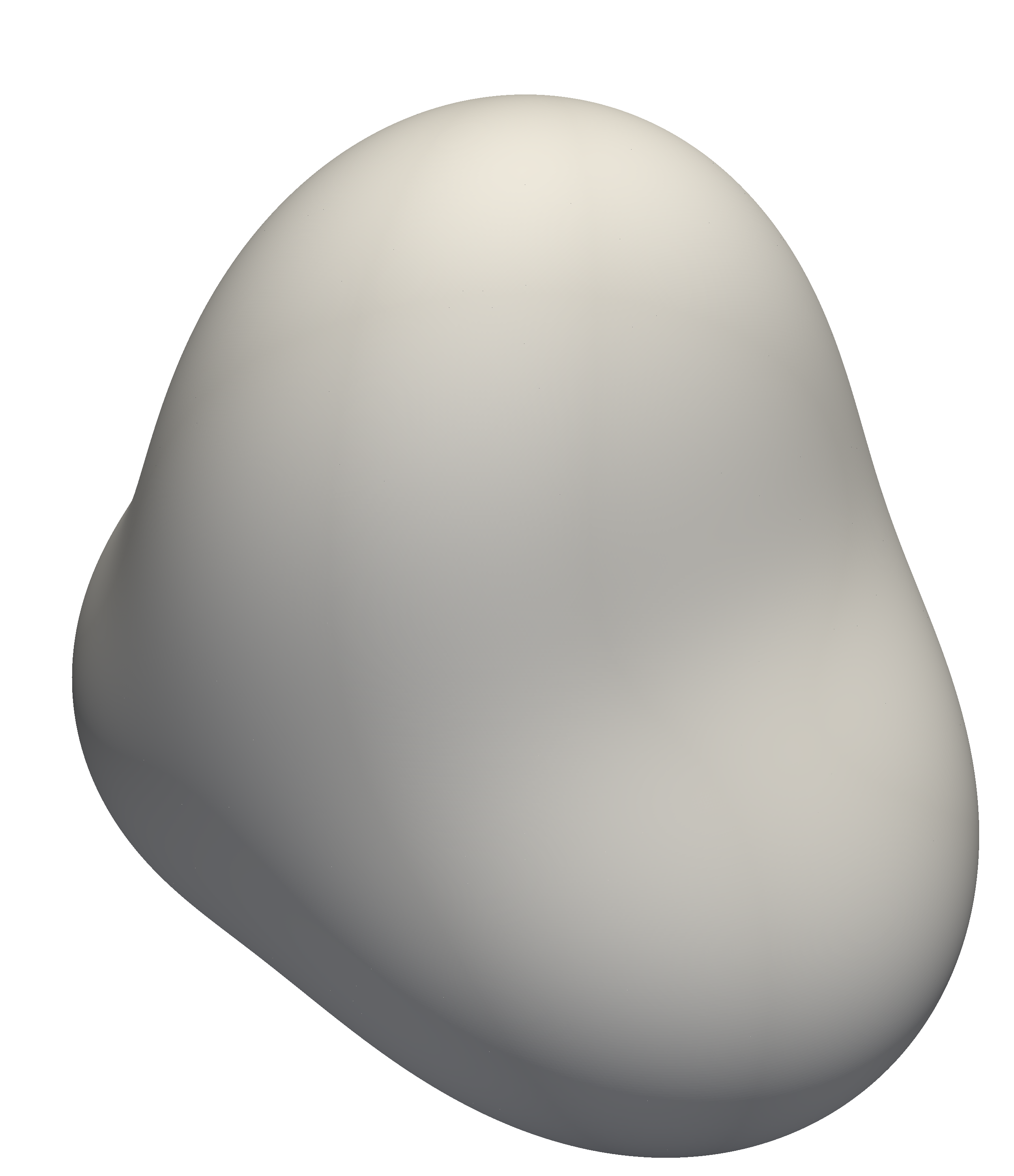}%
    \end{subfigure}
    \caption{Deformed sphere used as the test surface. Left: reference coarse sphere grid $\GBar$. Centre: piecewise linear discrete surfaces $\G_h^1$ with refinement levels 1 and 3. Right: visualization of the smooth surface $\G$.}\label{fig:flower_surface}
    \figurealt{From left to right, a coarse triangulated sphere, piecewise linear approximations of the deformed sphere after one and three refinement levels, showing increasingly fine and nearly uniform triangles, and the smooth deformed sphere, with threefold azimuthal and twofold polar variation.}
\end{figure}

From the points $\xM\in\GBar$ on the reference surface, we construct a smooth parametrization
\[
	\Phi\colon\GBar\to\R^3;\;\Phi(\xM) = \Phi(\xM(r,\theta,\varphi)) \colonequals \xG(R(\theta,\varphi),\theta,\varphi)
\]
in terms of the spherical coordinates associated to $\xM$ with the explicit projection of the radius $r\mapsto R$.

The closest-point projection $\pi(\xGh)$ for points $\xGh\in\Gh$ is defined by an iterative solution process for equation \eqref{eq:closest-point-projection}, see, e.g., \citet{DemlowDziuk2007Adaptive}.

\subsection{Interpolation estimates}

We first test the interpolation estimates from \cref{sec:interpolation_errors}.
On the flat macro element $\M$, let $\bar f\colonequals f|_{\M}$ and
$\bar f_h\colonequals\IBar_h^\kg\bar f$.
Rather than selecting one derivative associated with a fixed multi-index, we
compute the full derivative tensor $\DM^{(\kd)}(\bar f_h-\bar f)$ and contract it
with a tensor field $V^{(\kd)}$ of the same rank. More precisely, the measured
quantities are
\[
	E^{\mathrm{flat}}_\kd \colonequals \left|\sum_{\T\in\TriM_{h,\M}}\int_{\T}
	V^{(\kd)}:\DM^{(\kd)}(\bar f_h-\bar f)\,\dM\right|,
	\qquad \kd=0,1,2,3,
\]
where $:$ denotes the full Frobenius contraction, including ordinary
multiplication for $\kd=0$. Each contraction is a finite sum of the
multi-index quantities estimated in
\cref{cor:weighted-flat-interpolation-estimate,R:weighted-flat-interpolation-estimate-odd,lem:improved-first-derivative},
and hence has the same predicted order. We use
\begin{align*}
	V^{(0)}(x)      & = \sin(x_1 x_2 x_3), &
	V^{(1)}(x)_i    & = \sin(x_i),\\
	V^{(2)}(x)_{ij} & = \sin(x_i)\cos(x_j), &
	V^{(3)}(x)_{ijk}& = \sin(x_i)\cos(x_j)e^{x_k}.
\end{align*}
The function being interpolated, restricted to $\M$ in the flat test and to
$\G$ in the surface test below, is
\[
	f(x) = \sin(x_1) + \cos(2x_2) e^{3x_3}.
\]

\cref{tab:macro-interpolation-error-orders} reports the observed orders on one
macro element. The first derivative behaves like the function-value case, as
expected from \cref{lem:improved-first-derivative} with a sufficiently smooth
weight. The second- and third-derivative cases exhibit the derivative-dependent
parity predicted by
\cref{cor:weighted-flat-interpolation-estimate,R:weighted-flat-interpolation-estimate-odd}.

\begin{table}[ht]
	\centering
	\begin{tabular}{lrrrrrrll}
		\hline
		Derivative & $\order=1$ & $\order=2$ & $\order=3$ & $\order=4$ & $\order=5$ & $\order=6$ & $\order$ odd & $\order$ even \\
		\hline
		$\kd=0$ & $2.00$ & $4.01$ & $4.00$ & $6.01$ & $6.00$ & $7.93$ & $\order+1$ & $\order+2$ \\
		$\kd=1$ & $1.99$ & $4.00$ & $3.99$ & $6.00$ & $5.99$ & $8.00$ & $\order+1$ & $\order+2$ \\
		$\kd=2$ & $0.00$ & $1.99$ & $1.84$ & $3.97$ & $3.91$ & $5.87$ & $\order-1$ & $\order$ \\
		$\kd=3$ & $0.00$ & $0.00$ & $2.00$ & $1.97$ & $3.44$ & $3.82$ & $\order-1$ & $\order-2$ \\
		\hline
	\end{tabular}
	\caption{Asymptotic experimental orders of convergence for weighted interpolation-error integrals $E^{\mathrm{flat}}_\kd$ on the macro element patch $\M$. Compare \cref{cor:weighted-flat-interpolation-estimate,R:weighted-flat-interpolation-estimate-odd,lem:improved-first-derivative}.}
	\label{tab:macro-interpolation-error-orders}
\end{table}

For the lifted surface test, let
$f_h\colonequals\II_{h,\kg}^{\ku}f$. We measure
\begin{align*}
	E^{\mathrm{surf}}_0 &\colonequals
		\left|\int_{\G}f\,\dG-\int_{\Gh}f_h\,\dGh\right|, &
	E^{\mathrm{surf}}_1 &\colonequals
		\left\|\int_{\G}\gradG f\,\dG
		-\int_{\Gh}\gradGh f_h\,\dGh\right\|.
\end{align*}
Thus, the scalar absolute value is used for $\kd=0$, while for $\kd=1$ all
components of the surface-gradient error are integrated and the Euclidean,
equivalently Frobenius, norm is taken afterwards. These are the global
estimates in \cref{cor:global_interpolation_estimate} with constant weight
$g_1=g_2=1$. \cref{tab:interpolation-error-orders-low} reports the
function-value and first-derivative cases. The corresponding higher-order
tensor errors are collected in
\cref{sec:numerical-interpolation-estimates}.

\begin{table}[ht]
	\centering
	\setlength{\tabcolsep}{5pt}
	\begin{subtable}{0.48\textwidth}
		\centering
		\begin{tabular}{c|rrrrrr}
			\diagbox[width=3.8em]{$\ku$}{$\kg$} & $1$ & $2$ & $3$ & $4$ & $5$ & $6$\\
			\hline
			$1$ & $1.9$ & $2.0$ & $2.0$ & $2.0$ & $2.0$ & $2.0$\\
			$2$ & $2.0$ & $3.9$ & $4.0$ & $4.0$ & $4.0$ & $4.0$\\
			$3$ & $2.0$ & $4.0$ & $4.0$ & $4.0$ & $4.0$ & $3.9$\\
			$4$ & $2.0$ & $4.0$ & $4.0$ & $6.0$ & $6.0$ & $6.0$\\
			$5$ & $2.0$ & $4.0$ & $4.0$ & $6.0$ & $6.0$ & $5.6$\\
			$6$ & $2.0$ & $4.0$ & $4.0$ & $6.0$ & $6.0$ & $8.0$\\
		\end{tabular}
		\caption{$\kd=0$}
	\end{subtable}
	\hfill
	\begin{subtable}{0.48\textwidth}
		\centering
		\begin{tabular}{c|rrrrrr}
			\diagbox[width=3.8em]{$\ku$}{$\kg$} & $1$ & $2$ & $3$ & $4$ & $5$ & $6$\\
			\hline
			$1$ & $2.0$ & $2.0$ & $2.0$ & $2.0$ & $2.0$ & $2.0$\\
			$2$ & $2.0$ & $4.0$ & $4.0$ & $4.0$ & $4.0$ & $4.0$\\
			$3$ & $2.0$ & $4.0$ & $4.0$ & $4.0$ & $4.0$ & $4.0$\\
			$4$ & $2.0$ & $4.0$ & $4.0$ & $6.0$ & $6.0$ & $6.0$\\
			$5$ & $2.0$ & $4.0$ & $4.0$ & $6.0$ & $6.0$ & $6.0$\\
			$6$ & $2.0$ & $4.0$ & $4.0$ & $6.0$ & $6.0$ & $8.0$\\
		\end{tabular}
		\caption{$\kd=1$}
	\end{subtable}
	\caption{Experimental orders of convergence for lifted interpolation integrals $E^{\mathrm{surf}}_0$ with $\kd=0$ and $E^{\mathrm{surf}}_1$ with $\kd=1$ on the deformed sphere under red refinement. Rows give $\ku$ and columns give $\kg$. Compare \cref{cor:global_interpolation_estimate}.}
	\label{tab:interpolation-error-orders-low}
\end{table}

\subsection{Application to geometric quantities}

The following numerical tests compare geometric quantities on $\Gh$ with their smooth counterparts on $\G$.
We use one smooth scalar test function $f$ as before, two smooth vector fields $\bm{v}(x)=V^{(1)}(x)$ and $\bm{w}(x)_i=\cos(x_i) e^{x_i}$, and one smooth tensor field $\bm{A}(x)=V^{(2)}(x)$.

The quantities most directly related to the estimates proved in the previous sections are
\begin{align*}
	E_{\rho}
		&\colonequals \left|\int_{\Gh} (f\circ\mathfrak{L})\,\rho\,\dGh\right|,
	&
	E_{\Weingarten}
		&\colonequals \left|\int_{\Gh} (\Weingarten\circ\mathfrak{L}-\Weingartenh):(\bm{A}\circ\mathfrak{L})\,\dGh\right|,
	\\
	E_{\PPh\nn}
		&\colonequals \left|\int_{\Gh} (\bm{v}\circ\mathfrak{L})\cdot \PPh(\nn\circ\mathfrak{L})\,\dGh\right|,
	&
	E_{\GaussCurvature}
		&\colonequals \left|\int_{\Gh} (\GaussCurvature\circ\mathfrak{L}-\GaussCurvatureh)(\bm{v}\circ\mathfrak{L})\cdot(\bm{w}\circ\mathfrak{L})\,\dGh\right|,
	\\
	E_{\nn-\nnh}
		&\colonequals \left|\int_{\Gh} (\bm{v}\circ\mathfrak{L})\cdot(\nn\circ\mathfrak{L}-\nnh)\,\dGh\right|,
	&
	E_{\GaussCurvature,\PPh}
		&\colonequals \left|\int_{\Gh} (\GaussCurvature\circ\mathfrak{L}-\GaussCurvatureh)\PPh(\bm{v}\circ\mathfrak{L})\cdot\PPh(\bm{w}\circ\mathfrak{L})\,\dGh\right|,
\end{align*}
with lifting $\mathfrak{L}=L$ or $\mathfrak{L}=\pi$.
The first column contains the distance integral from
\cref{cor:dist_estimate_global} and the two normal-vector quantities from
\eqref{eq:Pnh} and \eqref{eq:n-nh}. The second column corresponds to the
Weingarten estimate \eqref{eq:dn-dnh}, the Gaussian curvature estimate
\eqref{eq:Gauss-curv-estimate}, and the projected curvature inner product used
in the analysis of a surface Stokes equation by \citet{HP2024Parametric}. All
quantities above are written as integrals on $\Gh$. For the Weingarten and
Gaussian curvature errors, this differs from the difference between the exact-
and discrete-surface integrals in the cited estimates by the integration error
for a smooth quantity. This additional term is of higher order and does not
change the rates being tested. The quantity
$E_{\GaussCurvature,\PPh}$ differs from $E_{\GaussCurvature}$ only by the
additional discrete tangential projections on the test fields.

We first consider surface parametrizations from refined macro-elements as used
in the analysis. Since the parametrization lifting $L$ is tied directly to this
construction and is used in the main interpolation arguments, we report the
complete set of integral quantities for $\mathfrak{L}=L$ in
\cref{tab:flowersurf-direct-integral-orders}. For $\kg=6$, the last refinement
level is already affected by roundoff in some integral quantities, and we
therefore use the last stable refinement step.

\begin{table}[ht]
	\centering
	\begin{tabular}{lrrrrrrll}
		\hline
		Quantity & $\kg=1$ & $\kg=2$ & $\kg=3$ & $\kg=4$ & $\kg=5$ & $\kg=6$ & $\kg$ odd & $\kg$ even \\
		\hline
		$E_\rho$ & $2.00$ & $4.03$ & $4.00$ & $7.27$ & $5.99$ & $9.66$ & $\kg+1$ & $\kg+2$ \\
		$E_{\PPh\nn}$ & $2.00$ & $3.99$ & $4.00$ & $6.00$ & $5.99$ & $7.98$ & $\kg+1$ & $\kg+2$ \\
		$E_{\nn-\nnh}$ & $2.04$ & $4.00$ & $4.00$ & $6.00$ & $5.99$ & $7.98$ & $\kg+1$ & $\kg+2$ \\
		$E_{\Weingarten}$ & $0.00$ & $2.13$ & $2.01$ & $3.95$ & $4.05$ & $5.88$ & $\kg-1$ & $\kg$ \\
		$E_{\GaussCurvature}$ & $0.00$ & $2.00$ & $2.00$ & $4.00$ & $4.00$ & $5.98$ & $\kg-1$ & $\kg$ \\
		$E_{\GaussCurvature,\PPh}$ & $0.01$ & $2.00$ & $2.00$ & $3.99$ & $3.99$ & $5.95$ & $\kg-1$ & $\kg$ \\
		\hline
	\end{tabular}
	\caption{Asymptotic experimental orders of convergence for the integral quantities on the deformed sphere surface using the macro-element parametrization $\Phih$ with red refinement and lifting $\mathfrak{L}=L$. The last two columns give the expected odd and even rates.}
	\label{tab:flowersurf-direct-integral-orders}
\end{table}

For the distance and normal-vector integrals the odd orders show the standard rate $\kg+1$, while the even orders show the improved rate $\kg+2$ predicted by the analysis. For the Weingarten and Gaussian curvature terms the same odd-even structure appears shifted by two derivatives: the observed rates are approximately $\kg-1$ for odd $\kg$ and $\kg$ for even $\kg$.

In addition to these integral quantities, we approximate the following $L^\infty$-norms by taking maxima over the evaluation points used in the numerical computations:
\begin{align*}
	N_{\rho}
		&\colonequals \| (f\circ\mathfrak{L})\rho \|_{L^\infty(\Gh)},&
	N_{\nn}
		&\colonequals \|\nn\circ\mathfrak{L}-\nnh\|_{L^\infty(\Gh)},\\
	N_{\Weingarten}
		&\colonequals \|\Weingarten\circ\mathfrak{L}-\Weingartenh\|_{L^\infty(\Gh)},&
	N_{\GaussCurvature}
		&\colonequals \|\GaussCurvature\circ\mathfrak{L}-\GaussCurvatureh\|_{L^\infty(\Gh)}.
\end{align*}
These norm errors are not expected to show the same cancellation as the integrals.
The approximate $L^\infty$-norms in \cref{tab:flowersurf-direct-norm-orders} show the standard pointwise behaviour. This separates the cancellation in the integral quantities from ordinary pointwise approximation of the geometry.

\begin{table}[ht]
	\centering
	\begin{tabular}{lrrrrrrl}
		\hline
		Quantity & $\kg=1$ & $\kg=2$ & $\kg=3$ & $\kg=4$ & $\kg=5$ & $\kg=6$ & expected \\
		\hline
		$N_\rho$ & $2.00$ & $3.00$ & $4.02$ & $5.00$ & $6.02$ & $6.93$ & $\kg+1$ \\
		$N_{\nn}$ & $1.00$ & $2.01$ & $2.99$ & $4.00$ & $5.00$ & $6.03$ & $\kg$ \\
		$N_{\Weingarten}$ & $0.00$ & $1.02$ & $2.02$ & $3.00$ & $4.02$ & $4.92$ & $\kg-1$ \\
		$N_{\GaussCurvature}$ & $0.00$ & $0.98$ & $2.00$ & $3.00$ & $4.00$ & $4.96$ & $\kg-1$ \\
		\hline
	\end{tabular}
	\caption{Asymptotic experimental orders of convergence for the approximate $L^\infty$-norms on the deformed sphere surface using the macro-element parametrization $\Phih$ with red refinement and lifting $\mathfrak{L}=L$. The last column gives the expected pointwise rate.}
	\label{tab:flowersurf-direct-norm-orders}
\end{table}

The estimates for the geometric quantities allow both liftings $L$ and $\pi$, see
\cref{cor:dist_estimate_global,lem:non-standard-estimates,lem:weingarten_est,cor:gauss-curvature-estimate}.
The two liftings also give virtually identical experimental orders. It is
therefore enough to retain one representative quantity from each of the two
integral regimes and one pointwise quantity for $\mathfrak{L}=\pi$.
\cref{tab:flowersurf-direct-orders-lifting-pi} verifies the same rates without
duplicating the complete tables for $\mathfrak{L}=L$.

\begin{table}[ht]
	\centering
	\begin{tabular}{lrrrrrrll}
		\hline
		Quantity & $\kg=1$ & $\kg=2$ & $\kg=3$ & $\kg=4$ & $\kg=5$ & $\kg=6$ & $\kg$ odd & $\kg$ even \\
		\hline
		$E_\rho$ & $2.00$ & $4.00$ & $4.00$ & $6.00$ & $6.00$ & $8.06$ & $\kg+1$ & $\kg+2$ \\
		$E_{\Weingarten}$ & $0.00$ & $2.11$ & $2.01$ & $3.95$ & $4.05$ & $5.88$ & $\kg-1$ & $\kg$ \\
		$N_{\nn}$ & $1.00$ & $2.01$ & $2.99$ & $4.00$ & $5.00$ & $6.03$ & $\kg$ & $\kg$ \\
		\hline
	\end{tabular}
	\caption{Representative asymptotic experimental orders on the deformed sphere using the macro-element parametrization $\Phih$, red refinement, and lifting $\mathfrak{L}=\pi$. The rows represent the distance integral, a curvature integral, and a pointwise geometric error.}
	\label{tab:flowersurf-direct-orders-lifting-pi}
\end{table}

\subsection{Scope and summary}

The experiments in this section test the flat interpolation estimates from \cref{thm:flat-interpolation-estimate,cor:weighted-flat-interpolation-estimate,lem:improved-first-derivative}, the lifted interpolation estimates from \cref{thm:interpolation_estimate_-1,thm:interpolation_estimate_0,thm:interpolation_estimate_1st_derivatives_standard,thm:interpolation_estimate_1st_derivatives,cor:global_interpolation_estimate}, the distance estimate from \cref{cor:dist_estimate_global}, the normal-vector estimates \eqref{eq:Pnh} and \eqref{eq:n-nh}, the Weingarten estimate \eqref{eq:dn-dnh}, and the Gaussian curvature estimate \eqref{eq:Gauss-curv-estimate}. Additional numerical tests outside the scope of the main theorems are collected in \cref{sec:additional-numerical-experiments}.

Within this scope, the experiments support the main message of the analysis. On the deformed sphere surface, the integral quantities show the predicted odd-even split. Odd geometry orders show the standard rates, while even geometry orders gain one order. The approximate $L^\infty$-norms do not show this improvement, which confirms that the effect is caused by cancellation in the integrals rather than by better pointwise approximation of the geometry.

\section*{Funding}
This work was supported by the Deutsche Forschungsgemeinschaft (DFG, German Research Foundation) through FOR 3013, project TP06, project number 417223351, to H.H. and S.P., and through project number 386450667 to G.Z.

\section*{Acknowledgements}
The authors used AI-assisted tools in a limited capacity to support code development and the writing process, including editorial revisions, reformulations, and language-level cleanups of proofs. All mathematical content, numerical results, and final text were reviewed and validated by the authors, who take full responsibility for the manuscript.

\section*{Data availability}
The code and data underlying this article are available from Zenodo
\citep{Praetorius2026OddGeometriesCode}.

\section*{Conflict of interest}
None declared.

\bibliographystyle{abbrvnat}
\bibliography{references}

\appendix
\appendixpage
\addappheadtotoc
\section{Higher-order derivatives}\label{sec:appendix-higher-order-derivatives}

This appendix contains the notation, estimates, and numerical tests for lifted
interpolation estimates involving derivatives of order at least two.  We use
the geometric setting, interpolation operators, liftings, and symmetry
assumptions introduced in \cref{sec:discrete surface,sec:interpolation_errors}.
The estimates are included for completeness, while the main text only uses the
function-value and first-derivative cases.

Higher-order surface derivatives are written as iterated covariant derivatives,
\[
  \nabla^{(\kd)} f \colonequals \nabla(\nabla^{(\kd-1)} f), \qquad \nabla^{(1)} f \colonequals \nabla f.
\]
Here $\nabla$ denotes the Levi-Civita covariant derivative, extended
componentwise to tensor-valued quantities. Let $\tauB_i\in T_x\G$, $x\in\G$,
and define the lifted tangential directions
$\tauM_i \colonequals \DM\Phi^+(\tauB_i\circ\Phi)\in T_{\xM}\M$ with
$x=\Phi(\xM)$. Then

\begin{align}
  \gradG^{(\kd)} f(x)[\tauB_1,\tauB_2,\ldots,\tauB_\kd] = \gradM^{(\kd)}\bar{f}(\xM)[\tauM_1,\tauM_2,\ldots,\tauM_\kd] + \sum_{s=1}^{\kd-1} \gradM^{(s)}\bar{f}\odot A_{(\kd,s)}(\Phi,\{\tauM_i\}_i).\label{eq:high-order-derivative}
\end{align}
The first term is the highest-order chain-rule contribution. The lower-order
terms contain derivatives of $\bar f$ multiplied by derivatives of $\Phi$ and
its inverse. These factors are collected in $A_{(\kd,s)}$. In particular,
$A_{(\kd,s)}(\Phi,\{\tauM_i\}_i)$ contains up to $\kd-s+1$ derivatives of $\Phi$,
and $\odot$ denotes the corresponding tensor contraction.

For $\kd=2$, the lower-order term in \eqref{eq:high-order-derivative} is the
usual Christoffel-symbol contribution:
\[
  A_{(2,1)}(\Phi,\eM_i,\eM_j)=(-\Gamma_{ij}^k)_{k=1}^2.
\]

\subsection{Interpolation errors}

We now state the higher-order analogue of the lifted interpolation estimates.
We use the notation of higher-order derivatives from
\eqref{eq:high-order-derivative} in tangential directions $\tauB$.

The parity improvement depends on the derivative order.  For integers $\order\ge1$ and $\kd\ge0$, set
\begin{equation}\label{eq:even-odd-order-\kd}
	\widehat{\order}_{\kd} \colonequals
	\begin{cases}
		\order+2, & \text{if } \order+2-\kd \text{ is even},\\
		\order+1, & \text{if } \order+2-\kd \text{ is odd}.
	\end{cases}
\end{equation}
This is exactly the parity condition in
\cref{cor:weighted-flat-interpolation-estimate}.

\begin{theorem}\label{thm:interpolation_estimate_higher_order}
  Let $\GM,\GhM$ be parametrized over $\M\in\TriBar$.
  Assume the local hypotheses of
  \cref{sec:surface-parametrization,sec:symmetric-macro-triangulation},
  including the boundary hypothesis from
  \cref{rem:symmetric-boundary-triangulation} where first-derivative boundary
  cancellation is used.
  Let $f\in C^{\ku+2}(\GM;\R)$ be a scalar function, let
  $f_h\colonequals \II_{h,\kg}^{\ku}(f)$ be the interpolation of $f$, and let
  $g\in W^{1,1}(\GM)$.
  Fix an integer $\kd$ with $2\le \kd\le \ku$ and let
  $\bm{v}_j\in W^{\kd+1,\infty}(U_\delta(\G);\R^3)$, $j=1,\ldots,\kd$, be ambient
  direction fields.
  Then,
  \begin{multline}
    \left| \int_{\GM} \gradG^{(\kd)}f[\bm{P}\bm{v}_{1},\ldots,\bm{P}\bm{v}_{\kd}]\,g\,\dG
      - \int_{\GhM} \gradGh^{(\kd)}f_h[\bm{P}_h\bm{v}_{1},\ldots,\bm{P}_h\bm{v}_{\kd}]\,(g\circ L)\,\dGh\right| \\
    \leq C h^{\min\{\widehat{\ku}_{\kd},\widehat{\kg}_{\kd}\}-\kd}
    \|f\|_{C^{\ku+2}(\GM)}\|g\|_{W^{1,1}(\GM)}
  \end{multline}
  with $\widehat{\ku}_{\kd}$ and $\widehat{\kg}_{\kd}$ as in \eqref{eq:even-odd-order-\kd}. The constant $C$ may depend on the
  $W^{\kd+1,\infty}$-norms of the direction fields, but is independent of $h$.
\end{theorem}
\begin{proof}
	We denote by $\tauB_j\colonequals\bm{P}\bm{v}_j$, and $\tauB_{h,j}\colonequals\bm{P}_h\bm{v}_j$ tangential directions on $\GM$ and $\GhM$, respectively, and define $\tauM_j\colonequals \DM\Phi^+(\tauB_j\circ\Phi)$ and $\tauM_{h,j}\colonequals (\DM\Phih)^+(\tauB_{h,j}\circ\Phih)$, $j=1,\ldots,\kd$, as the associated tangential vectors in $\M$.
	For simplicity we denote by $\tauB_{\cdots} \colonequals(\tauB_1,\ldots,\tauB_\kd)$ the sequence of tangential vectors.

	We estimate similar to the proof of \cref{thm:interpolation_estimate_1st_derivatives_standard},
	\begin{align} \label{eq:high-order-derivatives-on-M}
		\Big| \int_{\GM} \gradG^{(\kd)} f[\tauB_{\cdots}]\, g \,\dG - \int_{\GhM} \gradGh^{(\kd)}f_h[\tauB_{h\cdots}]\, (g\circ L)\,\dGh \Big|
			\leq
			\left| \int_{\M} \gradG^{(\kd)} (f-f_h\circ L^{-1})[\tauB_{\cdots}]\circ \Phi\, \bar{g}\mG \,\dM \right|\phantom{.}\\
			+ \left| \int_{\M} \left(\gradG^{(\kd)} (f_h \circ L^{-1})[\tauB_{\cdots}]\circ \Phi \, \mG - \gradGh^{(\kd)} f_h[\tauB_{h\cdots}]\circ \Phih\mGh\right)\bar{g}\,\dM \right|.\notag
	\end{align}
	We write for $\bar{f}\colonequals f\circ\Phi$ and $\bar{f}_h\colonequals\IBar_h^\ku(\bar{f})$
	\begin{align*}
		\gradG^{(\kd)}(f-f_h\circ L^{-1})[\tauB_{\cdots}] \circ \Phi & = \gradM^{(\kd)}(\bar{f} - \bar{f}_h)[\underbrace{\DM\Phi^+(\tauB\circ\Phi)_{\cdots}}_{\tauM_{\cdots}}] + \mathcal{R}(\bar{f} - \bar{f}_h, \Phi)
	\end{align*}
	with a remainder $\mathcal{R}$ depending on up to $\kd-1$ derivatives of $\bar{f} - \bar{f}_h$ and $\kd$ derivatives of $\Phi$.

	The assumption $\bm{v}_j\in W^{\kd+1,\infty}(U_\delta(\G),\R^3)$ ensures that every direction-dependent coefficient of $\bar g\colonequals g\circ \Phi\in W^{1,1}(\M)$ produced by \eqref{eq:high-order-derivative} yields a weight function $\omega(\tauM_{\cdots})$ that can be estimated in the $W^{1,1}(\M)$-norm by $\| g\|_{W^{1,1}(\GM)}$. Thus, by using \cref{cor:weighted-flat-interpolation-estimate} or standard estimates, we obtain
	\begin{align}
		\left| \int_{\M} \gradM^{(\kd)}(\bar{f} - \bar{f}_h)[\tauM_{\cdots}]\,\bar{g}\mG \,\dM \right|
		\leq C h^{\widehat{\ku}_{\kd}-\kd}\|f\|_{C^{\ku+2}(\GM)}\| g\|_{W^{1,1}(\GM)}.
	\end{align}
	The remainder has the form
	\[
		\mathcal{R}(\bar{f} - \bar{f}_h, \Phi) = \sum_{s=1}^{\kd-1}\gradM^{(s)}(\bar{f} - \bar{f}_h)[\tauM^{(s)}_{\cdots}]
	\]
	where the directions $\tauM^{(s)}_{\cdots}$ depend on $\tauM_{\cdots}$ and up to $\kd-s+1$ derivatives of $\Phi$ and $\Phi^{+}$. In particular a weight function $\omega^{(s)}(\tauM^{(s)}_{\cdots},\bar{g},\mG)$ composed of products of components of its arguments has regularity $\omega^{(s)}\bar{g}\mG\in W^{1,1}(\M)$. By \cref{cor:weighted-flat-interpolation-estimate}, \cref{lem:improved-first-derivative}, or standard estimates, we find that each addend of the remainder is of order $h^{\widehat{\ku}_{s}-s}$. Since $s\leq \kd-1$, this is at least of order $h^{\widehat{\ku}_{\kd}-\kd}$, and hence
	\[
	\left|\int_{\M}\mathcal{R}(\bar{f} - \bar{f}_h, \Phi)\bar{g}\mG \,\dM\right| \leq C h^{\widehat{\ku}_{\kd}-\kd}\|f\|_{C^{\ku+2}(\GM)}\| g\|_{W^{1,1}(\GM)}.
	\]

	For the second term in \eqref{eq:high-order-derivatives-on-M}, we express the derivatives on $\M$,
	\begin{align*}
		\gradG^{(\kd)} (f_h \circ L^{-1})[\tauB_{\cdots}]\circ \Phi &= \gradM^{(\kd)}\bar{f}_h[\underbrace{\DM\Phi^+(\bm{P}\bm{v}\circ\Phi)_{\cdots}}_{\tauM_{\cdots}}] + \mathcal{R}(\bar{f}_h,\Phi) \\
	\intertext{and}
		\gradGh^{(\kd)} f_h[\tauB_{h\cdots}]\circ \Phih
		&= \gradM^{(\kd)}\bar{f}_h[\underbrace{(\DM\Phih)^+(\bm{P}_h\bm{v}\circ\Phih)_{\cdots}}_{\tauM_{h\cdots}}]  + \mathcal{R}(\bar{f}_h,\Phih).
	\end{align*}
	For the difference we add and subtract zeros to obtain due to the multi-linearity of the derivative
	\begin{multline*}
		\int_{\M} \left(\gradG^{(\kd)} (f_h \circ L^{-1})[\tauB_{\cdots}]\circ \Phi \, \mG
		 - \gradGh^{(\kd)} f_h[\tauB_{h\cdots}]\circ \Phih\, \mGh \right) \bar{g}\,\dM\\
		\begin{aligned}
		  &= \int_{\M} \gradM^{(\kd)}\bar{f}_h[\tauM_{\cdots}]\,\bar{g}\mG \,\dM - \int_{\M} \gradM^{(\kd)}\bar{f}_h[\tauM_{h\cdots}] \, \bar{g}\mGh \,\dM
		   + \int_{\M} \mathcal{R}(\bar{f}_h,\Phi)\bar{g}\mG \,\dM - \int_{\M} \mathcal{R}(\bar{f}_h,\Phih) \bar{g}\mGh \,\dM\\
		  &= \int_{\M} \left(\gradM^{(\kd)}\bar{f}[\tauM_{\cdots}]+\mathcal{R}(\bar{f},\Phi)\right)\bar{g}(\mG -\mGh) \,\dM\\
		  &\phantom{=} +\int_{\M} \left(\left(\gradM^{(\kd)}\bar{f}_h-\gradM^{(\kd)}\bar{f}\right) [\tauM_{\cdots}]+\mathcal{R}(\bar{f}_h-\bar{f},\Phi)\right)\bar{g}(\mG -\mGh) \,\dM\\
		  &\phantom{=} + \int_{\M} \gradM^{(\kd)}\bar{f}_h[(\tauM-\tauM_h)_{\cdots}]\,\bar{g}\mGh\,\dM
		   + \int_{\M} (\mathcal{R}(\bar{f}_h,\Phi)-\mathcal{R}(\bar{f}_h,\Phih))\bar{g}\mGh \,\dM.
		\end{aligned}
	\end{multline*}
	Note that the differences always have the structure of first-order derivatives of interpolation errors, or products of interpolation errors.

	The first term can be estimated by \cref{cor:weighted-diff-of-area}. The second term produces by standard interpolation error estimates a term of order $h^{\ku+1-\kd+\kg}$. For the third term, the multilinearity of $\gradM^{(\kd)}\bar f_h$ produces terms with one or more direction differences. For each direction,
	\begin{align*}
		\tauM_j-\tauM_{h,j} &= \DM\Phi^+(\bm{P}\bm{v}_j\circ\Phi) - (\DM\Phih)^+(\bm{P}_h\bm{v}_j\circ\Phih) \\
		&= (\DM\Phi^+ - (\DM\Phih)^+)(\bm{P}\bm{v}_j\circ\Phi) + \DM\Phi^+(\bm{P}\bm{v}_j\circ\Phi - \bm{P}_h\bm{v}_j\circ\Phih)\\
		&\quad - (\DM\Phi^+ - (\DM\Phih)^+)(\bm{P}\bm{v}_j\circ\Phi - \bm{P}_h\bm{v}_j\circ\Phih).
	\end{align*}
	Terms containing two or more such differences are estimated directly with $L^\infty$ geometry bounds. Terms containing exactly one difference are reduced, by adding and subtracting the smooth counterparts of the remaining factors, to the weighted flat interpolation estimates for first derivatives of $\Phi-\Phih$, the pseudo-inverse estimate \cref{lem:pseudo-inverse-geometry-mapping}, and the projection estimate $\|\bm{P}-\bm{P}_h\|_{L^\infty}\leq Ch^\kg$.
	Finally, the difference $\mathcal{R}(\bar f_h,\Phi)-\mathcal{R}(\bar f_h,\Phih)$ contains the additional higher-order terms. These are derivatives of the geometric interpolation error up to order $\kd$, multiplied by smooth coefficients coming from the chain rule. They are again estimated componentwise by \cref{cor:weighted-flat-interpolation-estimate}, \cref{lem:improved-first-derivative}, or the standard estimate, depending on derivative order and parity. This yields the asserted order $h^{\min\{\widehat{\ku}_{\kd},\widehat{\kg}_{\kd}\}-\kd}$.
\end{proof}

\begin{corollary}\label{cor:global_interpolation_estimate_higher_order}
	Let $\Phi\colon\GBar\to\G$ be continuous and patchwise sufficiently smooth, and let $\Gh$ be parametrized elementwise over $\T\in\TriBar_h$. Let $f\in C^{\ku+2}\big(\G;\R\big)$ be a scalar function and let $f_h \colonequals \II_{h,\kg}^\ku f\in C^0\big(\Gh;\R\big)$ be the interpolation of $f$. Assume that the local hypotheses of \cref{thm:interpolation_estimate_higher_order} hold on every macro element. Fix an integer $\kd$ with $2\le \kd\le \ku$, let $g\in W^{1,1}(\G)$, and let $\bm{v}_j\in W^{\kd+1,\infty}(U_\delta(\G);\R^3)$, $j=1,\ldots,\kd$, be ambient direction fields. Then, with $\widehat{\ku}_{\kd}$ and $\widehat{\kg}_{\kd}$ as in \eqref{eq:even-odd-order-\kd},
  \begin{align}
		\Big|
			\int_{\G} \gradG^{(\kd)}f[\bm{P}\bm{v}_{1},\ldots,\bm{P}\bm{v}_{\kd}]\,g\,\dG
			- \int_{\Gh} \gradGh^{(\kd)}f_h[\bm{P}_h\bm{v}_{1},&\ldots,\bm{P}_h\bm{v}_{\kd}]\,(g\circ L)\,\dGh
		\Big| \notag\\
		&\leq Ch^{\min\{\widehat{\ku}_{\kd},\widehat{\kg}_{\kd}\}-\kd}
		\|f\|_{C^{\ku+2}(\G)}\|g\|_{W^{1,1}(\G)}.
		\label{eq:global-interpolation-higher-derivatives}
	\end{align}
	The constant $C$ may depend on the macro parametrizations and on the $W^{\kd+1,\infty}$-norms of the direction fields, but it is independent of $h$.
\end{corollary}
\begin{proof}
	This follows by summing \cref{thm:interpolation_estimate_higher_order} over the finitely many macro elements, exactly as in the proof of \cref{cor:global_interpolation_estimate}.
\end{proof}

\subsection{Numerical experiments}\label{sec:numerical-interpolation-estimates}

The flat macro-element interpolation experiments and the lifted surface tests
for $\kd=0$ and $\kd=1$ are reported in \cref{sec:numerical-experiments}. Here we
collect the higher-order lifted surface tests for $\kd=2$ and $\kd=3$, which
correspond to the estimates above. As in the lower-order tests, we compare the
tensor-valued integrals of $\D^{(\kd)} f$ over $\G$ and $\D_h^{(\kd)} f_h$ over $\Gh$, and
take the Frobenius norm after integration. This is a finite sum of scalar
directional estimates and therefore has the same expected order.

\begin{table}[ht]
	\centering
	\setlength{\tabcolsep}{5pt}
	\begin{subtable}{0.48\textwidth}
		\centering
		\begin{tabular}{c|rrrrrr}
			\diagbox[width=3.8em]{$\ku$}{$\kg$} & $1$ & $2$ & $3$ & $4$ & $5$ & $6$\\
			\hline
			$1$ & $0.0$ & $0.0$ & $0.0$ & $0.0$ & $0.0$ & $0.0$\\
			$2$ & $-0.0$ & $2.0$ & $2.0$ & $2.0$ & $2.0$ & $2.0$\\
			$3$ & $0.0$ & $2.0$ & $2.0$ & $2.0$ & $2.0$ & $2.0$\\
			$4$ & $0.0$ & $2.0$ & $2.0$ & $4.0$ & $4.0$ & $4.0$\\
			$5$ & $0.0$ & $2.0$ & $1.9$ & $4.0$ & $4.0$ & $5.4$\\
			$6$ & $0.0$ & $2.0$ & $1.9$ & $4.0$ & $4.0$ & $6.0$\\
		\end{tabular}
		\caption{$\kd=2$}
	\end{subtable}
	\hfill
	\begin{subtable}{0.48\textwidth}
		\centering
		\begin{tabular}{c|rrrrrr}
			\diagbox[width=3.8em]{$\ku$}{$\kg$} & $1$ & $2$ & $3$ & $4$ & $5$ & $6$\\
			\hline
			$1$ & $0.0$ & $0.0$ & $0.0$ & $0.0$ & $0.0$ & $0.0$\\
			$2$ & $0.0$ & $0.0$ & $0.0$ & $0.0$ & $0.0$ & $0.0$\\
			$3$ & $0.0$ & $0.0$ & $2.0$ & $2.0$ & $2.0$ & $2.0$\\
			$4$ & $0.0$ & $0.0$ & $2.0$ & $2.0$ & $2.0$ & $2.0$\\
			$5$ & $0.0$ & $0.0$ & $2.0$ & $2.0$ & $4.0$ & $4.1$\\
			$6$ & $0.0$ & $0.0$ & $2.0$ & $2.0$ & $4.0$ & $4.1$\\
		\end{tabular}
		\caption{$\kd=3$}
	\end{subtable}
	\caption{Experimental orders of convergence for lifted interpolation integrals with $\kd=2$ and $\kd=3$ on the deformed sphere under red refinement. For each derivative order $\kd$, rows give $\ku$ and columns give $\kg$. Compare \cref{cor:global_interpolation_estimate_higher_order}.}
	\label{tab:interpolation-error-orders-high}
\end{table}

\section{Additional numerical experiments}\label{sec:additional-numerical-experiments}

This appendix contains numerical results for two settings that are not covered
directly by the main theorems. The first setting is projected refinement, a
standard construction in surface finite element computations that differs from
the direct macro-element parametrization used in the analysis. The second
setting is newest-vertex bisection, where the refinement pattern is not covered
by the symmetric red-refinement argument.

\subsection{Projected refinement}

Starting from a piecewise linear interpolating surface, newly generated vertices during refinement are projected to the smooth surface. The higher order geometry is then described by an interpolation of the closest-point projection $\pi$ over the elements of this piecewise linear surface \citep[cf.][]{Hei2004Isoparametric,Demlow2009Higher}. This construction differs from the direct macro-element parametrization used in the analysis above. In particular, the composition of a piecewise affine parametrization with the closest-point projection is only piecewise smooth and is therefore not directly covered by the assumptions of \cref{sec:interpolation_errors}.

Projected refinements are standard in higher-order surface finite element methods, and the corresponding geometry approximation estimates are classical \citep[see][]{Hei2004Isoparametric,Demlow2009Higher}. For the quantities considered here, however, the improved integral orders come from the symmetry cancellation studied in the preceding sections. The numerical results in \cref{tab:flowersurf-projected-integral-orders} indicate that this cancellation is still present when the projected refinement is generated from a structured refinement of a coarse reference grid.

\begin{table}[ht]
	\centering
	\begin{tabular}{lrrrrrr}
		\hline
		Quantity & $\kg=1$ & $\kg=2$ & $\kg=3$ & $\kg=4$ & $\kg=5$ & $\kg=6$ \\
		\hline
		$E_{\PPh\nn}$ & $2.00$ & $4.00$ & $4.01$ & $6.01$ & $6.00$ & $8.04$ \\
		$E_{\nn-\nnh}$ & $1.97$ & $4.00$ & $4.01$ & $6.01$ & $6.00$ & $8.04$ \\
		$E_{\Weingarten}$ & $0.00$ & $2.00$ & $2.01$ & $4.01$ & $4.01$ & $6.06$ \\
		$E_{\GaussCurvature}$ & $0.00$ & $2.00$ & $2.00$ & $4.00$ & $4.00$ & $6.03$ \\
		$E_{\GaussCurvature,\PPh}$ & $0.01$ & $2.00$ & $2.00$ & $4.00$ & $4.00$ & $6.02$ \\
		\hline
	\end{tabular}
	\caption{Experimental orders of convergence for projected refinement on the deformed sphere surface with red refinement. The $\kg=6$ values are taken from the last stable refinement level.}
	\label{tab:flowersurf-projected-integral-orders}
\end{table}

The approximate $L^\infty$-norms on the deformed sphere surface in \cref{tab:flowersurf-projected-norm-orders} follow the standard pointwise rates. Thus, also for projected refinements, the improved rates appear in the integral quantities and not in the pointwise geometric errors.

\begin{table}[ht]
	\centering
	\begin{tabular}{lrrrrrr}
		\hline
		Quantity & $\kg=1$ & $\kg=2$ & $\kg=3$ & $\kg=4$ & $\kg=5$ & $\kg=6$ \\
		\hline
		$N_\rho$ & $2.00$ & $2.97$ & $3.97$ & $4.93$ & $5.98$ & $6.80$ \\
		$N_{\nn}$ & $0.98$ & $2.00$ & $3.01$ & $4.00$ & $5.00$ & $6.02$ \\
		$N_{\Weingarten}$ & $0.00$ & $1.01$ & $2.01$ & $3.01$ & $4.00$ & $5.03$ \\
		$N_{\GaussCurvature}$ & $0.00$ & $1.01$ & $1.98$ & $3.01$ & $4.01$ & $5.05$ \\
		\hline
	\end{tabular}
	\caption{Experimental orders of convergence for approximate $L^\infty$-norms on projected refinements of the deformed sphere surface.}
	\label{tab:flowersurf-projected-norm-orders}
\end{table}

The projected construction is not identical to the interpolation of a globally smooth macro-element parametrization over the flat reference domain. The arguments of the preceding sections would apply if there were smooth parametrizations over the reference patches $\M$ whose Lagrange interpolation preserves the projected Lagrange nodes and the associated discrete function spaces. Such parametrizations need not coincide with the piecewise affine parametrization followed by closest-point projection. Their existence is not shown here, but the computations are consistent with this interpretation for projected refinements generated from structured refinements of a coarse reference grid.

\subsection{Newest-vertex bisection}

The proof of the improved estimates uses the patchwise symmetry of the refined
triangulation. Red refinement has this structure in a transparent way: inside
each macro element the refined triangles can be grouped into symmetric pairs,
while unmatched triangles occur only at macro-element boundaries.

Newest-vertex bisection, as generated by \textsc{Dune-Alugrid}, is visually more subtle. The refinement closure near vertices where several macro elements meet creates regions with a different refinement level, see \cref{fig:surface_nvb_refinement}. These regions still appear to contain local symmetric pairs of triangles, but the boundary between paired and unpaired regions is no longer confined to the macro-element boundary in the same simple way as for red refinement.

\begin{figure}[ht]
    \begin{subfigure}{0.2\linewidth}
        \includegraphics[width=.9\textwidth]{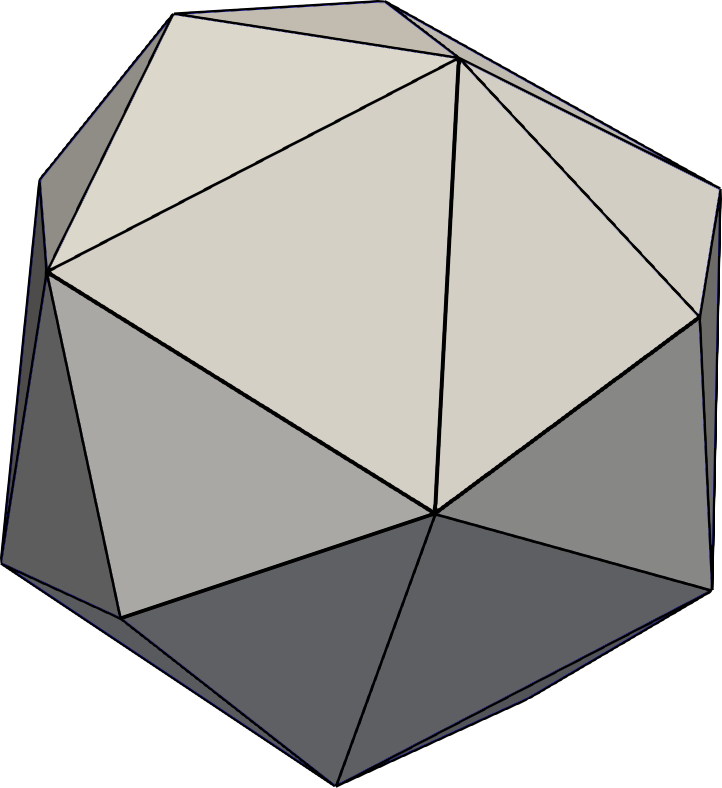}%
    \end{subfigure}\hfill%
    \begin{subfigure}{0.2\linewidth}
        \includegraphics[width=.9\textwidth]{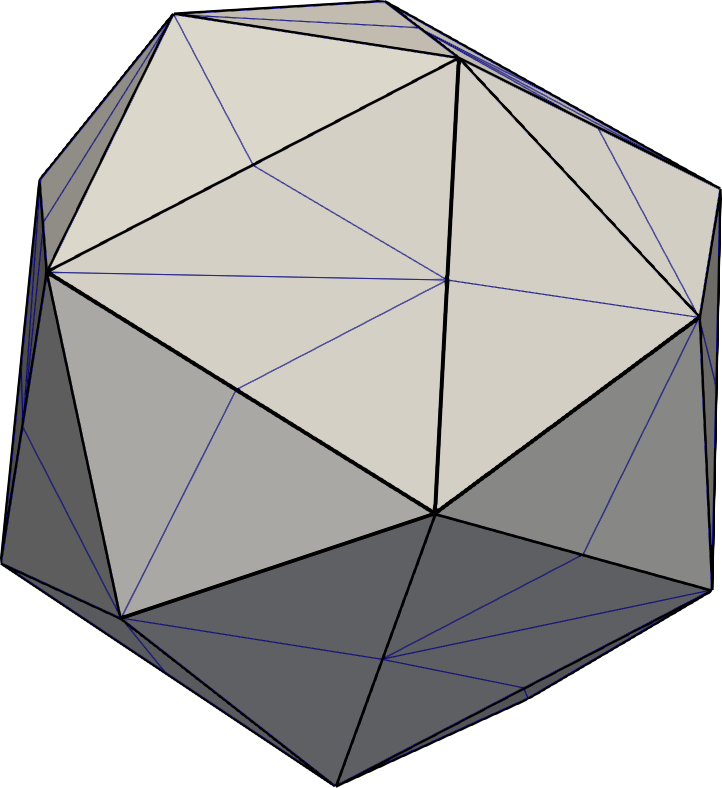}%
    \end{subfigure}\hfill%
    \begin{subfigure}{0.2\linewidth}
        \includegraphics[width=.9\textwidth]{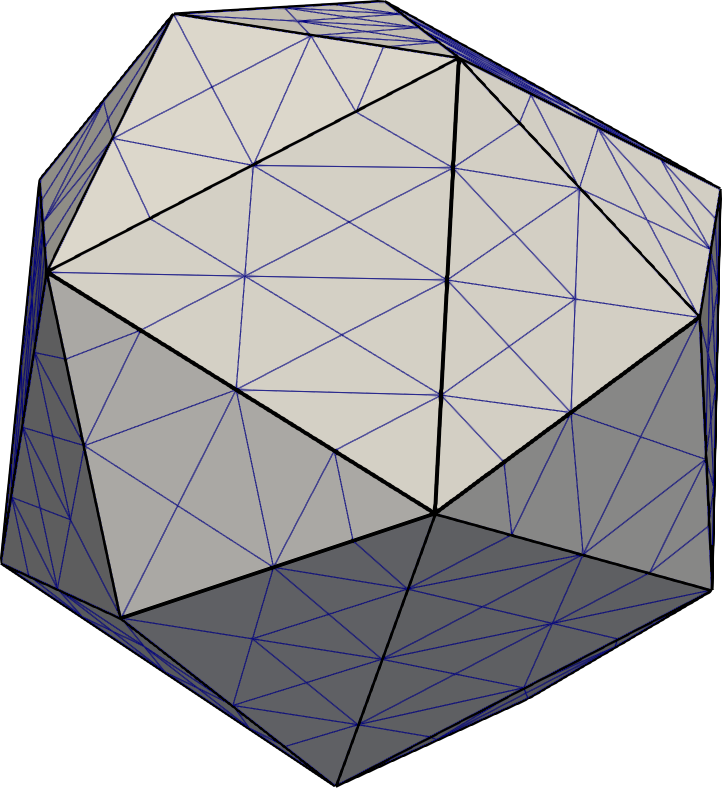}%
    \end{subfigure}\hfill%
    \begin{subfigure}{0.2\linewidth}
        \includegraphics[width=.9\textwidth]{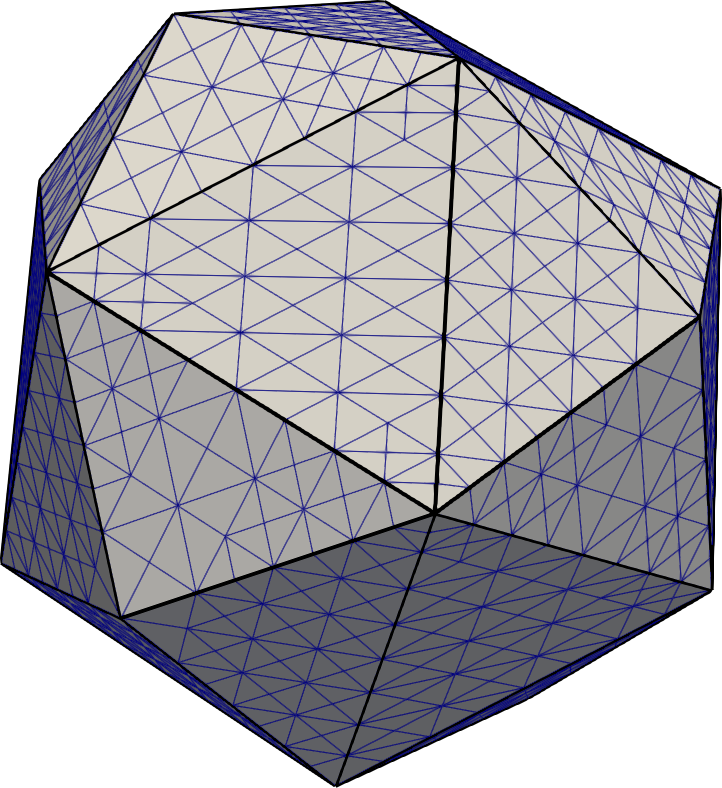}%
    \end{subfigure}\hfill%
    \begin{subfigure}{0.2\linewidth}
        \includegraphics[width=.9\textwidth]{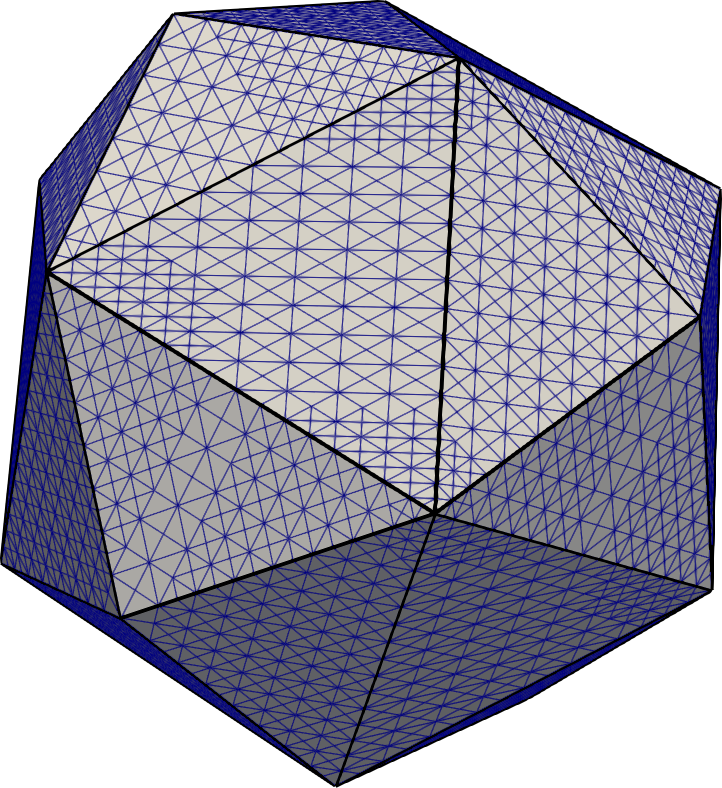}%
    \end{subfigure}
    \caption{Newest-vertex bisection refinement of a triangular reference surface $\GBar$. The refinement closure creates local regions with different refinement levels near macro-element vertices, so the symmetric-pair structure is less visible than for red refinement.}\label{fig:surface_nvb_refinement}
    \figurealt{From left to right, the coarse triangulated sphere and four progressively refined meshes generated by newest-vertex bisection. Fine triangles become increasingly dense, while additional closure refinements remain visible near vertices and extend into the interiors of coarse triangles.}
\end{figure}

The numerical evidence nevertheless shows the same odd-even behaviour for both refinement strategies. \cref{tab:flowersurf-refinement-pattern-orders} compares representative integral quantities on the deformed sphere surface for red refinement and newest-vertex bisection. For $\kg=6$, the values are taken from the last stable refinement step, since the finest level is affected by roundoff in some integral quantities.

\begin{table}[ht]
	\centering
	\begin{tabular}{llrrrrrr}
		\hline
		Refinement & Quantity & $\kg=1$ & $\kg=2$ & $\kg=3$ & $\kg=4$ & $\kg=5$ & $\kg=6$ \\
		\hline
		red & $E_{\PPh\nn}$ & $2.00$ & $3.99$ & $4.00$ & $6.00$ & $5.99$ & $7.98$ \\
		NVB & $E_{\PPh\nn}$ & $2.01$ & $3.93$ & $4.05$ & $6.06$ & $6.05$ & $8.08$ \\
		red & $E_{\Weingarten}$ & $0.00$ & $2.11$ & $2.01$ & $3.95$ & $4.05$ & $5.88$ \\
		NVB & $E_{\Weingarten}$ & $0.00$ & $2.02$ & $2.06$ & $3.89$ & $4.08$ & $4.36$ \\
		red & $E_{\GaussCurvature}$ & $0.00$ & $2.00$ & $2.00$ & $4.00$ & $4.00$ & $5.99$ \\
		NVB & $E_{\GaussCurvature}$ & $0.00$ & $2.02$ & $2.02$ & $4.04$ & $4.04$ & $6.06$ \\
		\hline
	\end{tabular}
	\caption{Experimental orders of convergence for representative integral quantities on the deformed sphere surface. The table compares red refinement (red) with newest-vertex bisection (NVB) on the same surface.}
	\label{tab:flowersurf-refinement-pattern-orders}
\end{table}

These data suggest that newest-vertex bisection preserves enough cancellation structure to obtain the same improved integral estimates. A precise characterization of the necessary symmetric-pair structure for newest-vertex bisection, and more generally for bisection-type refinement strategies with closure, is not part of the present analysis and remains open.

\end{document}